\documentclass[11pt]{article}
\usepackage{times}
\usepackage[english]{babel}
\usepackage[dvips]{graphicx}
\usepackage{amssymb}
\usepackage{amsfonts}
\usepackage{amsmath}
\usepackage{mathrsfs}
\usepackage{url}
\usepackage{hyperref}
\usepackage{color}
\usepackage{pstricks}
\usepackage{setspace}
\usepackage{algorithm}
\usepackage{algpseudocode}

\newcommand{\qed}{\rule{7pt}{7pt}}

\newenvironment{proof}{\noindent{\bf Proof:}\hspace*{1em}}{\proofend}
\newenvironment{proof-sketch}{\noindent{\bf Sketch of Proof}\hspace*{1em}}{\qed\bigskip}
\newenvironment{proof-idea}{\noindent{\bf Proof Idea}\hspace*{1em}}{\qed\bigskip}
\newenvironment{proof-of-lemma}[1]{\noindent{\bf Proof of Lemma #1}\hspace*{1em}}{\qed\bigskip}
\newenvironment{proof-attempt}{\noindent{\bf Proof Attempt}\hspace*{1em}}{\qed\bigskip}

\def\pr{\mathbb{P}}
\def\E{\mathbb{E}}
\def\Exp{\mathbb{E}}

\def\ind{\mathbb{I}}

\def\a{\alpha}
\def\al{\alpha}
\def\b{\beta}
\def\bt{\beta}

\def\e{\epsilon}

\def\g{\gamma}

\def\m{\mu}
\def\n{\nu}

\def\s{\sigma}

\def\t{\tau}

\def\E{{\bf Exp}}

\def\hf{{1\over2}}

\newtheorem{lemma}{Lemma}[section]
\newtheorem{theorem}[lemma]{Theorem}

\newtheorem{remark}[lemma]{Remark}

\newcommand{\proofend}{\hspace*{\fill}\mbox{$\Box$}}

\begin{document}

\title{Containing Viral Spread on Sparse Random Graphs: Bounds, Algorithms, and Experiments}

\author{
Milan Bradonji\'c\thanks{Research partially supported by NIST grant 60NANB10D128. Part of this work was done at Los Alamos National Laboratory.}\\
Mathematics of Networks and Communications Group\\Bell Labs, Alcatel-Lucent\\
600 Mountain Avenue, Murray Hill, New Jersey 07974, USA\\E-mail: \textbf{{\tt milan@research.bell-labs.com}}
\and
Michael Molloy\\
Department of Computer Science\\ University of Toronto\\
10 King's College Road, Toronto, ON, Canada, M5S 3G4\\
E-mail: {\tt molloy@cs.toronto.edu}
\and 
Guanhua Yan\\
Information Sciences Group\\ Los Alamos National Laboratory\\
Los Alamos, New Mexico, 87545, USA\\ 
E-mail: {\tt ghyan@lanl.gov}
}

\date{\today}
\maketitle

\begin{abstract}
Viral spread on large graphs has many real-life applications such as malware propagation in computer networks and rumor (or misinformation) spread in Twitter-like online social networks. Although viral spread on large graphs has been intensively analyzed on classical models such as Susceptible-Infectious-Recovered, there still exits a deficit of effective methods in practice to contain epidemic spread once it passes a critical threshold. Against this backdrop, we explore methods of containing viral spread in large networks with the focus on sparse random networks. The viral containment strategy is to partition a large network into small components and then to ensure the sanity of all messages delivered across different components. With such a defense mechanism in place, an epidemic spread starting from any node is limited to only those nodes belonging to the same component as the initial infection node. We establish both lower and upper bounds on the costs of inspecting inter-component messages. We further propose heuristic-based approaches to partition large input graphs into small components. Finally, we study the performance of our proposed algorithms under different network topologies and different edge weight models.
\end{abstract}

\section{Introduction}
\label{sec:introduction}

Albeit computer worms came to existence more than two decades ago,
they are still severely threatening the Internet security
nowadays. Modern computer malware have commonly applied social
engineering tricks for their propagation, due to the fact that social
trusts among computer users have made them less vigilant against
potential malware threats. For instance, the ILOVEYOU worm managed to
infect tens of millions of Windows computers worldwide through e-mail
attachments in 2000~\cite{bib:iloveyou}. Moreover, the emerging
popularity of online social network sites such as Facebook~\cite{facebook} and Twitter~\cite{twitter}
has provided a new playground for computer malware, as evidenced by a
few recently spotted worms (for example, Koobface~\cite{bib:koobface} and
W32/KutWormer~\cite{bib:orkut-wiki}) that specifically targeted these
networks. In contrast to traditional Internet worms such as Code
Red~\cite{bib:code-red-modeling} and Slammer~\cite{bib:inside-slammer}
that use port scanning to discover vulnerable machines, computer worms
based on social engineering pose an even more severe cyber threat to
many enterprise networks as they can easily penetrate through
enterprise firewalls or Intrusion Detection Systems (IDSes).

Similarly to malware propagation in computer networks, rumor or
misinformation spread in social networks has a destructive nature. The
growing popularity of social networking sites such as Twitter and Facebook has made them
become one of the major news sources for many people. 
Unfortunately, misinformation can also be spread on these social media networks, which, in some cases, cause
undesirable consequences such as public panic. One such example is the spread of
rumors regarding swine flu on Twitter in 2009~\cite{bib:swine-flu-twitter}.

The challenge of containing viral spread on large graphs is common to both malware propagation in computer networks and misinformation
spread in social networks. A folklore fact from epidemiology is that when an
epidemic spread passes a critical threshold or a take-off point, it may break out to become a pandemic~\cite{bib:size-of-outbreaks}. One well-studied strategy in containing
epidemic spread is \emph{immunization}~\cite{bib:efficient-immunization}, which
ensures that a set of nodes are immune to infection.
Applying immunization schemes to fight against malware propagation in computer
networks or misinformation spread in social networks, however, has its
limitations. Due to the distributed nature of computer networks,
computers are often administered by different domains or organizations, making
it a daunting task to immunize a specific computer. Also, in the
context of social networks, it may be difficult to convince a specific user to stop
spreading misinformation.

Realizing the challenges of immunization in containing viral spread on large
graphs, we instead focus on a different strategy, which is to partition a large
graph into a number of small \emph{islands} and then deploy message sanitization
techniques to ensure the sanity of all messages that are delivered across
islands. Hence, when a node is infected and starts spreading viral information
(either malicious messages for spreading computer malware or misinformation in
social networks), epidemic spreading can only take place within the nodes in the
same island as the initial infection point. From a practical standpoint, the
strategy considered is only applicable to scenarios where viral
messages across islands can be inspected and stopped. For our
problems of interest, social relationships exploited by computer
malware or misinformation spread are usually maintained at centralized servers or social networking sites,
such as Gmail, Twitter and Facebook, where communication messages can be inspected for removing viral information.

To be effective in containing viral spread on a large social graph, the strategy
under consideration  must ensure that none of the islands after graph partition
is too large. Note that the problem is different from the balanced graph partitioning problem~\cite{bib:balanced-partition} which aims at balancing the sizes of different
components. Although an ideal approach is to
sanitize every message in the network so that every component contains
exactly one node, the solution would be computationally prohibitive in
reality due to the enormous communication messages to inspect in a
large network like Twitter\footnote{There are $50,000,000$ messages
transmitted within the Twitter network every
day~\cite{bib:twitter-blog}.}. Hence, in a practical setting, it is
crucial to strike a balance between efficiency and effectiveness.

Motivated by such a tradeoff, we study how to contain viral spread on large social graphs under limited operational
resources. More specifically, we focus on the following problem. Consider an
undirected graph $G(V,E)$, where the set of nodes $V$ denotes the set of users, and the set of edges $E$ contains all friendships among the users in the network. The weight $w(u,v)$ of an edge $(u,v)$ is the normalized number of
messages sent between users $u$ and $v$ in the past. 
The goal is to find a subset of edges
$E' \subseteq E$ that minimizes the overall cost $B = \sum_{e \in E'}
w(e)$, given the constraint that the size of the largest connected component
after removing all edges in $E'$ from $G$ must not be greater than a certain
threshold. Social network graphs are typically sparse~\cite{bib:connection-subgraphs}, so in this
work we only consider sparse random graphs where the numbers of nodes and edges
are of the same asymptotic order. Extending this work beyond sparse random graph
models remains as our future work.

In a nutshell, our main contributions are given as
follows. First, we consider minimizing the overall cost of disintegrating sparse Bernoulli (Erd\H os-R\'enyi) random
graphs under constrained edge deletion such that the size of each connected
component is no greater than a certain threshold. Under this random graph model,
we provide both: (i) the threshold on the number of edges to be deleted such
that every connected component in the remaining graph has a size of at most a
given constant, and (ii) the threshold on the number of edges to be
deleted below which the remaining graph always has a connected component of size linear in
the number of nodes. 

We extend our results to a more generic type of random graph models, that is, random graphs with a given degree sequence. In particular we provide lower and upper bounds on the cost of disintegrating sparse random graphs with a given degree distribution under different edge weight distributions, including uniform edge weights, bounded edge weights, and unbounded edge weights with finite mean. These thresholds on the costs are related to the expected value and concentration of the the maximum spanning forest of an input graph.

Finally, based on insights gained from our theoretical analysis, we propose heuristic-based algorithms to disintegrate sparse random graphs into small connected components. Our method first computes the maximum spanning forest of
the original graph, and then uses one of two different heuristics to disintegrate the forest. We further generate synthetic graph topologies using sparse random graph models and study experimentally the performance of our proposed algorithm under different edge weight models, including uniform, exponential, and power-law distributions. The experimental results confirm that our theoretical analysis guides us towards better heuristic-based approaches in containing viral spread on sparse random graphs.

Our results establish theoretical bounds on the performance of containing
viral spread on large sparse random graphs based on graph partitioning and thus,
shed light on its limitation when it is deployed on real-world networks. Although put in the
context of containing malware or misinformation spread in large social networks,
the conclusions drawn from this work have independent interests in other
application domains as well, as they essentially deal with a
fundamentally theoretical problem on how to disintegrate networks
under the edge deletion constraints.

The rest of the paper is organized as follows. Section~\ref{sec:related.work} presents related work. 
In Section~\ref{sec:bounds.budget.er}, we establish bounds on the cost for
disintegrating sparse Erd\H os-R\'enyi random graphs. Section~\ref{sec:bounds.budget.rggds} provides bounds on the cost for
disintegrating sparse random graphs with a given degree sequence, where
edge weights are either constant, or bounded, or unbounded i.i.d. random variables. 
In Section~\ref{sec:experiments}, we provide a heuristic-based algorithm
to partition sparse random graphs, and study its performance on synthetically
generated graph topologies. %Section~\ref{sec:conclusions} concludes this work.

\section{Related work}
\label{sec:related.work}

One motivating application behind this work is to contain propagation of malware
based on online social contacts. Zou~{\it et al.}~developed a model that characterizes
propagation of computer worms on email networks that typically follow a
heavy-tailed distribution~\cite{DBLP:conf/icccn/ZouTG04}. They found that
computer worms spread fast in scale-free networks, but by selectively immunizing those highly connected nodes,
it is possible to slow down malware spread significantly in such networks. As we
shall discuss later, targeted immunization for containing spread of social-based
malware has its limitations. Xu~{\it el al.}~proposed a correlation scheme to monitor a small set
of nodes for detecting malware spread in online social networks, but
their algorithm works only for malware detection rather than containment. Similarly, techniques proposed in~\cite{bib:catching-im-worms} focus on detection, instead of containment, of malware spread in IM (Instant Messaging) social networks.
Using a dataset collected from a real-life online
social network, Yan~{\it et al.}~analyzed its social graph
and user activity patterns and found that both play a critical role in
malware spreading in online social
networks~\cite{bib:online-social-worm}; they further tried a community structure detection algorithm to partition the
social graph into small connection components, and found that a significant fraction of edges have to be removed in order to disintegrate the graph effectively. 
The study in~\cite{bib:online-social-worm} was done empirically on a specific network topology, and
thus, does not have a strong theoretic foundation. 

Another motivating application of this work is containment of rumor or
misinformation diffusion in large social networks. Research on this topic is
still at its infancy. Budak~{\it et al.}~considered the problem of limiting the spread of misinformation in social networks~\cite{bib:limiting-misinformation}. Their approach was to convince a small set of users in the online social
network to spread ``good'' rumors that cancel out the influence of bad ones. Convincing people to spread ``good'' rumors in social networks, albeit an interesting idea, may not be feasible in practice. The strategy we consider in this work, however, does
not require involvement of individual users. 

The percolation theory has established the critical threshold for
wide-scale epidemic spreading and has been widely applied to study
epidemic spreading in diverse network structures, such as small-world
networks~\cite{bib:percolation-small-world}, heterogeneous
networks~\cite{bib:percolation-heterogeneous}, and sensor
networks~\cite{bib:sensor-malware}. In the context of scale-free 
networks, selectively immunizing those highly connected nodes is an
effective approach to slowing down epidemic spread in such
networks~\cite{DBLP:conf/icccn/ZouTG04,bib:immunization-complex}. 
In both problems of containing social-based malware and misinformation spread in social
networks, however, the key challenge that faces node immunization is the difficulty
of interacting with individual users due to the distributed nature of social
networks. An alternative approach would be to achieve node immunization
by sanitizing all messages that come to or come from those nodes to be immunized.
In our problem, however, this may not be the most cost-effective approach because highly
connected nodes could generate a large number of communication messages. 

\section{Bounds on the cost of disintegrating Erd\H os-R\'enyi random graphs}
\label{sec:bounds.budget.er}

To tackle the problem of disintegrating sparse random graphs effectively, we
first consider the simple Erd\H os-R\'enyi random graphs
$G_{n,M}$~\cite{erdos-1959-random}. We are interested in establishing
theoretical bounds on the cost $B$ that is necessary to disintegrate a sparse
Erd\H os-R\'enyi random graph. In this work, we will use the abbreviation {\it a.a.s.} for ``asymptotically almost surely'' to denote  with probability tending to one as the number of nodes tends to infinity.

We fix $c>1/2$ and consider a random graph $G$ from $G_{n,M=cn}$.  Let
$L=L(c)$ be the unique positive solution to $L=1-e^{-2cL}$, and let $R=R(c)=cL\left(1+e^{-2cL}\right)$. 
The giant component of $G$ has $Ln +o(n)$ vertices~\cite{erdos-1960-evolution} and $Rn+o(n)$ edges. 
Then, we can establish the following theorem.
\begin{theorem}\label{t1} For any constants $\g,c>0$ there exist constants $\a,t>0$ such that
a.a.s. $G_{n,M=cn}$ is such that:
\item[(a)] If we remove fewer than $(R-L-\g)n$ edges, then there must exist a component of size at least $\a n$.
\item[(b)]   We can remove
$(R-L+\g)n$ edges so that every component has size at most $t$.
\end{theorem}

That is, the critical threshold occurs at $(R-L)n$ edges.
In Section~\ref{sec:bounds.budget.rggds}, we will describe how to extend this result to other random graph models, including random graphs with a fixed power-law degree sequence.

In order to prove Theorem~\ref{t1}, we first bound the number of edges needed to be removed from a tree with bounded degree, such that each of resulting components has size $\leq t$, for some given positive $t$. 

\begin{lemma}
\label{thm:trim.tree}
Let $H$ be a tree with  maximum degree at most $d$. For any $\bt \in (0,1)$, we can remove at most $\bt |H|$ edges from $H$, so that each resulting tree has size less than equal $t  = t(\beta,d)= \max \left\{4d, r\right\}$, 
where $r$ is the largest root of $x - (3.5/\bt) \log x =0$.
\end{lemma}

\begin{proof}
Denote the number of nodes in the tree $n = |H|$. If $n \leq t$, then no edge needs to be removed. Hence, consider the case $n \geq t+1$. We prove the assertion by induction on the size of a given tree $H$. 

Suppose that the assertion is true for any tree of size $ \leq k$, where $k \geq t$. Pick a node $v$, which is not a leaf of the tree, and denote its degree $d_v = \deg(v)$. 
Since $v$ is not a leaf $d_v \geq 2$ and moreover $d_v \leq d$ by the conditions of the lemma. 
Denote $K_1, K_2, \dots, K_{d_v}$ the corresponding tree components incident to $v$. From 
$|K_1| + \cdots + |K_{d_v}| = n-1 \geq t$,
by the pigeonhole principle it follows that at least one of those components, denoted $H'$, has size $|H'| \geq t /d_v \geq t/d$. Let $e$ be the edge connecting $v$ and $H'$. Consider the two newly obtained trees $H'$ and $H'' = V \setminus (H' \cup \{e\} )$ of the sizes $a = |H'|$ and $b = |H''|$. Then $a \geq t/d$ and $b \geq 2$ since $v$ is not a leaf. 
$H',H''$ both have maximum degree at most $d$, and so we can apply the inductive hypothesis to those components.

We now prove that the number of edges necessary to remove in order to achieve the assertion of the lemma is at most $\phi_t(k)$ defined as follows. For a given integer $t>0$, define $\phi_t$ on $\mathbb{N}_0 = \{0,1,2,\dots\}$, such that: 
\begin{equation*} 
\phi_t(k) = 
\begin{cases} 
\bt k - \al \log k, & \text{if } k > t,\\
0, & \text{if } k \leq t, 
\end{cases} 
\end{equation*}
where $\bt$ is given by statement of the lemma and $\al$ is to be chosen later. Notice that $\phi_t(k) \geq 0$ for any $k$, given $t$ is at least the largest root (among the two) of the equation $\bt x - \al \log x =0$. This simple fact follows by examining the function $\bt x - \al \log x$ on $\mathbb{R}^{+}$.

By the inductive hypothesis, the number of edges needed to be removed, including $e$, is at most:
\begin{eqnarray}
%\label{eq.ind.step}
\nonumber &\leq& 1 + \phi_t(a) + \phi_t(b) = 1 + (\bt a - \al \log a) + (\bt b - \al \log b) \\
\nonumber &=& 1 + \bt (a+b) - \al \log ab = 1 + \bt n - \al \log ab.
\end{eqnarray}
It suffices to show 
\begin{equation}
\nonumber
1 + \bt n - \al \log ab \leq \phi_t(n) = \bt n - \al \log(a+b), 
\end{equation}
which is equivalent to 
\begin{equation}
\label{eq.abbt}
\frac{1}{a} + \frac{1}{b} \leq e^{-1/\al}.
\end{equation}
Recall $a \geq t/d$ and $b \geq 2$. Choosing $\al=3.5$ and the conditions of the lemma yield~(\ref{eq.abbt})
\begin{equation}
\nonumber
\frac{1}{a} + \frac{1}{b} \leq \frac{d}{t} + \frac{1}{2} \leq \frac{d}{4d} + \frac{1}{2} = \frac{3}{4} < e^{-1/3.5} \approx 0.751477 \,,
\end{equation}
which concludes the proof.
\end{proof}

Now we can prove Theorem~\ref{t1}.

\begin{proof}
{\em Part (a):}
Let $G'$ be formed by removing edges from $G=G_{n,M}$. 
A standard lemma (see Lemma~\ref{claim:standard.fact} in Appendix~\ref{app}) implies that there exists $\a=\a(\g,c)>0$ such
that every subgraph $H\subset G$ with fewer than $\a n$ vertices has fewer than $(1+\hf\g)|H|$ edges.
Therefore, if every component of $G'$ has size less than $\a n$ then the total number of edges remaining
from the giant component of $G$ must be less than the sum of $(1+\hf\g)|H|$, taken  over every component $H$ of $G'$ such that $H$ is a subgraph of the giant component of $G$.  This sums to $(Ln+o(n))(1+\hf\g)<Ln+\hf\g n$, since $L<1$. Therefore, we must have removed more than $(R-L-\g)n$ edges. This is contradiction, so there must be  a component of size at least $\a n$.

{\em Part (b):}
Choose any spanning tree of the giant component. Remove all edges not in that tree.~A.a.s.~every other component that is not a tree has exactly one cycle. For each such component, remove an edge from that cycle.
We expect to remove $O(1)$ such edges~\cite{bollobas-2001-book}. We now have a forest, and a.a.s. we have removed a total of 
$Rn-Ln+o(n)$ edges, since the spanning tree of the giant component has $Ln+o(n)$ vertices.

Standard results on the tail of the degree sequence of $G_{n,p}$ (see e.g.~\cite{bollobas-2001-book}) show that the
total number of edges incident with vertices of degree at least $d$ is at most $\e n$, where $\e\rightarrow 0$
as $d\rightarrow\infty$. So we can
choose $d$ large enough in terms of $\g,L,c$ so that the total number of edges touching  vertices of degree greater than $d$ will be less than $\frac{1}{2}\g  n$.
Remove all such edges. What remains is a forest of maximum degree at most $d$.
By Lemma~\ref{thm:trim.tree} with $\b=\g/4$, we can remove $\frac{\g}{2} |H|$ edges from each component $H$ of that forest, so that each resulting tree has size at most $t=t(d,\g)=O(1)$. The total number of edges removed
in this step is at most $\frac{\g}{2} \sum|H|<\frac{\g}{2}n$. So in all, we
have removed  at most $(R-L+\g)n$ edges.
\end{proof}

\subsection{Bounded edge weights}
Now we extend Theorem \ref{t1} to the setting where the edges have weights.  We begin with the case
where the weights are bounded. Let $W$ be the maximum edge-weight.  In this case, we allow an adversary
to first examine the random graph, and then add edge-weights.

Let $L'n$ be the total weight of a maximum weight spanning tree of the giant component.  Let $R'n$ be the total
weight of the edges in the giant component.

\begin{theorem} 
\label{t2} For any constants $\g,c>0$ there exist constants $\a,t>0$ such that
a.a.s.\ $G_{n,M=cn}$ is such that if an adversary places positive weights of up to $W$ on each edge then:
\item[(a)] If we remove edges of total weight less than $(R'-L'-\g)n$, then there must be a component of size at least $\a n$.
\item[(b)]  We can remove
edges of total weight at most $(R'-L'+\g)n$ so that every component has size at most $t$.
\end{theorem}

\begin{proof}
{\em Part (a):} Remove any set of edges.
From Lemma~\ref{claim:standard.fact} it follows
that for any $\g,W>0$ there exists $\a>0$ such
that every subgraph $H\subset G$ with less, than $\a n$ vertices has at most $(1+\frac{\g}{4W})|H|$ edges.
Therefore, if every component has size less, than $\a n$, then each such component $H$ consists of a MaxST (maximum spanning tree)
of $H$ plus at most $\frac{\g}{4W}|H|+1$ additional edges (since the MaxST has $|H|-1$ edges). 

If $|H|>12W/\g$ then $\frac{\g}{4W}|H|+1<\frac{\g}{3W}|H|$.  
The expected number of cycles of length at most $12W/\g$ in $G_{n,M=cn}$ is $O(1)$, for constant $c$.  Therefore, a.a.s. there are fewer than $\sqrt{n}$ components $H$
with $|H|\leq12W/\g$ containing any additional edges besides the MaxST.

Therefore, the total weight of the remaining edges that
are in the giant component of $G$ must be at most $W\sqrt{n}$ plus the sum over every component $H_i$ that is a subgraph of that giant component of the weight of a MaxST of $H_i$ plus $W \frac{\g}{3W}|H_i|$.
The total weight of these
MaxST's is at most $L'n$, the weight of a MaxST of the giant component. To see this, note that we can form a spanning tree of the giant component by adding edges to join together MaxST's of all these $H$'s. So the total weight of all remaining giant component edges is at most $W\sqrt{n}+L'n+\frac{\g}{3} \sum|H_i|\leq L' n+\frac{\g}{2} n$.

{\em Part (b):} We proceed as in the proof of Theorem
\ref{t1} part (b). This time, we choose a maximal spanning tree of the giant component. The total weight of the giant component edges not in that tree is $(R'-L')n$. We choose $d$ so that the total number of edges touching vertices of degree greater than $d$
is at most $\frac{1}{4W}\g n$.  We remove at most $\frac{\g}{4W} |H|$ edges from each component $H$ of that forest, so that each resulting tree has size at most $t=t(d,\g,W)=O(1)$.  Since each edge has weight
at most $W$, the total weight of the edges removed is less than $(R'-L'+\g/2)n$.
\end{proof}

\subsection{Unbounded edge weights}
\label{sec:unbounded.edge.weights}

We now consider random graphs with  unbounded edge weights. This time, we do not permit an adversary to weight the graph; instead, the weights of the edges are chosen at random. 
More specifically, the weights are i.i.d. random variables, chosen from any probability distribution $f(w)$ that has a sufficiently small tail $\pr \left(W > n\right) = o(n^{-2})$.

Let $\mu=\int_0^\infty wf(w)dw$ denote the expected weight and $F(x)=\int_{0}^w f(w)dw$ denote the cumulative distribution function of an edge. Note that $\pr \left(W > n\right) = o(n^{-2})$ implies finite mean $\mu<\infty$.  

Again, we let $L'n$ be the total weight of a maximum weight spanning tree of the giant component.  Let $R'n$ be the total weight of the edges in the giant component.

\begin{theorem}
\label{t3} 
Consider any probability distribution $f(w)$ such that $\pr \left(W > n\right) = o(n^{-2})$. 
For any constants $\g,c>0$ there exist constants $\a,t>0$ such that
a.a.s.\ $G_{n,M=cn}$ is such that:
\item[(a)] If we remove edges of total weight less than $(R-L-\g) n$, then there must be a component of size at least $\a n$.
\item[(b)]  We can remove
edges of total weight at most $(R-L+\g)n$ so that every component has size at most $t$.
\end{theorem}

\begin{proof}
We first argue that there exists $W$ sufficiently large such that a.a.s. the total weight of all edges of weight greater than $W$ is at most $\g/(4n)$.
Let us consider a random variable $\varphi(x)=\sum_{i=1}^M w_i \ind \{w_i \geq x\}$ which depends on $x \geq 0$. Moreover $\Exp(\varphi(x)) = M \int_{x}^{\infty} w d F(w)$. 
By~Lebesgue's dominated convergence theorem (see~\cite{billingsley-1979-probability}), given finite mean $\int_{0}^{\infty} w d F(w) < \infty$, it follows that for any constants $\g,c>0$ there exists sufficiently large $W$ such that $M^{-1}\varphi(W)=\int_{W}^{\infty} w d F(w) \leq \g/(5c)$. (This statement would not necessarily hold for infinite mean $\int_{0}^{\infty} w d F(w) = \infty$.) Therefore, $\Exp(\varphi(W)) \leq M \g/(5c) = \g n/5$.
A random variable $w_i \ind\{w_i \geq W\}$ is absolutely integrable since $\Exp(w_i) < \infty$. Hence, by the strong law of large numbers  $n^{-1} \varphi(W)$ tends to $n^{-1}\Exp(\varphi(W)) < \gamma/5$ a.a.s.

{\em Part (a):} We proceed as in the proof of Theorem~\ref{t2}(a).  As in that proof,  the total weight of
the remaining giant component edges of weight at most $W$ is at most $(L+\g/3)n$.  The total weight of
the remaining edges of weight greater than $W$ is at most  $\frac{\g}{4}n$. So the total  weight of the remaining giant component edges is less than $(L+\g )n$ and so we must have removed at least $(R-L-\g)n$ weight.

{\em Part (b):} 
We proceed as in the proof of Theorem~\ref{t2}(b). As in that proof, the removed edges of weight at most $W$ total less than $(R-L+\g/2)n$. The removed edges of weight greater than $W$ total at most $\frac{\g}{4}n$.
\end{proof}

We close this section by noting that, for the model considered here, we can determine $R',L'$.

Determining $R'$ is straightforward. A.a.s. the number of edges in the giant component
is $R=R(c)$, as given at the beginning of this section.
Since the edges are weighted independently, standard concentration arguments, as in the proof above, yield that a.a.s. the total weight of those edges is $R' n+o(n)$ where 
$R'=R\mu$.

The techniques of~\cite{frieze-1985-value} yield the expected size of the $MaxSF$ of the entire graph. First determine the weights of the edges, without yet exposing their endpoints. Then sort these weights such that $w_1 \geq \cdots \geq w_n$. For $x>\frac{1}{2}$, define $g=g(x)$ to be the unique positive solution to $g=1-e^{-2xg}$. Thus, the giant component of $G_{n,M=xn}$ a.a.s. has size $g(x)n+o(n)$~\cite{erdos-1960-evolution}.
Analyzing Kruskal's Algorithm as in~\cite{frieze-1985-value} yields that the probability that the edge with weight $w_i$ belongs to the $MaxSF$ is $\pi(i/n)+o(1)$ where
\begin{equation} 
\label{eq:pi}
\pi(x) = \left\{ 
\begin{array}{lr} 
1,  & x \leq 1/2,\\ 
1 - g(x)^2, & x > 1/2.
\end{array} \right. 
\end{equation} 
\begin{remark}
This is simply the probability, upon selecting the endpoints of the edges in order of their weights, that the $i$th edge will have endpoints that are in different components
of the graph formed by the first $i-1$ edges.
\end{remark}

Thus, the expected total weight of the $MaxSF$ of the entire $G_{n,M}$ is given by
\begin{equation} 
\label{eq:maxsf}
\Exp(MaxSF(G_{n,M})) = \sum_{i=1}^M  \Exp(w_i)\pi(i/n)  + o(n).
\end{equation}
To obtain the weight of the $MaxST$ of just the giant component, we need to subtract the total weight of the rest of the forest.  The graph outside of the giant component is a.a.s. a forest plus $O(1)$ edges. The total weight of those $O(1)$ edges will be $o(n)$, and so we can simply subtract the total weight of all edges outside of the giant component.  There are $M-Rn$ such edges, and each has expected weight $\mu$.  This yields:
\begin{equation}
\label{eq:l_prim}
\Exp(L') = \sum_{i=1}^M  \Exp(w_i) \pi(i/n)   - (c-R)\mu n .
\end{equation}
Straightforward concentration arguments (using for example Azuma's Inequality plus an a.a.s. bound on the maximum edge weight) yield that $L'$ is concentrated around its mean.  We omit the details.

We now concentrate on (\ref{eq:l_prim}). For the probability density function $f(w)$ and the cumulative density function $F(w)$, 
the expected value $\Exp(w_i)$ of the order statistics $w_i$ is given by 
\begin{equation}
\label{eq:exp,wi}
\Exp(w_i) = M{M-1 \choose i-1} \int_0^{\infty} w f(w) F(w)^{M-i}(1-F(w))^{i-1} d w \,.
\end{equation}
In general (\ref{eq:exp,wi}) can be evaluated numerically. The following two cases when weights are drawn from the uniform or exponential distribution demonstrate a possibility to analytically express $\Exp(w_i)$.

\noindent
\textbf{Example 1: Uniform weight distribution.}
The edge weights follow $f(x)=1$ on $[0,1]$. Then $\mu = 1/2$ and $w_i$ follows the Beta distribution $B(M+1-i,i)$. Hence, 
$\Exp(w_i) = 1 - i/(M+1)$.

\noindent
\textbf{Example 2: Exponential weight distribution.}
The edge weights follow $f(x)=\lambda e^{-\lambda x}$ for $x \geq 0$. Then $\mu = 1/\lambda$ and $\Exp(w_i) = \lambda^{-1} M {M-1 \choose i-1} \sum_{k=0}^{M-1} {M-1 \choose k} (-1)^k/(k+i)^2$.

\subsection{Other random graph models}
\label{sec:bounds.budget.rggds}

One of the largest and most comprehensive recent studies on the structure of networks has been done in~\cite{leskovec-2009-community}. The results have been compared across more than one hundred large social and information networks, as well as small social networks, expanders, and networks with mesh-like or manifold-like geometry. That paper examined: the number of nodes, the number of edges, the fraction of nodes in the largest biconnected component, the fraction of edges in the largest biconnected component, the average degree, the empirical second-order average degree, average clustering coefficient, the estimated diameter, and the estimated average path length. These examined networks contained roughly from $5,000$ to $14,000,000$ nodes, and from $6,000$ to $100,000,000$ edges. A key conclusion relevant to our work is that all of the studied networks showed to be very sparse, with average degree $\approx 2.5$ for the networks of blog-posts, $400$ for the network of movie ratings from Netflix~\cite{netflix}, $10$ (with median $6$) for the social networks as well as most of the other studied networks, see~\cite{leskovec-2009-community}. 
Those arguments lead us to study sparse random graphs with given degree distribution as a model for large networks, which is the main subject of our work. 
However, we are aware of imperfections of this model. As an example, a study on online social networks in~\cite{kumar-2006-structure} shows that more than half nodes in online social networks do not belong to the largest component and mainly form stars. 

Given the previous, we now list some properties that will suffice for our theorems to extend to other models of large real-world networks. 

\begin{enumerate}
\item For any constants $\e,\g>0$, we can choose $d,\a$ such that 
\begin{enumerate}
\item[(a)] a.a.s. the sum of the degrees of all vertices of degree at least $d$ is less than $\e n$; 
\item [(b)] a.a.s. every subgraph $S$ of size at most $\a n$ contains at most $(1+\g)n$ {\em small edges}, where a small edge is one whose endpoints both have degree less than $d$ in the original graph.
\end{enumerate}
\item For any constant $c_1$, the expected number of cycles of length at most $c_1$ is $o(n)$.
\end{enumerate}

The proofs of Theorems~\ref{t1},~\ref{t2},~\ref{t3} extend easily to any random graph with those  properties. The only non-trivial modification is that, in part (a), we must account for the
at most $\e n$ remaining edges which are not small; by taking $\e<\hf\g$ (or $\e<\frac{\g}{10W}$), this number is
negligible.

For example, random graphs with a power-law degree sequence (when the fraction of nodes having degree $k$ is $\Theta(k^{-\gamma})$ for some $\gamma>2$) are easily seen to
satisfy these properties.  Properties 1(a) and 2 are trivial; see Appendix~\ref{app} for a proof that if our degree sequence satisfies Property 1(a), then Property 1(b) also holds.

\section{Algorithms and Experiments} 
\label{sec:experiments}

In Section~\ref{sec:unbounded.edge.weights} we have analyzed the expectation and concentration of the total
weight of the maximum spanning forest. Based on the
insights gained from our theoretical analysis, we propose a heuristic-based
approach to disintegrate a sparse random graph, and then study its performance
using synthetically generated graph topologies and edge weights.

\subsection{Heuristic-Based Algorithm} 
\label{sec:algorithm}

We now propose a heuristic-based algorithm that disintegrates a sparse random graph $G$, and demonstrate its experimental results. The main goal of the proposed Algorithm~\ref{algo:heur} is to break a forest of the input graph by using one of the two heuristics \emph{Susceptibility} or \emph{Edge Betweenness Centrality}. 
The steps of Algorithm~\ref{algo:heur} are as follows. 

For an input graph $G$, find the maximum spanning forest $F = MaxSF (G)$. Assign $B:=\sum_{e \in G \setminus F} w(e)$ to be the initial value of the overall cost of disintegrating $G$.
Let $F_1 : = F$, and $k$ be the number of edges in $F$. At every time step $\tau=1,\dots, k$, find an edge $e^*_{\tau}$ that maximizes $\Upsilon(F_\tau,e)$ of the current graph $F_\tau$, defined by~(\ref{eq:heur.susceptibilty}) or~(\ref{eq:heur.bc}). If multiple edges have the highest score, pick one among them uniformly at random. 
Add the weight of $w(e^*_{\tau})$ to the overall cost $B$, remove $e^*$ from the current graph $F_{\tau+1} : = F_{\tau} - e^*_{\tau}$ and repeat the process.

\begin{algorithm}[h!]
\caption{Heuristic-Based Algorithm}
\label{algo:heur}
\begin{algorithmic}[1]
\State{Find the maximum spanning forest $F := MaxSF(G)$ of an input graph $G$.}
\State{Assign $B:=\sum_{e \in G \setminus F} w(e)$ to be the initial value of the overall cost of disintegrating $G$.}
\State Let $F_1 : = F$, and $k$ be the number of edges in $F$.
\For{$\tau \gets 1, k$}
\State{Pick uniformly at random $e_{\tau}^* \in \arg\max_{e \in F_\tau} \Upsilon(F_\tau,e)$ (apply~(\ref{eq:heur.susceptibilty}) or (\ref{eq:heur.bc}), respectively.)} 
\State{$B:=B+w(e^*_{\tau})$.}
\State{$F_{\tau+1} : = F_{\tau} - e^*_{\tau}$.}
\EndFor
\end{algorithmic}
\end{algorithm} 
We now define the score function $\Upsilon(H, e)$ for the two heuristics: Susceptibility and Edge Betweenness Centrality.

\emph{Susceptibility Heuristic.} 
For every $e \in H$, calculate susceptibility~\cite{grimmett-1999-percolation} of the graph $H - e$. That is, let connected components $C_1, C_2, \dots$ partition $H - e$. Then define the score to be: 
\begin{equation}
\label{eq:heur.susceptibilty}
\Upsilon_S(H,e) =\frac{1}{|V(H)|}\sum_{i} |C_i|^2 \,. 
\end{equation}
The concept of susceptibility is imported from theoretical physics, see~\cite{grimmett-1999-percolation}. 
Moreover, notice that~(\ref{eq:heur.susceptibilty}) equals the expected component size of the given graph. 
We stress that susceptibility in this work should not be misinterpreted with susceptibility in epidemic models~\cite{kermack-1933-contributions}.
For example, susceptibility for a random graph with a given degree sequence is given in~\cite{janson-2010-susceptibility}.
Informatively, at every step the susceptibility heuristic chooses an edge that breaks the current graph into a larger number of small components. The norm $L_2$ has been used for purposes of our experiments, but any other norm $L_p$ with $p \geq 1$ could be used.

\emph{Edge Betweenness Centrality Heuristic.} 
Betweenness centrality of an edge $e$ represents the sum of the fraction of all pairs shortest paths that traverse $e$, 
\begin{equation}
\label{eq:heur.bc}
\Upsilon_{BC}(F_\tau, e) = \sum_{u, v \in F_\tau} \frac{\sigma{(u,v | e)}}{\sigma{(u,v)}} \,, 
\end{equation}
where $\sigma(u,v)$ is the number of shortest paths between $u$ and $v$, and $\sigma(u,v | e)$
is the number of shortest paths between $u$ and $v$ that traverse $e$.

\subsection{Experiments}
\label{sec:experiments}

To study the performance of the heuristic-based algorithm, we run it on synthetically generated sparse random graphs under different edge weight models. In our experiments, we consider the two following random graph models:
\begin{itemize} 
\item [(a)]
Random graphs $G_{50,100}$ with $n=50$ nodes and $M=100$ edges,
\item [(b)] 
Random graphs on $n=50$ nodes with a given power law degree distribution, $\textrm{Powerlaw}(3, 1/4)$, that is, $f(w)=w^{-3}/8$ for $w \in [1/4,\infty)$. 

\end{itemize}

For the edge weights, we consider the following three different  distributions in our experiments: 
\begin{itemize}
\item [(i)]
$\textrm{Unif}[0,1]$, that is, $f(w)=1$ for $w \in [0,1]$,
\item [(ii)]
$\textrm{Exp}(2)$, that is, $f(w)=2 e^{-2w}$ for $w \in [0,\infty)$,
\item [(iii)]
$\textrm{Powerlaw}(3, 1/4)$, that is, $f(w)=w^{-3}/8$ for $w \in [1/4,\infty)$. 
\end{itemize}
The parameters of each edge weight distribution are set in such a way that the average edge weights are the same for different distributions.

For each combination of graph topology and edge weight distribution, we randomly generate 100 instances. For each instance, we use four different heuristic-based methods to compute the relationship between the size of the maximum component and the total cost of removed edges:
\begin{itemize}
\item [1)] Algorithm described in Section~\ref{sec:algorithm} with the greedy heuristics ({\bf MaxSF-Susceptibility}),
\item [2)] Algorithm described in Section~\ref{sec:algorithm} with the betweenness-centrality heuristics ({\bf MaxSF-Betweenness}),
\item [3)] In a different algorithm from the one  described in Section~\ref{sec:algorithm}, we use the greedy heuristics as described in Section~\ref{sec:algorithm} directly on the entire input graph rather than on its maximum spanning forest ({\bf Full-Susceptibility}),
\item [4)] Similarly, we use the betweenness-centrality heuristics as described in Section~\ref{sec:algorithm} directly on the entire input graph rather than on its maximum spanning forest ({\bf Full-Betweenness}).
\end{itemize}

By comparing performance of the algorithms on maximum spanning forests against those directly on the entire topologies, we are interested in whether our theoretical analysis indeed guides us to find better heuristic-based approaches.
Figures~\ref{fig:gnm-unif}-\ref{fig:pw-seq-powerlaw} in Appendix~\ref{app} provide evaluation results under different scenarios. In each of these figures, the left column shows the size of the maximum component against the total cost of removed edges, and the right column shows its box-plot view. The boxplot diagrams are standard box diagrams in statistics: representing median, and the $25$th and $75$th percentiles (the lower and upper quartiles, respectively). 
We heuristically chose width of the boxes to be $5$. 

From the results we make the following observations. 
\begin{itemize}
\item For the two used heuristics, the susceptibility heuristic performs consistently better than the betweenness centrality heuristic in different combinations of network graphs and edge weight distributions. We, however, note that the algorithm based on the susceptibility heuristic has a higher computational overhead, given its greedy nature. For every edge removed, the algorithm has to recompute the score of each edge in the remaining graph. For a large graph with billions of edges, this may not be a practical solution. 
\item Comparing the results from MaxSF-Susceptibility and MaxSF-Betweenness against those from Full-Susceptibility and Full-Betweenness, we first observe that there are initial costs associated with the former to remove edges outside of the maximum spanning forests. Also, we find that the curves from the former drop off much more sharply. This suggests that for sparse random graphs, the strategy of removing edges outside the maximum spanning forest and then disintegrating the forest obtained really pay off, as long as the budget allows us to do that. The results confirm that our theoretical analysis indeed leads us to a better heuristic-based approach in disintegrating sparse random graphs.
\item Comparing results from different edge weight distributions, we find that power-law distributions tend to have some samples with long tails. These phenomena occur when some expensive edges have to be removed when disintegrating the graphs. This is in accordance with the fact that power law distributions are highly skewed, and thus, more likely to produce heavy-weight edges than the other two. In some cases, when a heavy-weight edge bridges two connected components, this edge has a high betweenness centrality measure and removing it can reduce the maximum component size significantly.
\end{itemize}

As evidenced from our experimental results, finding the maximum spanning forest plays an instrumental role in disintegrating sparse random graphs. Since our theoretical analysis concerns only sparse random graphs, the proposed algorithm may not perform well for other types of graphs. Pursuing efficient algorithms for those graphs remains as our future work.

\section{Conclusions} 
\label{sec:conclusions}

This work is motivated by the challenges that arise in containing malware spread based on social engineering tricks and rumor spread in large social networks. We consider the containment strategy that partitions the original graph into a number of small islands and sanitizes every message delivered across islands. We establish theoretical bounds on the cost necessary to disintegrate a sparse random graph such that none of the connected components has a size greater than a given threshold. We also derive the expectation on the total weight of the maximum spanning forest of the original graph, including its lower and upper bounds, as well as its concentration. Based on the insights gained from our theoretical analysis, we propose a heuristic-based approach to disintegrating a sparse random graph, and study its performance with synthetically generated topologies. Our results from this work not only shed light on how to contain viral spread in real-life networks effectively, but also have independent interests in other application domains as well as it essentially deals with a fundamental theoretic problem.

We hope that the theory developed in this study would shed light on the capability -- or the limitations -- of a scheme that relies on graph partitioning to contain propagation of computer malware and misinformation in real-world social networks. Admittedly, the containment strategy considered in this study is not a panacea, and it works only effectively in the following scenarios. First, if we want to take a preventive approach against viral spread of computer malware or misinformation in large-scale online social networks, we need to ensure that the system have sufficient computational resources to prevent these information from crossing among different islands. Although the proposed strategy relieves us from monitoring all the messages in the network, there can still be a large number of messages that traverse among different islands in a large online social network. Hence, filtering techniques that can efficiently detect inter-island messages with patterns of interests (for example, those that embed URLs pointing to suspicious domain names, or those carrying sensitive words indicative of rumors) would be complementary to the containment strategy considered in this study. In cases where the accuracy of a message is difficult to verify in real time (for example, a claim that there is an epidemic disease in Alaska), the system can attach a caution flag along with the message before it is delivered to a different island. Such a caution flag at least raises awareness among users before they blindly spread the message further, and can thus, slow down the propagation process of computer malware or misinformation and win time for a more effective method (for example, a counter campaign) to be further deployed. Second, in a dynamic environment where there are only sporadic adoptions of misinformation or computer malware, the containment strategy considered in this study can limit the infections to only the nodes that reside on the same islands as those initially contaminated ones. If, however, the viral spread has already entered a number of islands before the containment scheme starts sanitizing suspicious inter-island messages, it may become too late for the containment strategy to be effective. For such cases, other complementary containment strategies such as launching a counter campaign can be put in place to cancel out the effects of viral spread of computer malware or misinformation. 

\bibliographystyle{acm}
\bibliography{viral}

\appendix
\section{Appendix}
\label{app}

The following lemma is standard and has appeared in many other places. It first appeared in~\cite{luczak-1991-sharp-concetration}.

\begin{lemma}
\label{claim:standard.fact}
Consider $G_{n,M=cn}$. For any $\gamma>0$ there exists $\alpha=\a(\g,c)>0$ such that a.a.s. every subgraph $H \subseteq G_{n,M}$ with less than $\alpha n$ vertices
has at most $(1 + \gamma)|H|$ edges.
\end{lemma}

Next, we show that Property 1(b) of Section~\ref{sec:bounds.budget.rggds} holds for random graphs
on a fixed degree sequence, whenever that sequence satisfies Property 1(a).  Recall that Property 1(a) says that for all $\e>0$ there exists constant $d$ such that the sum of all degrees greater than $d$ is less than $\e n$.  Given $d$, we say that a {\em small edge} is an edge whose endpoints both have degree less than $d$. 
So we simply need to extend Lemma~\ref{claim:standard.fact} to random graphs on such a degree sequence.  We do so by adapting the standard proof.

\begin{lemma}
\label{claim:degseq}  Let $G$ be a random graph on a given degree sequence, as described above.
For any $\e,\g>0$ there exists $\a>0$ such that $H \subseteq G$ with fewer than $\alpha n$ vertices
has at most $(1 + \gamma)|H|$ small edges.
\end{lemma}

\begin{proof}
We use the configuration model of Bollob\'as (see eg.~\cite{bollobas-2001-book}).
So we create $\deg(v)$ copies of each vertex $v$, and take a random pairing of all the copies.
This yields a random multigraph in the obvious way. Other common properties of the degree sequence (see eg~\cite{molloy-1995-critical}) can allow us to condition on there being
no loops or multiedges and hence, obtain results for a random simple graph.

Suppose the total number of edges is $M=cn$; then the total number of vertex copies
is $2M=2cn$.

We bound the expected number of subgraphs $H$ with $a\leq\a n$ vertices and  fewer than $(1 + \gamma)|H|$ small edges. Clearly we
can assume that $H$ has no vertices of degree greater than $d$ (where $d=d(\e)$ comes from Property 1(a)) as a vertex of degree greater than $d$ cannot contribute any small edges.

First, choose a set of $a$ vertices.  Choose $2(1+\g)a$ copies of vertices from that set.
Then choose one of the 
\begin{equation}
\nonumber
\frac{(2(1+\g)a)!}{2^{(1+\g)a}((1+\g)a)!}< \left(\frac{2(1+\g)a}{e}\right)^{(1+\g)a} 
\end{equation}
ways to group
those copies into pairs.  The probability that each of those pairs arises in the random pairing
is 
\begin{equation}
\nonumber
\frac{1}{2M-1}\cdot \ldots \cdot \frac{1}{2M-2(1+\g)a+1}<\left(\frac{1}{cn}\right)^{(1+\g)a}
\end{equation}
for $a<\a n$ with $\a$ sufficiently small.  Putting it all together, and using the fact that
each of the chosen vertices has degree at most $d$, the expected number of
sets is at most
\begin{eqnarray*}
&&{n\choose a}{da\choose2(1+\g)a}\left(\frac{2(1+\g)a}{e}\right)^{(1+\g)a}\left(\frac{1}{cn}\right)^{(1+\g)a}\\
&<& \left(\frac{en}{a}\right)^a \left(\frac{eda}{2(1+\g)a} \right)^{2(1+\g)a} \left(\frac{2(1+\g)a}{e}\right)^{(1+\g)a}\left(\frac{1}{cn}\right)^{(1+\g)a}\\
&<&Z^a \left(\frac{a}{n} \right)^{\g a},
\end{eqnarray*}
where $Z=Z(d,\g)=Z(\e,\g)$ is a constant.  Taking $a<\a n$ where $\a^{\g}<\frac{1}{2Z}$, the rest now follows as in the proof of Lemma~\ref{claim:standard.fact}.
\end{proof}

The following diagrams represent the experiments in the study on the performance of the heuristic-based algorithm and are accompanied with Section~\ref{sec:experiments}.

\begin{figure}
\centering
\begin{tabular}{cc}
\includegraphics[width=0.49\textwidth,height=3.8cm]{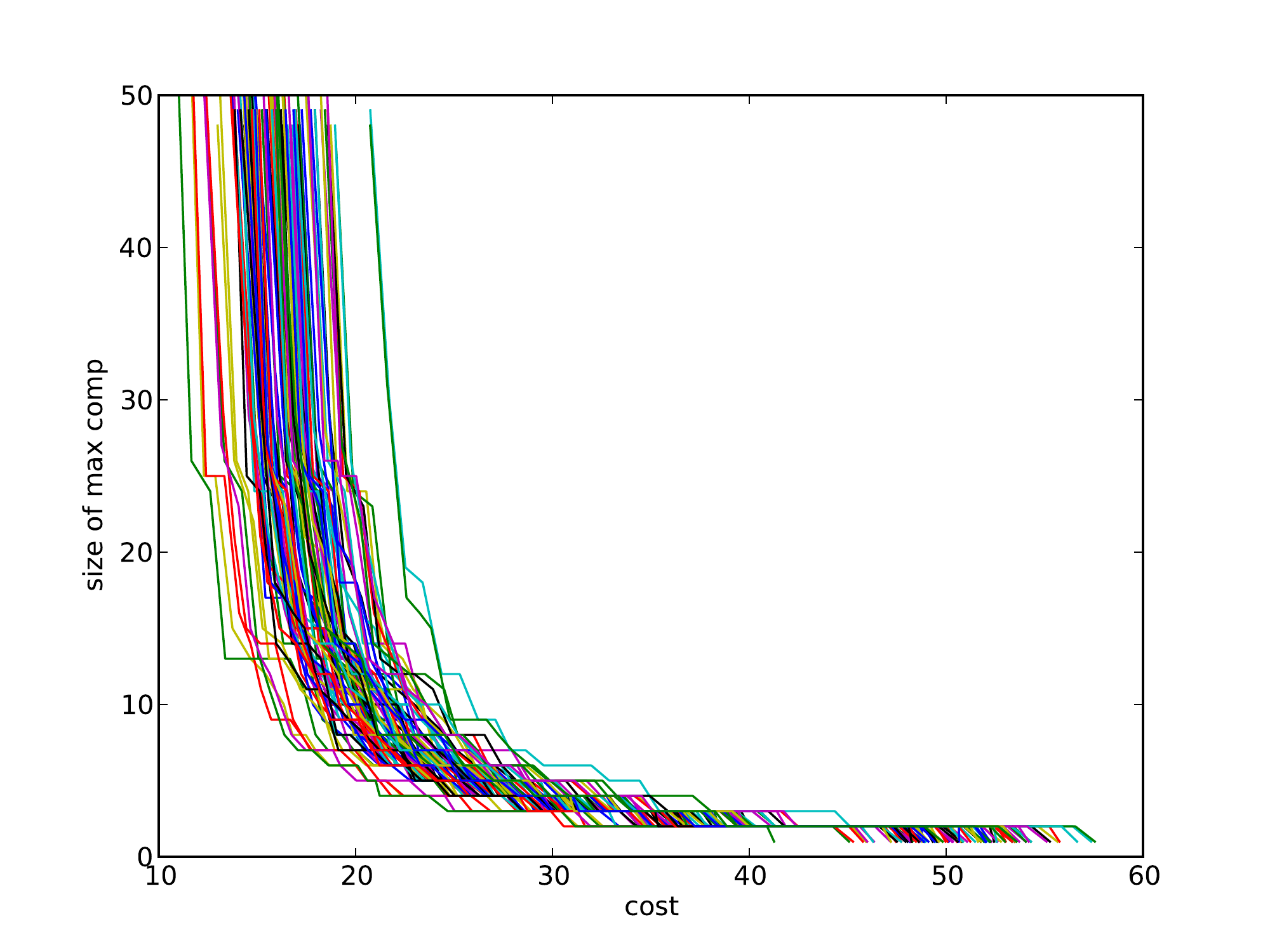} &
\includegraphics[width=0.49\textwidth,height=3.8cm]{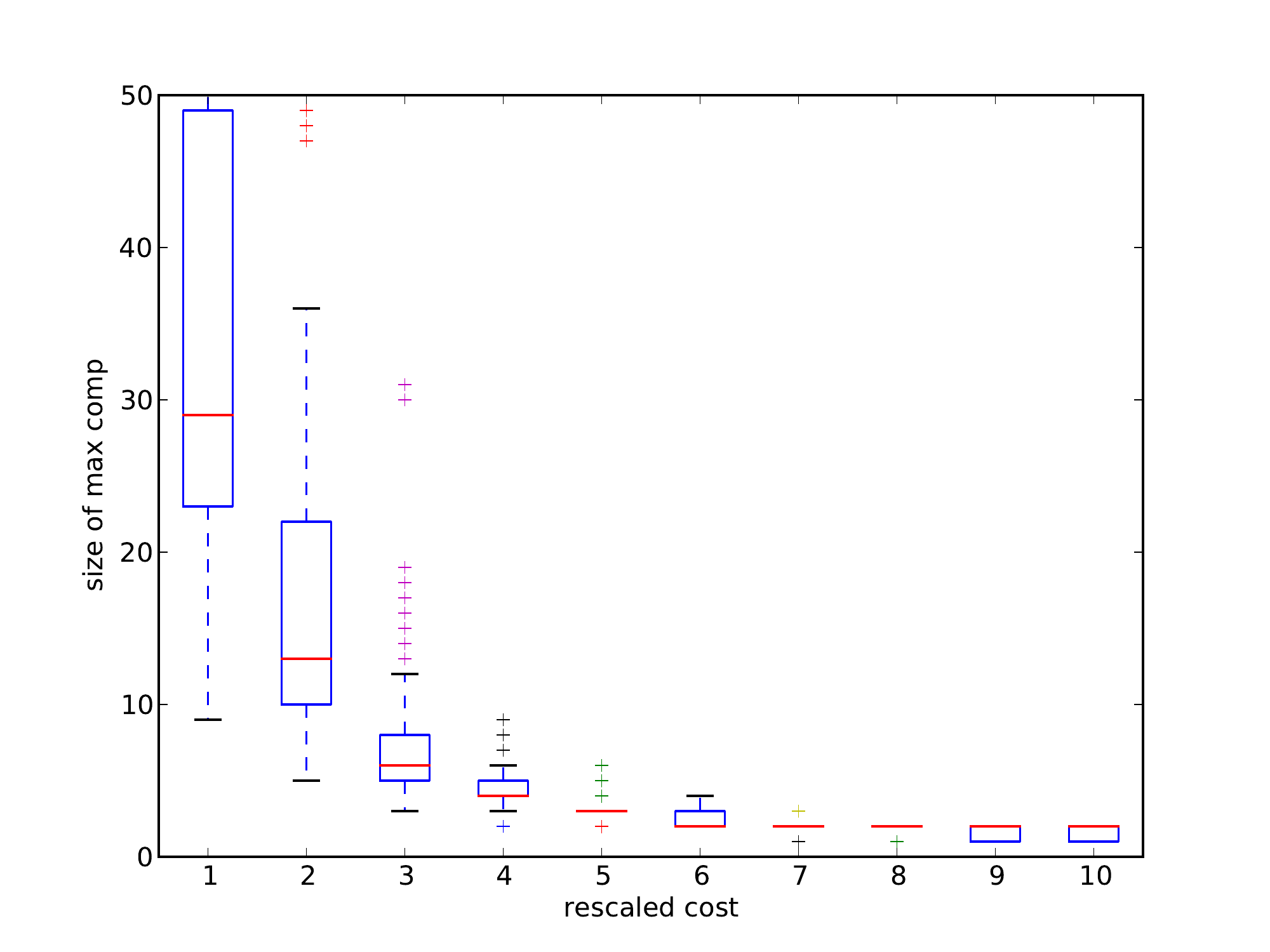} \\
(1) MaxSF-Susceptibility & (2) Box-plot view of (1) \\
\includegraphics[width=0.49\textwidth,height=3.8cm]{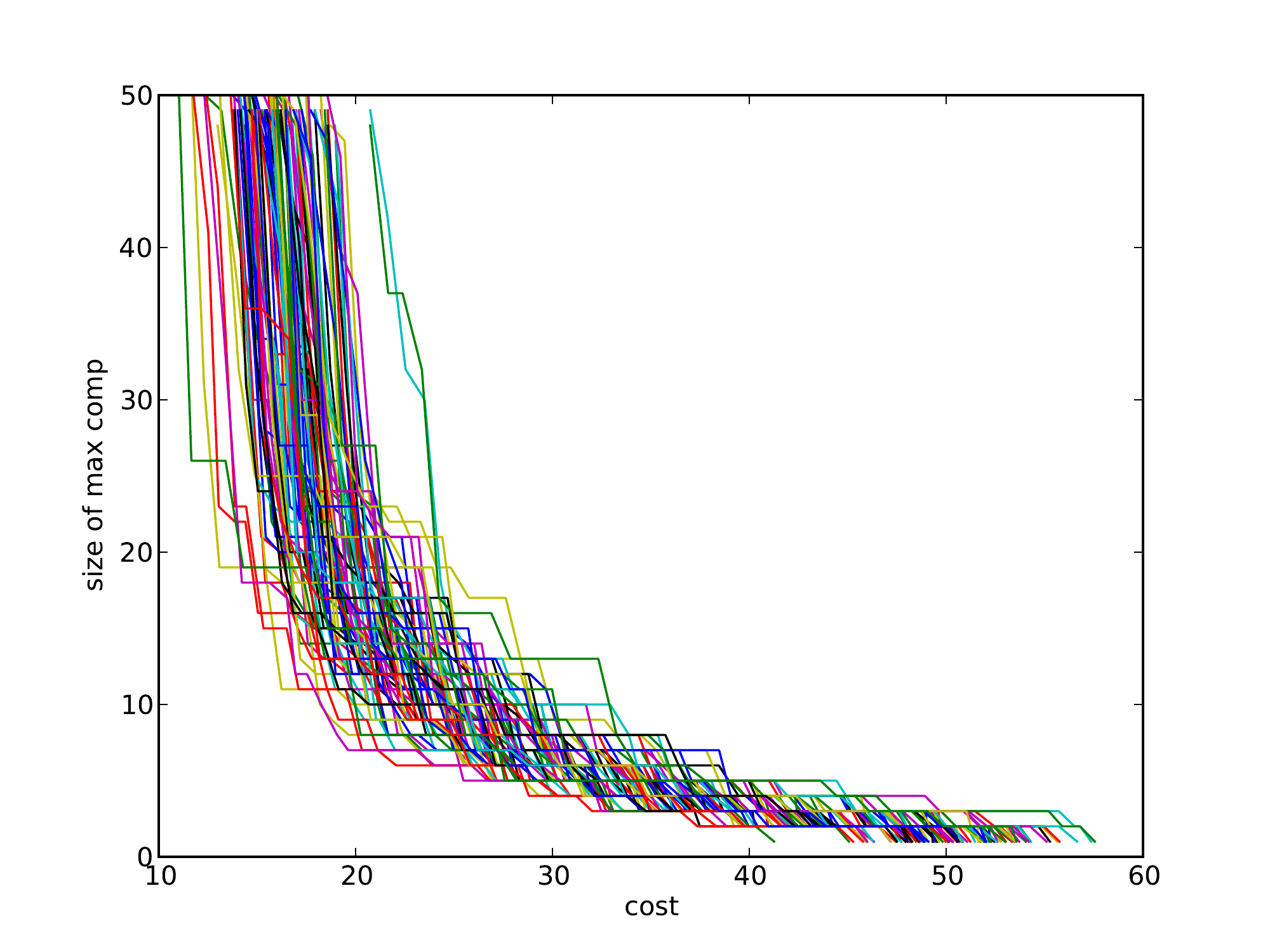} &
\includegraphics[width=0.49\textwidth,height=3.8cm]{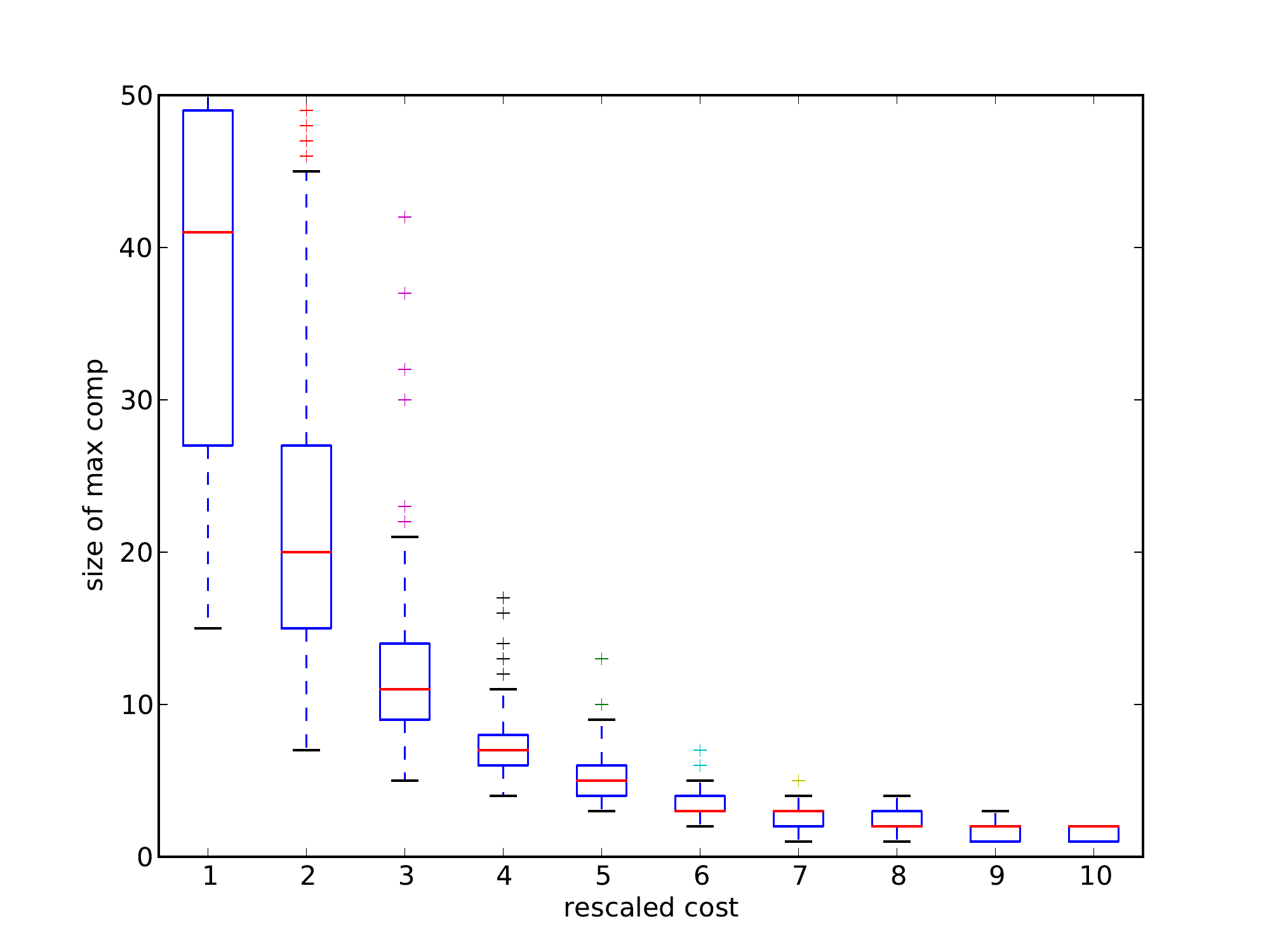} \\
(3) MaxSF-Betweenness & (4) Box-plot view of (3) \\
\includegraphics[width=0.49\textwidth,height=3.8cm]{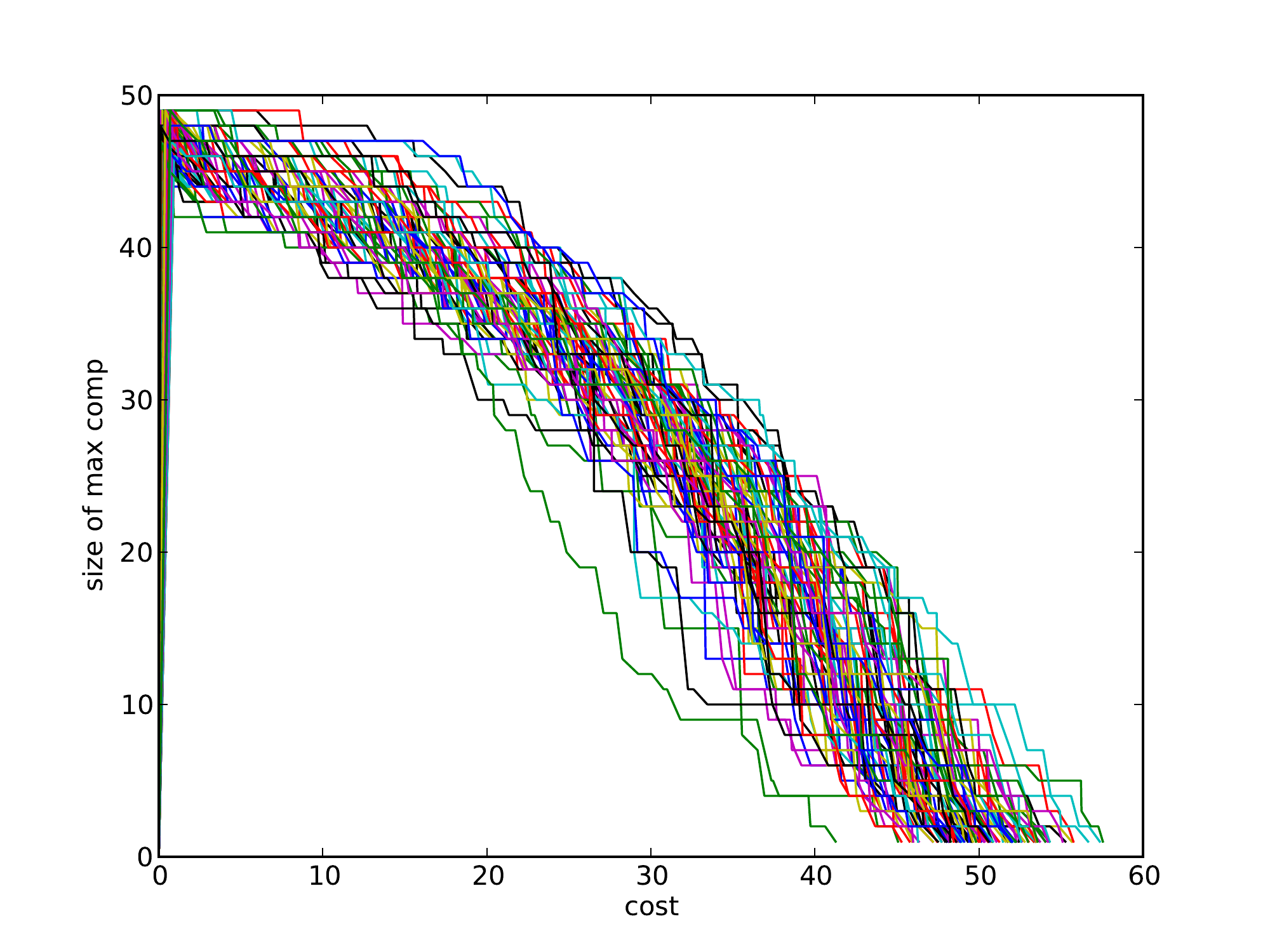} &
\includegraphics[width=0.49\textwidth,height=3.8cm]{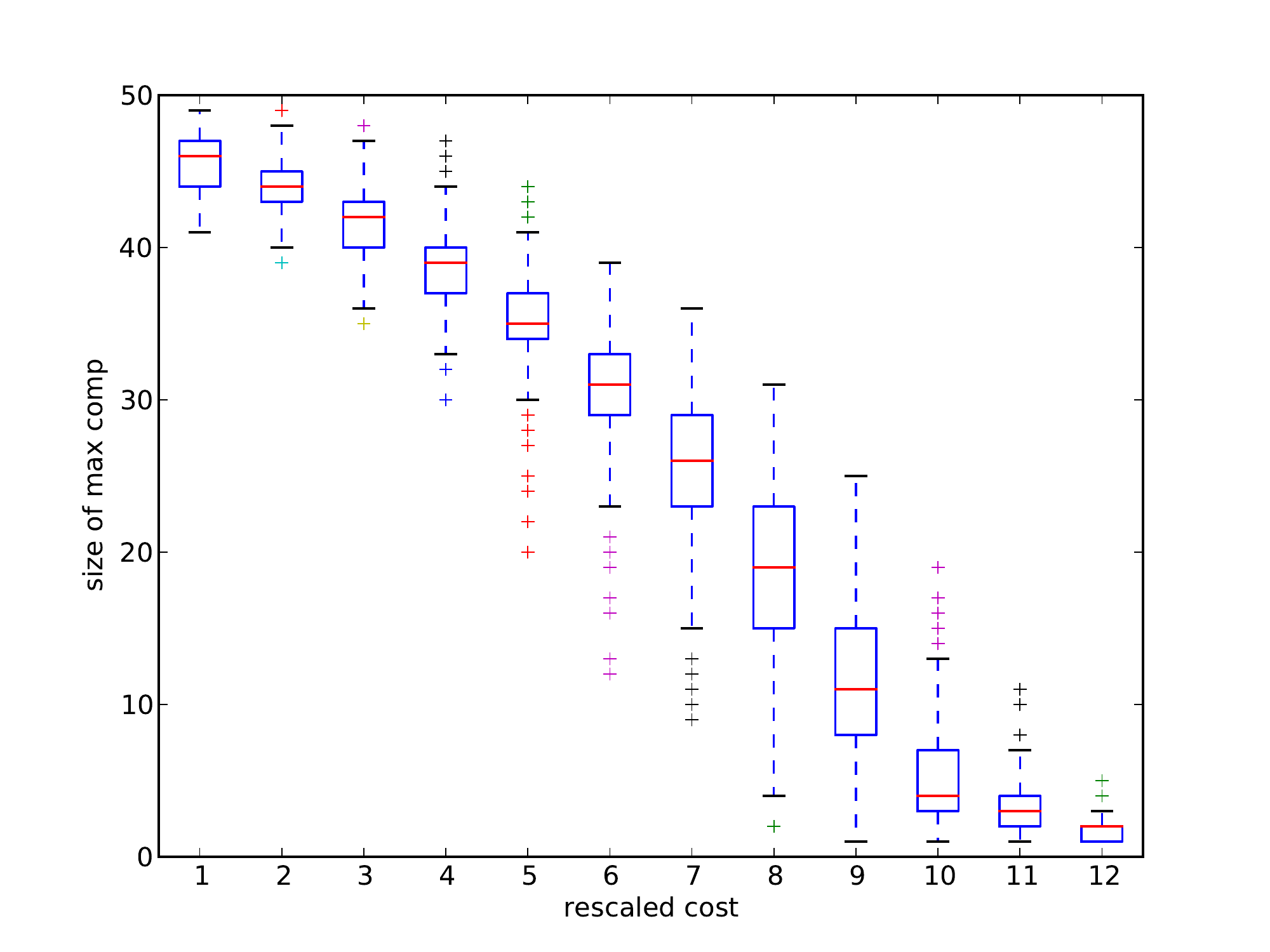} \\
(5) Full-Susceptibility & (6) Box-plot view of (5) \\
\includegraphics[width=0.49\textwidth,height=3.8cm]{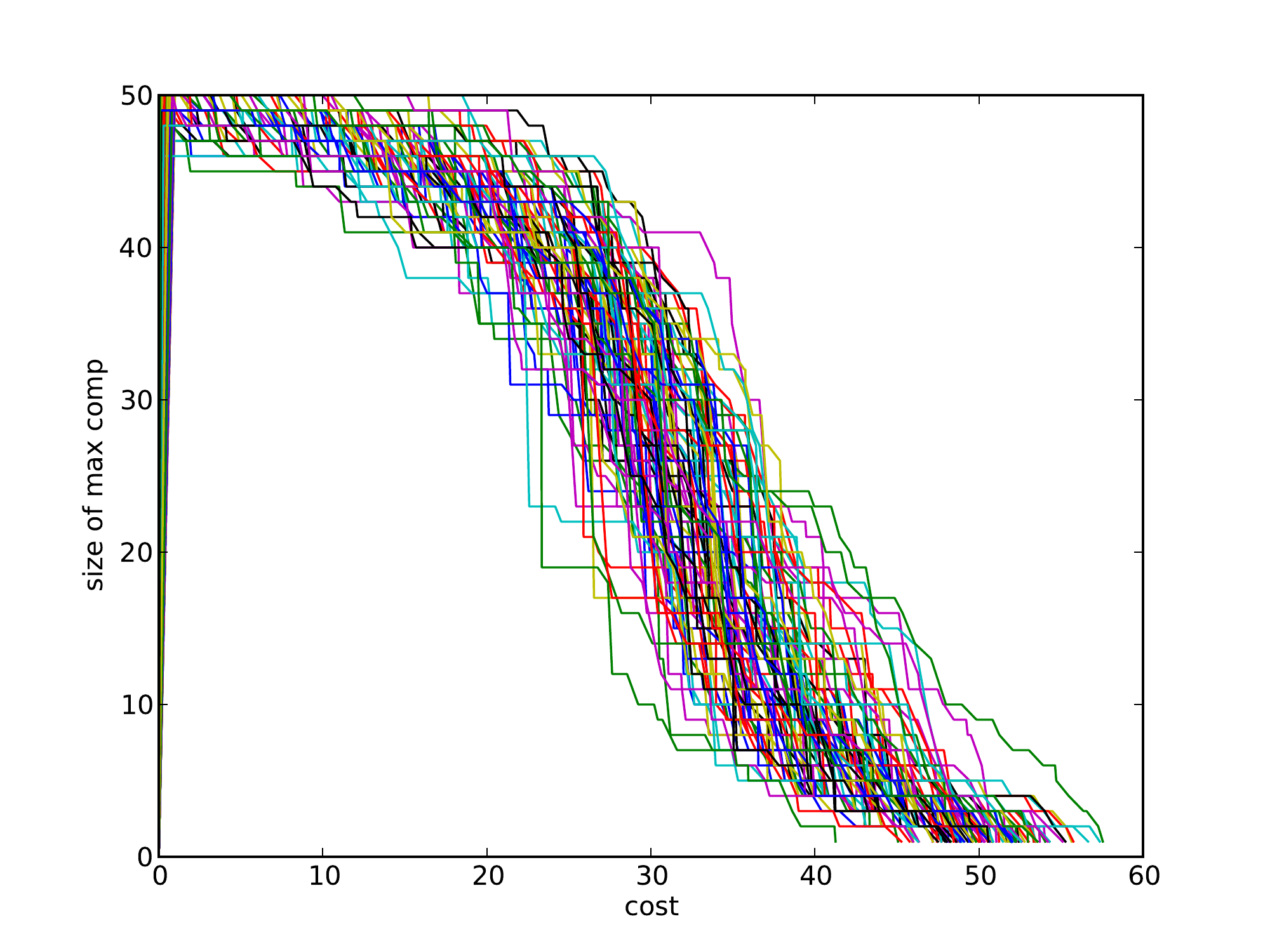} &
\includegraphics[width=0.49\textwidth,height=3.8cm]{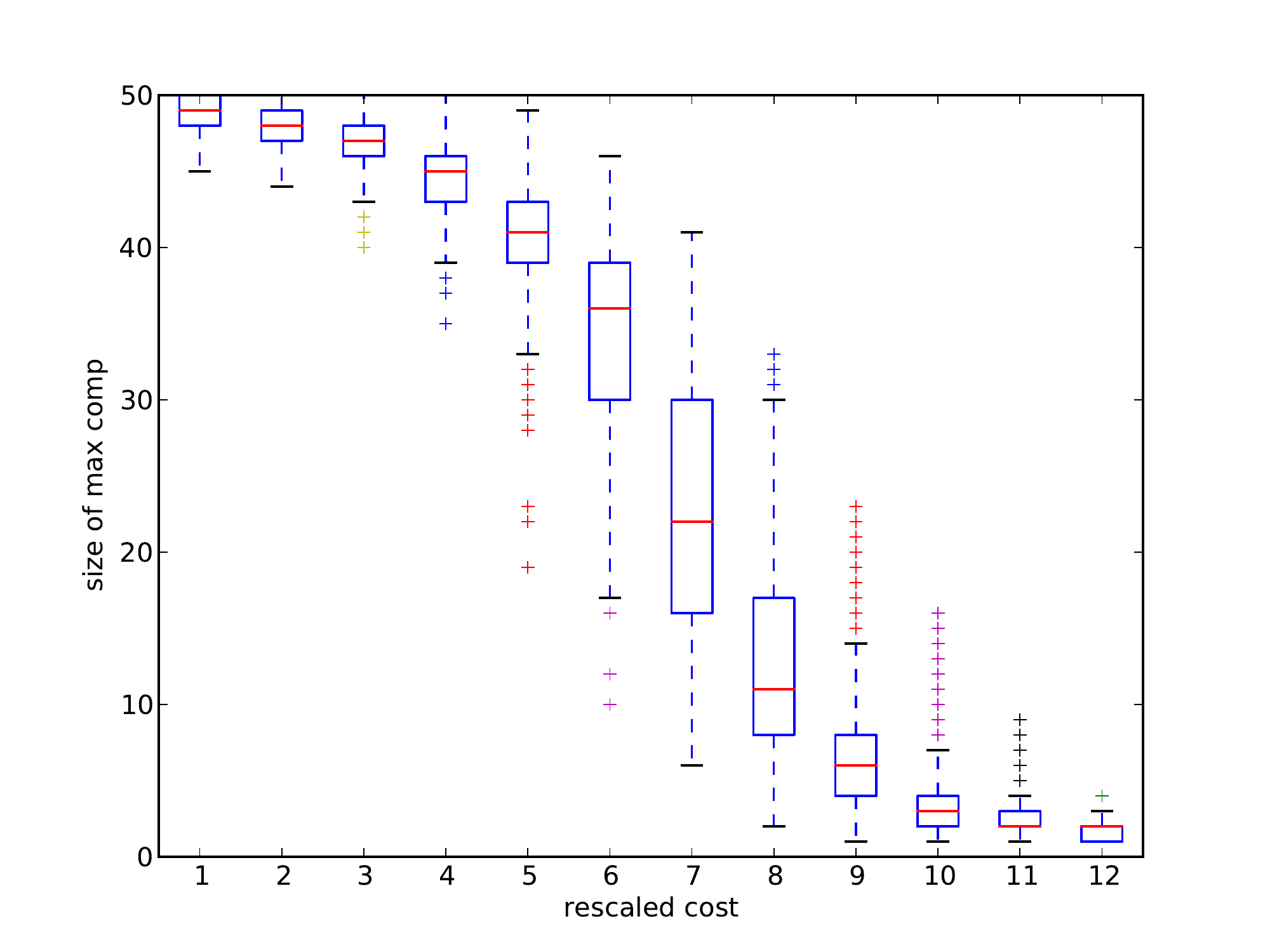} \\
(7) Full-Betweenness & (8) Box-plot view of (7)
\end{tabular}
\caption{Size of maximum component vs. edge removing cost for $G_{n,M}$ model with uniform edge weights.}
\label{fig:gnm-unif}
\end{figure}

\begin{figure}
\centering
\begin{tabular}{cc}
\includegraphics[width=0.49\textwidth,height=3.8cm]{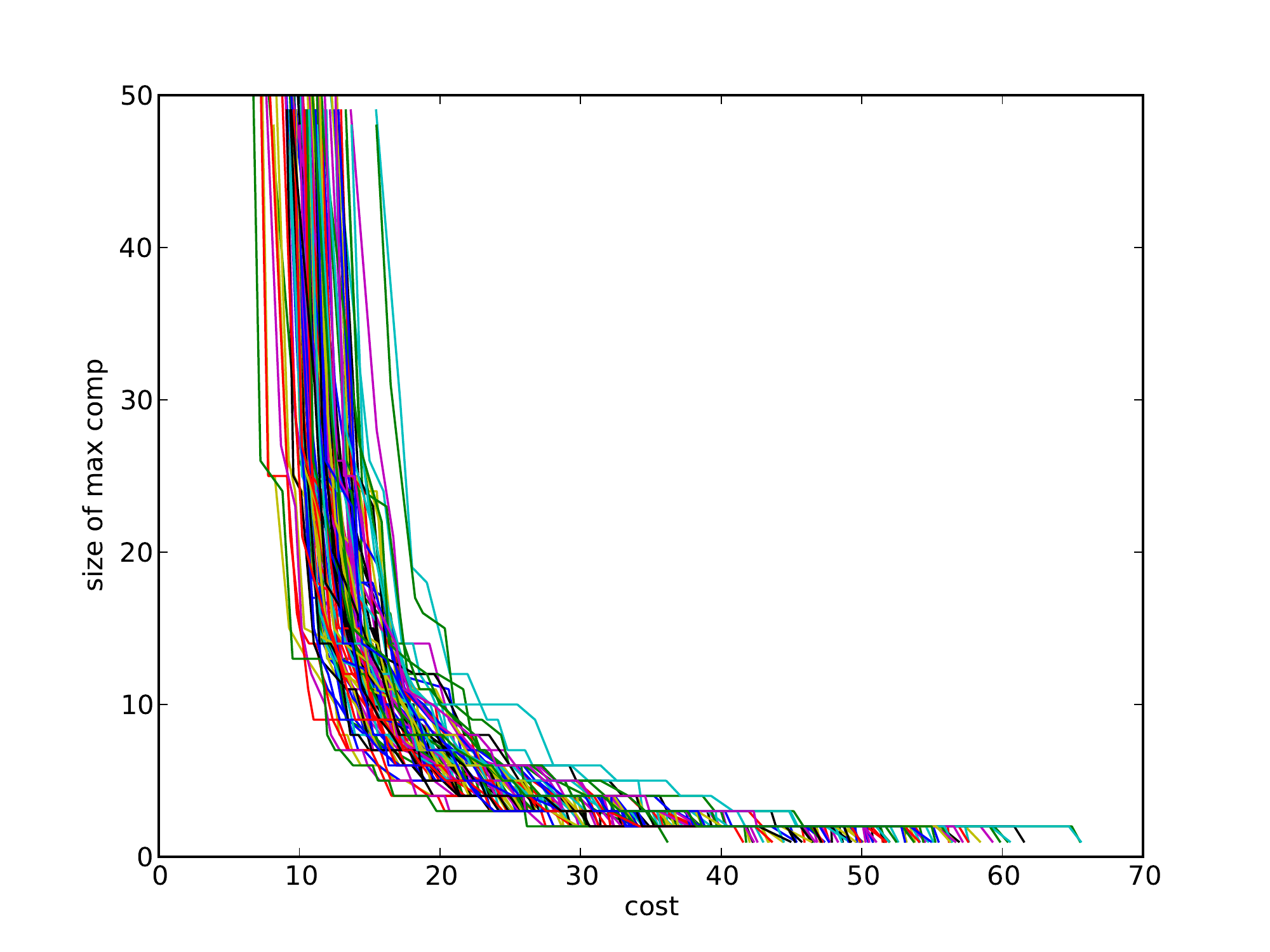} &
\includegraphics[width=0.49\textwidth,height=3.8cm]{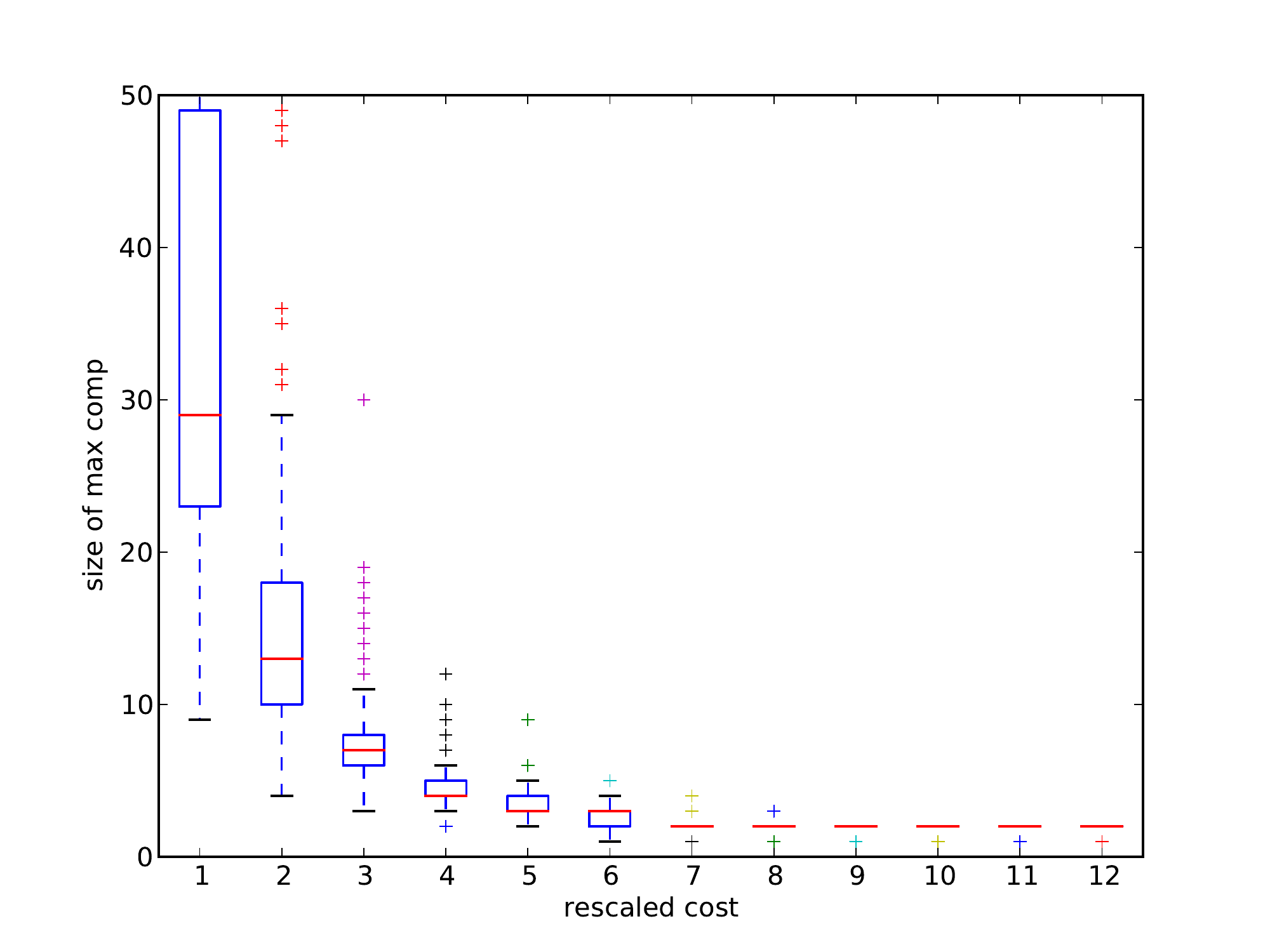} \\
(1) MaxSF-Susceptibility  & (2) Box-plot view of (1) \\
\includegraphics[width=0.49\textwidth,height=3.8cm]{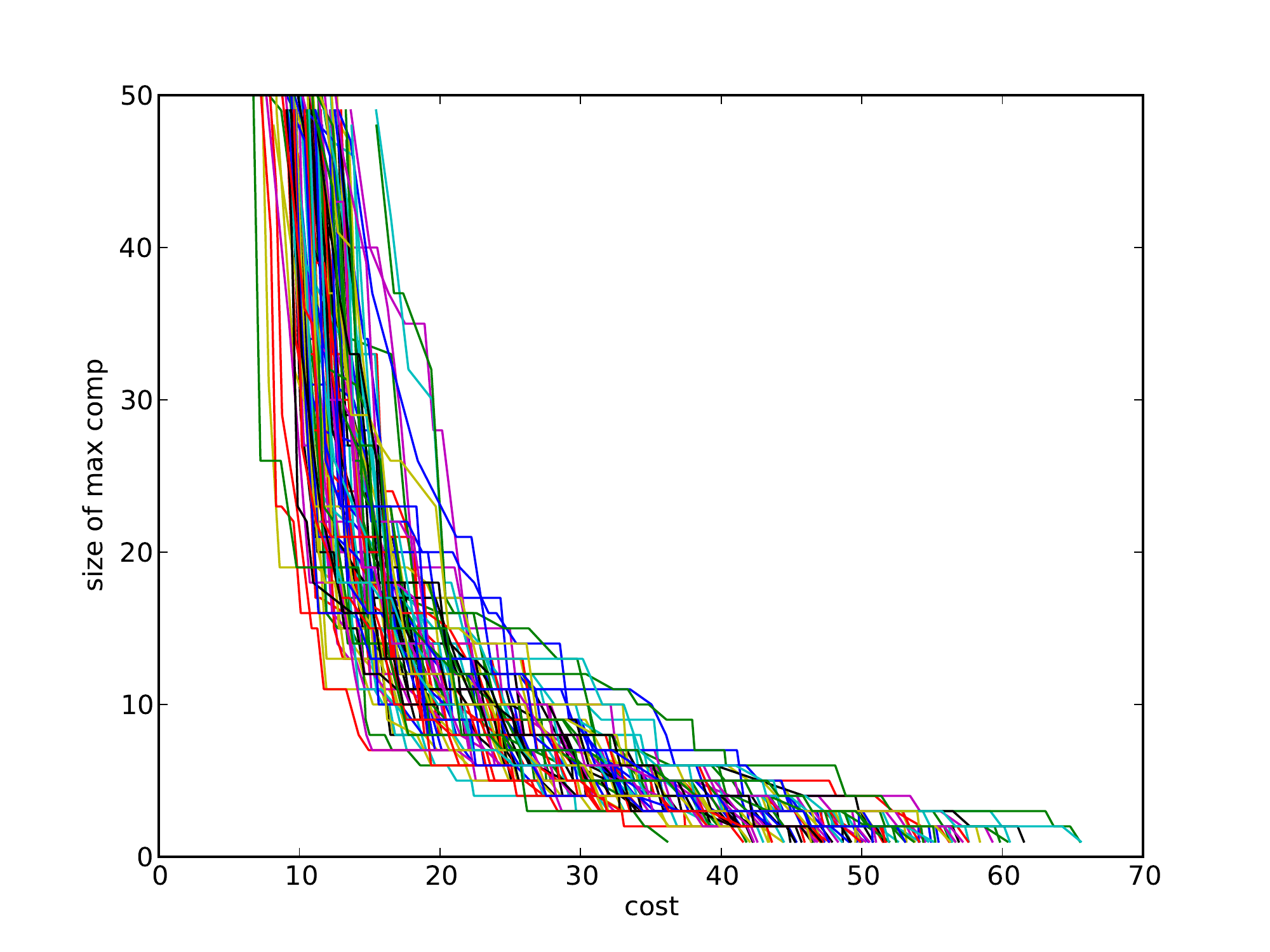} &
\includegraphics[width=0.49\textwidth,height=3.8cm]{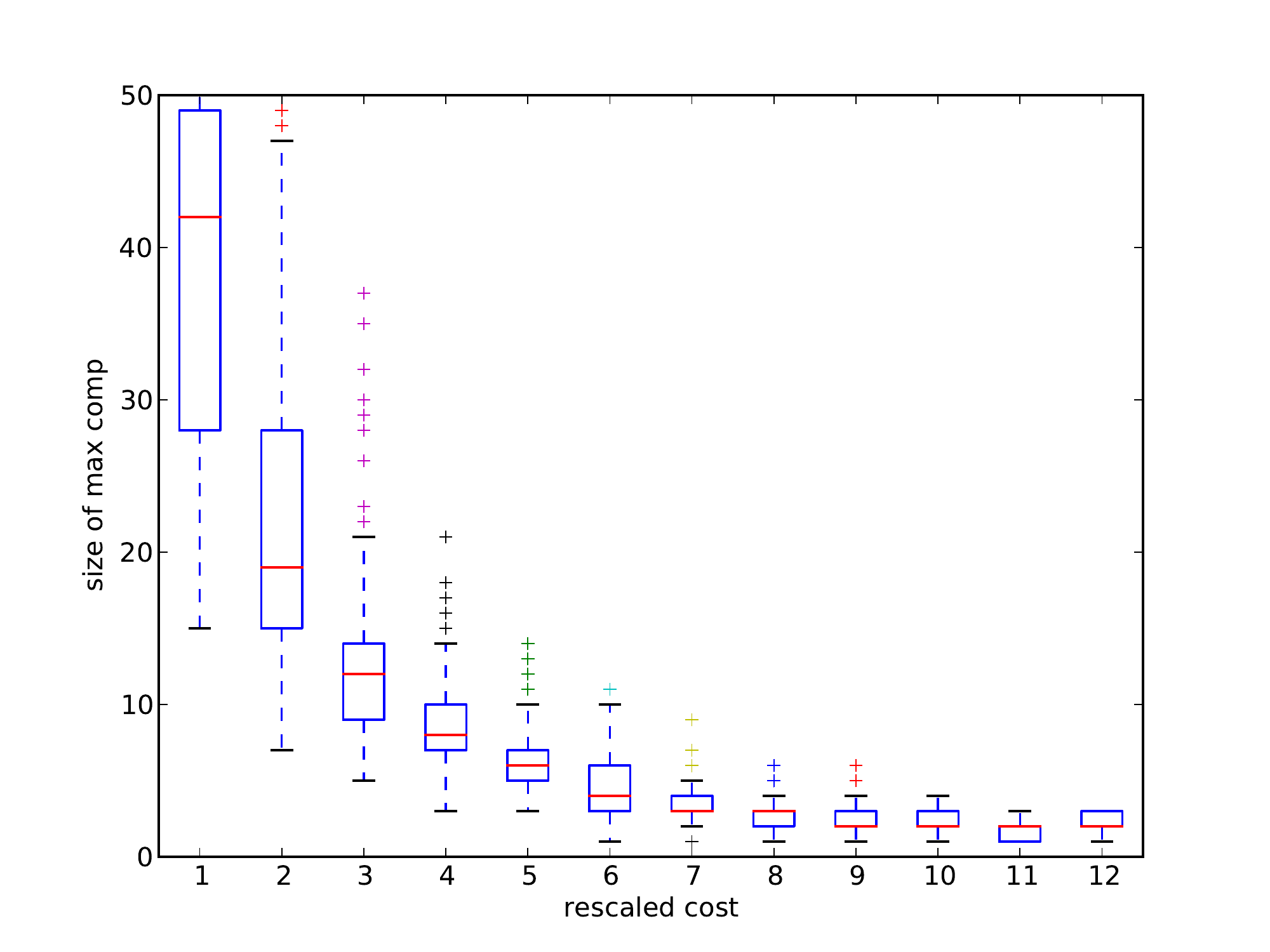} \\
(3) MaxSF-Betweenness & (4) Box-plot view of (3) \\
\includegraphics[width=0.49\textwidth,height=3.8cm]{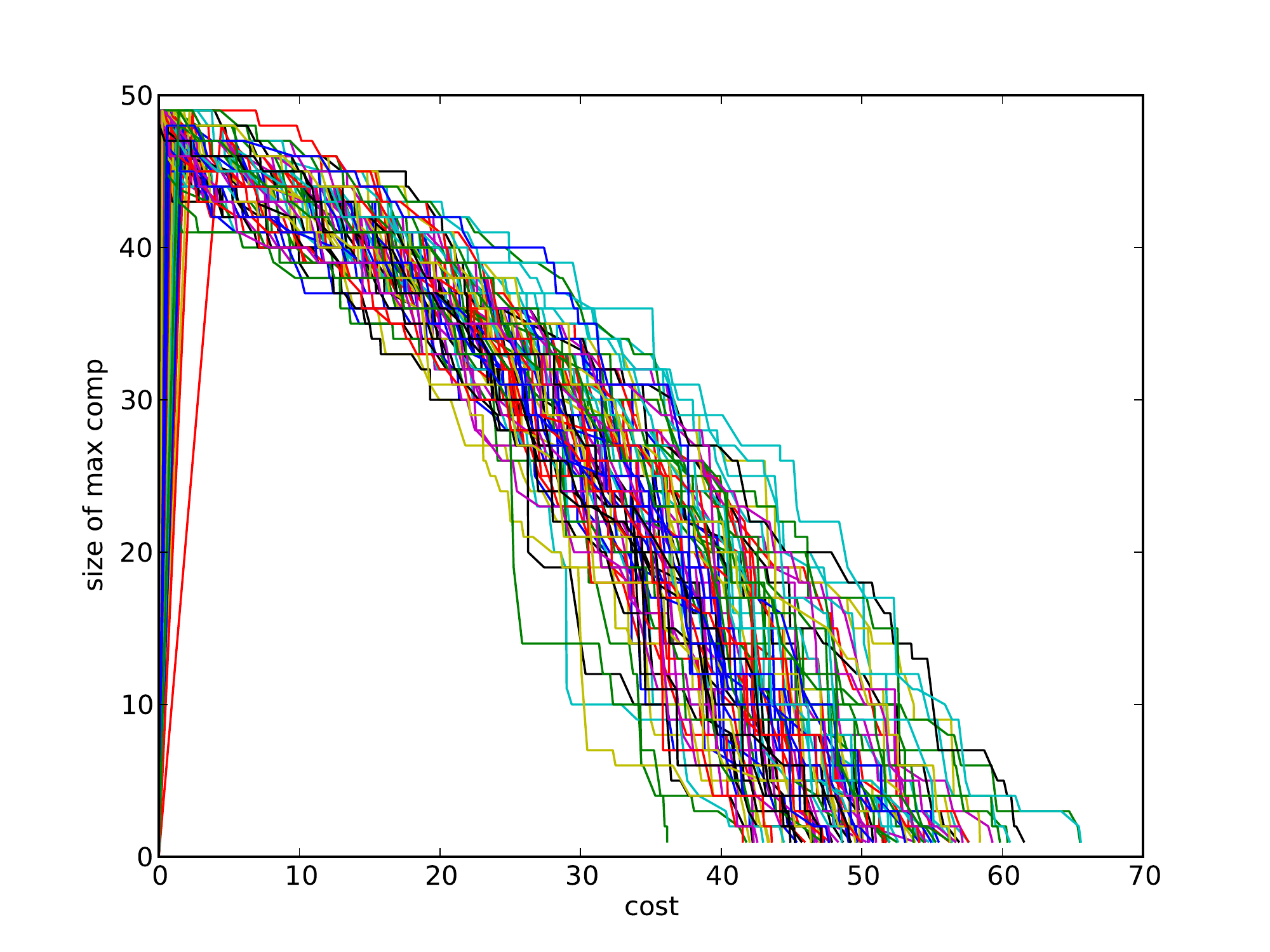} &
\includegraphics[width=0.49\textwidth,height=3.8cm]{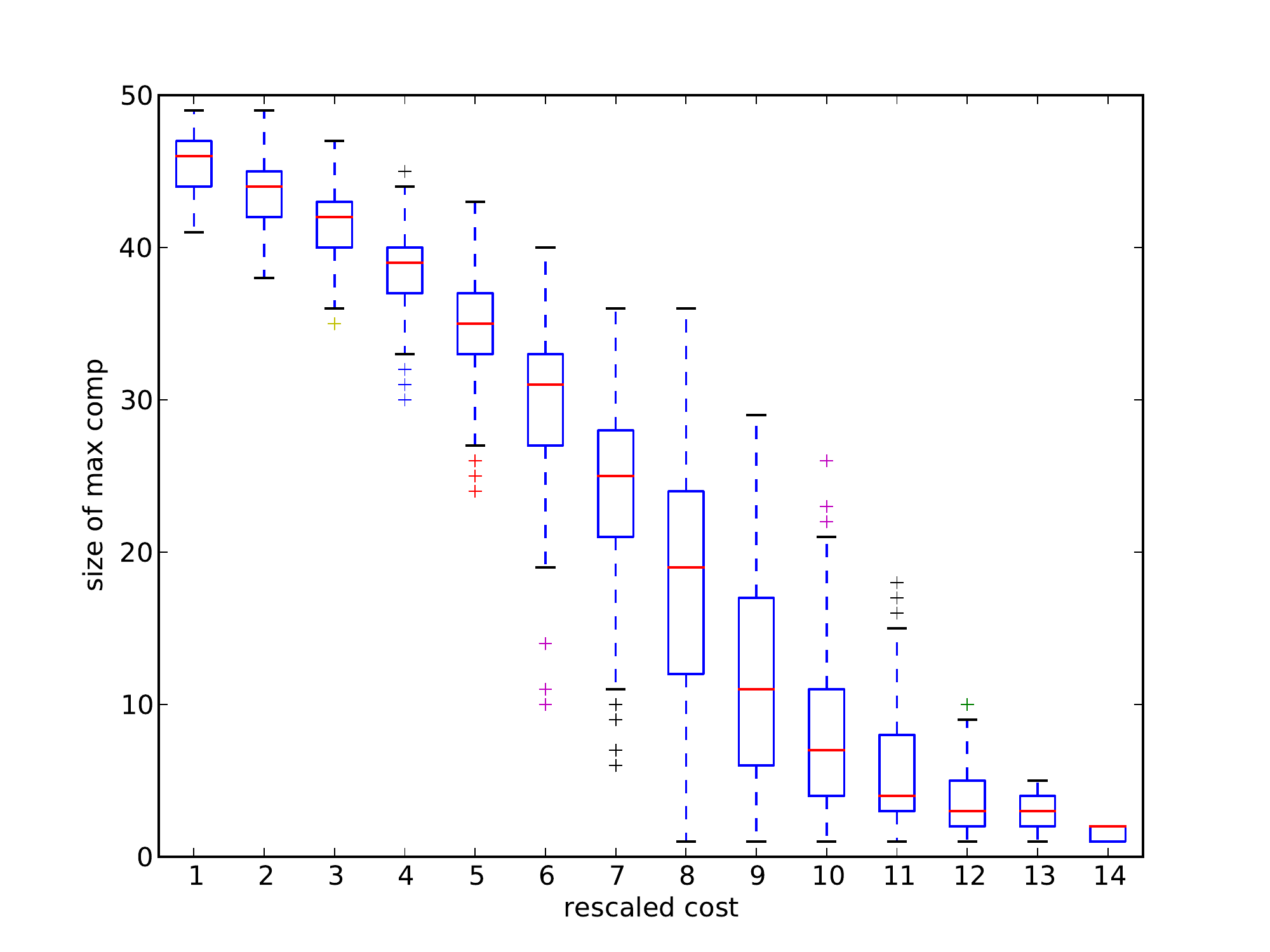} \\
(5) Full-Susceptibility & (6) Box-plot view of (5) \\
\includegraphics[width=0.49\textwidth,height=3.8cm]{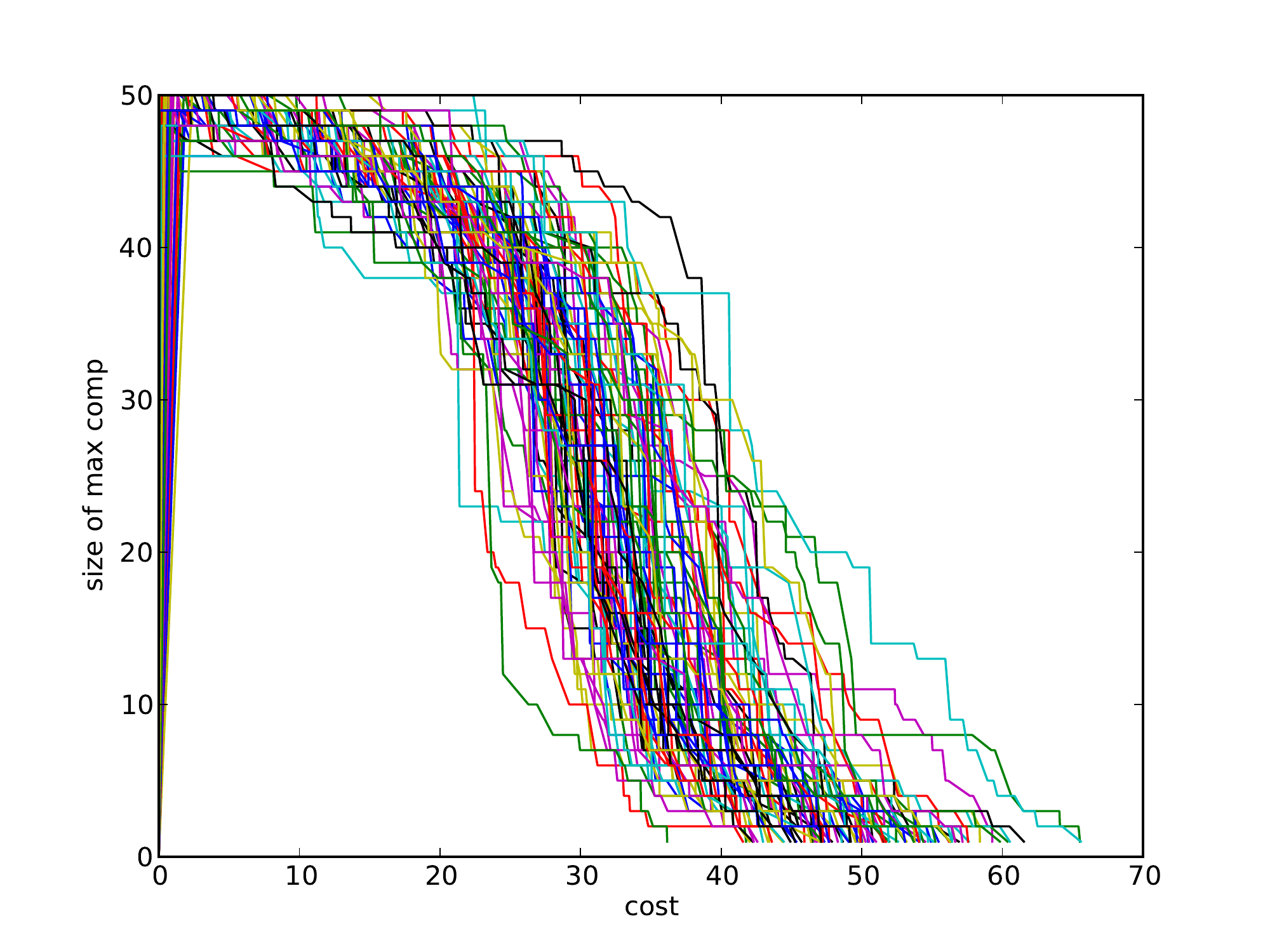} &
\includegraphics[width=0.49\textwidth,height=3.8cm]{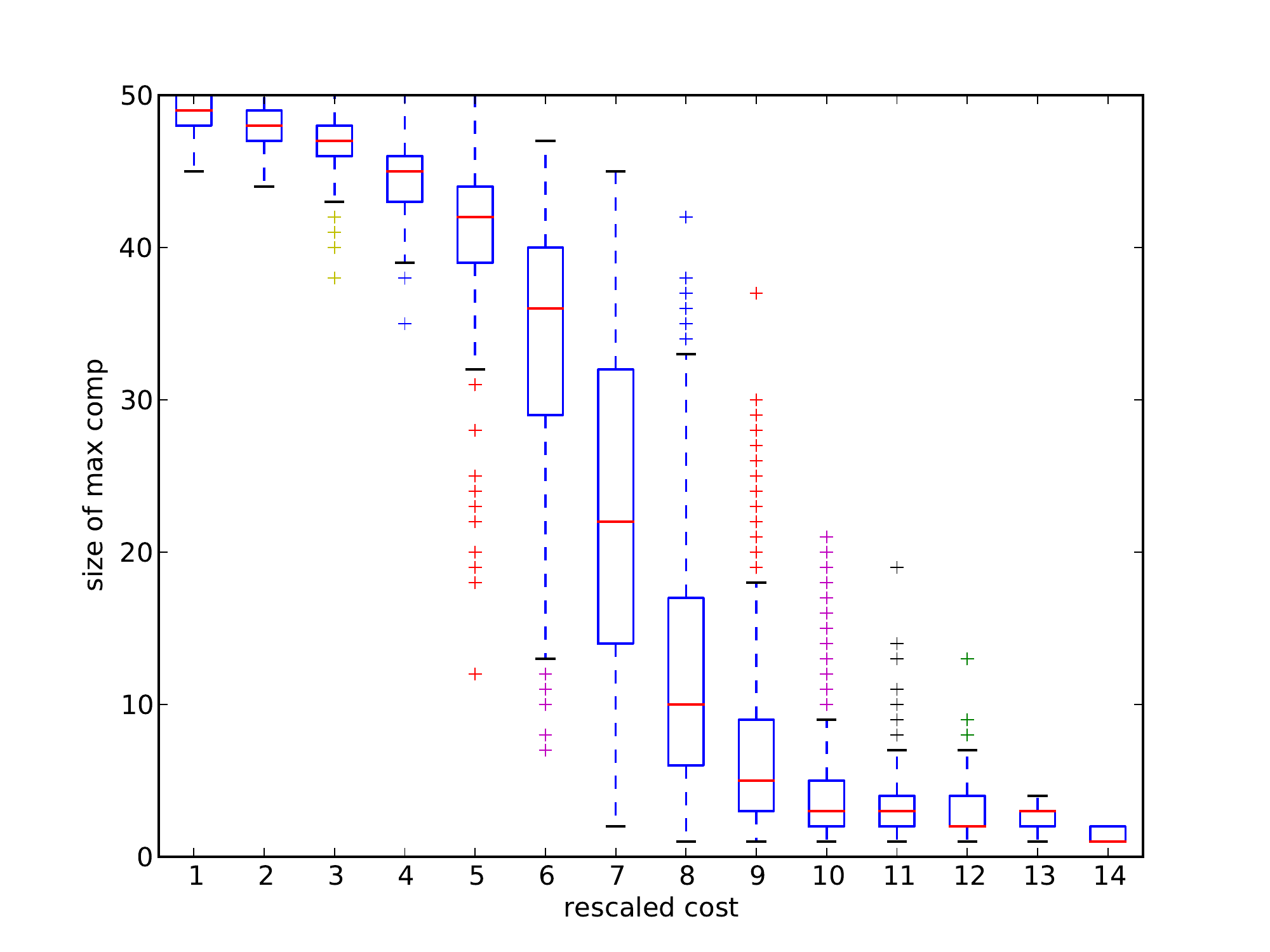} \\
(7) Full-Betweenness & (8) Box-plot view of (7) 
\end{tabular}
\caption{Size of maximum component vs. edge removing cost for $G_{n,M}$ model with exponential edge weights.}
\label{fig:gnm-exp}
\end{figure}

\begin{figure}
\centering
\begin{tabular}{cc}
\includegraphics[width=0.49\textwidth,height=3.8cm]{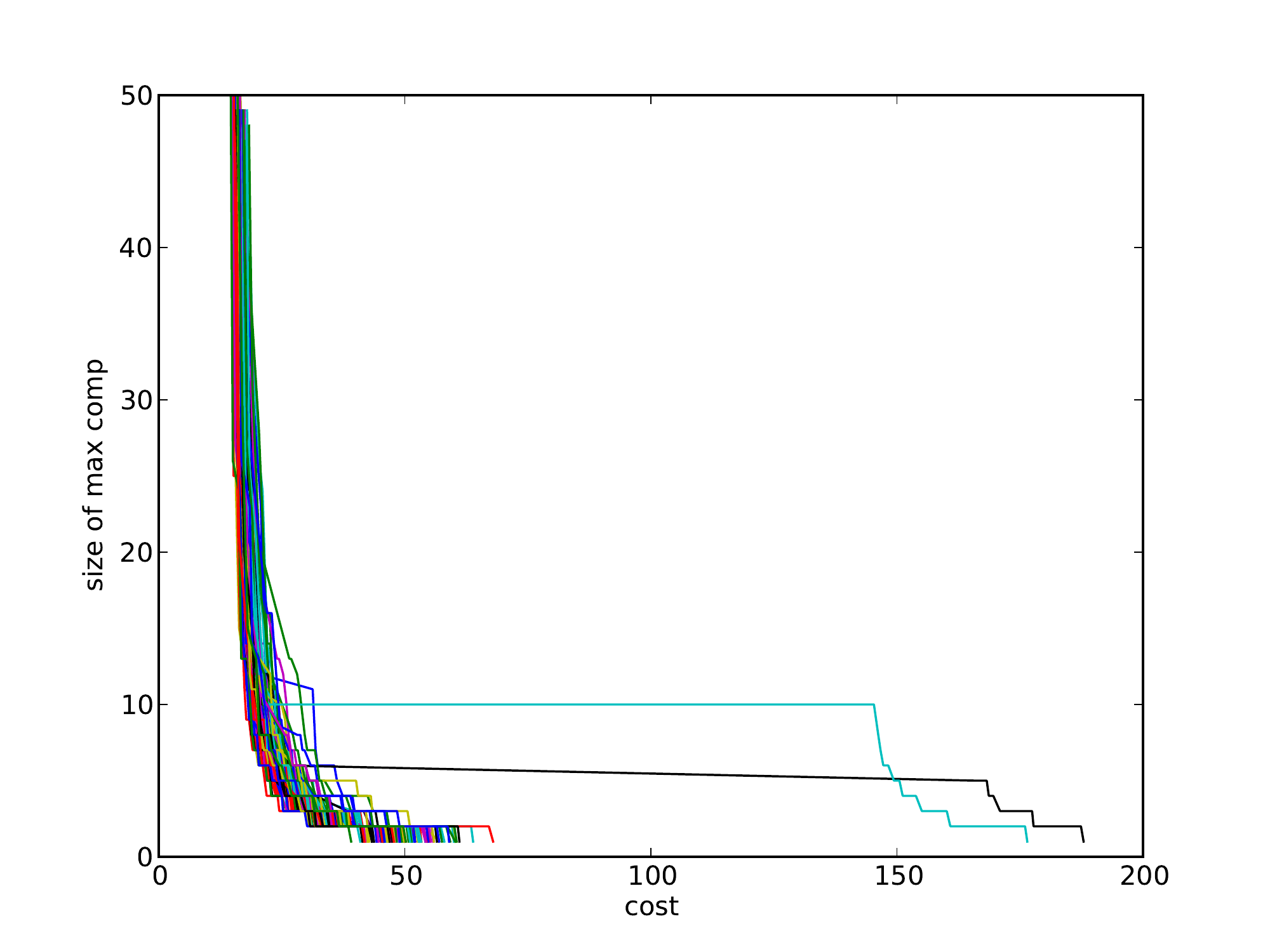} &
\includegraphics[width=0.49\textwidth,height=3.8cm]{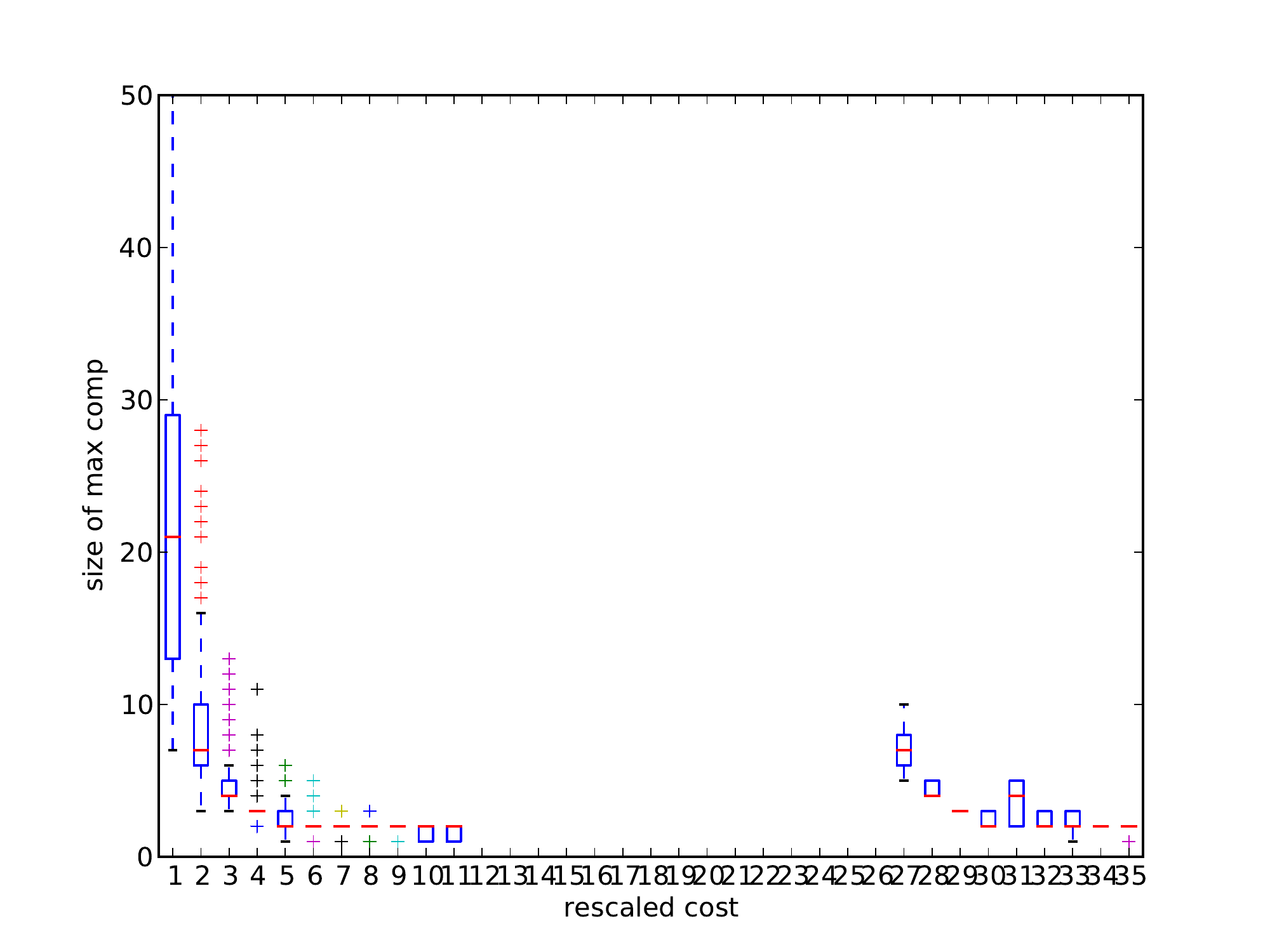} \\
(1) MaxSF-Susceptibility & (2) Box-plot view of (1) \\
\includegraphics[width=0.49\textwidth,height=3.8cm]{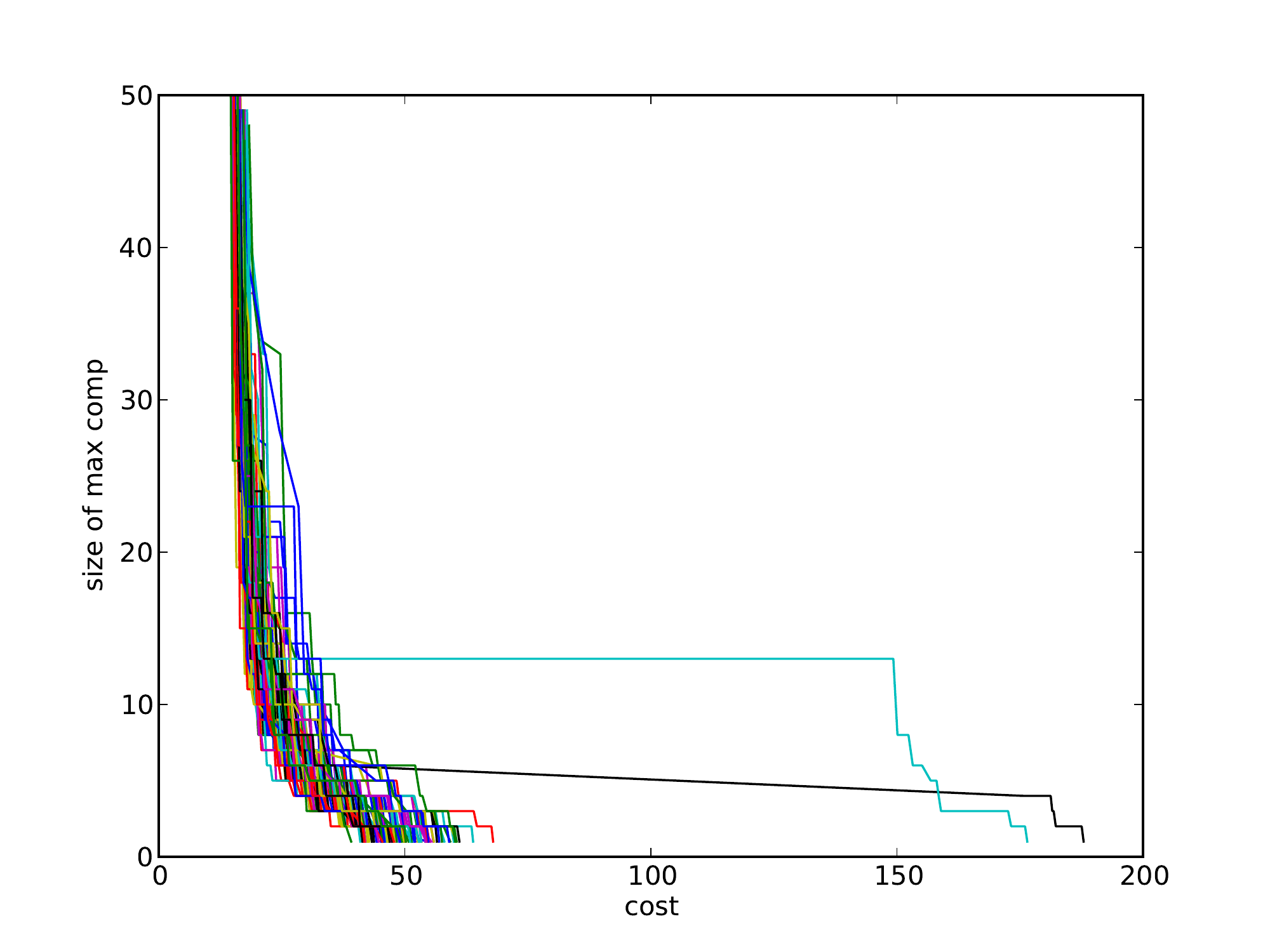} &
\includegraphics[width=0.49\textwidth,height=3.8cm]{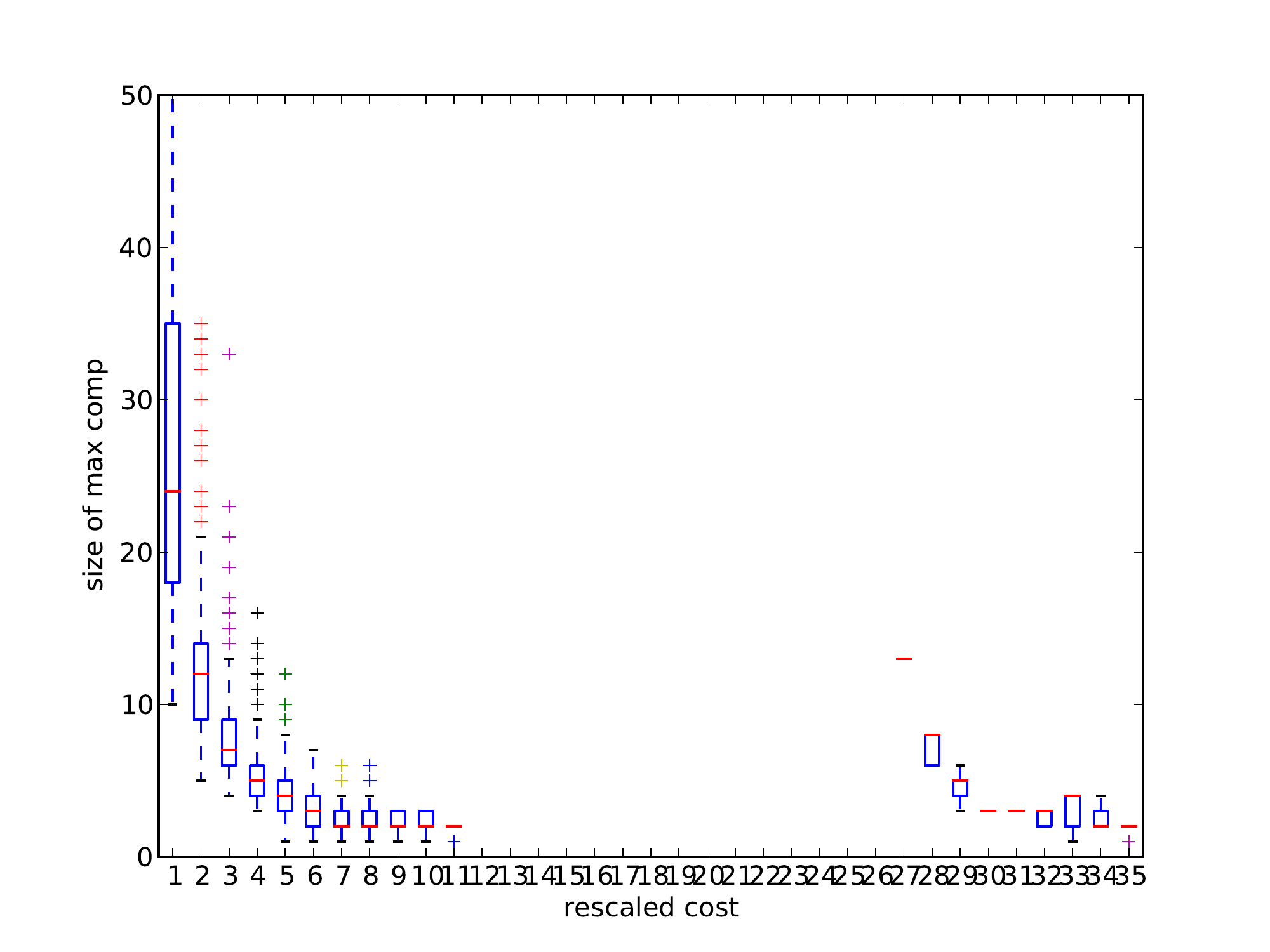} \\
(3) MaxSF-Betweenness & (4) Box-plot view of (3) \\
\includegraphics[width=0.49\textwidth,height=3.8cm]{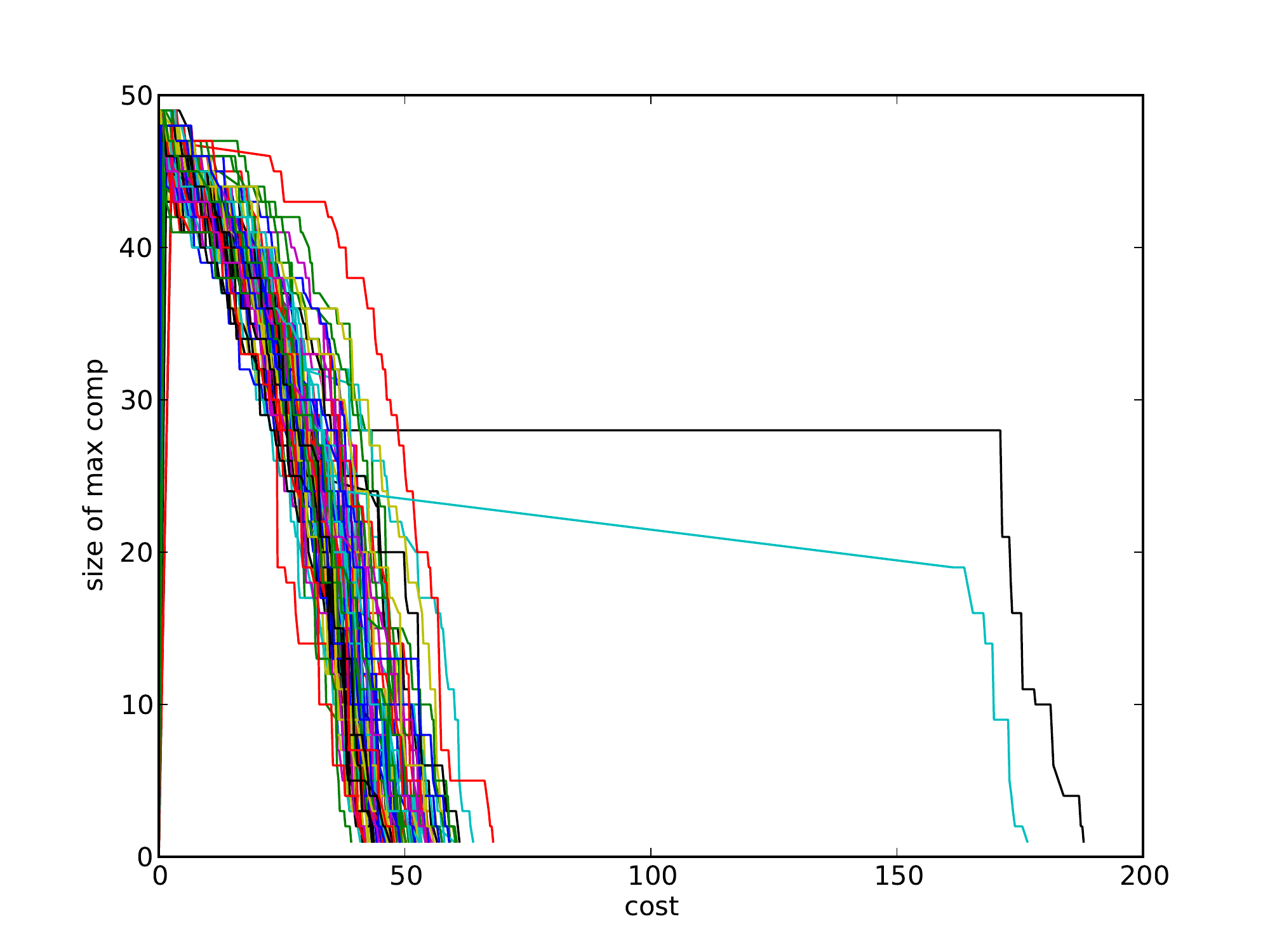} &
\includegraphics[width=0.49\textwidth,height=3.8cm]{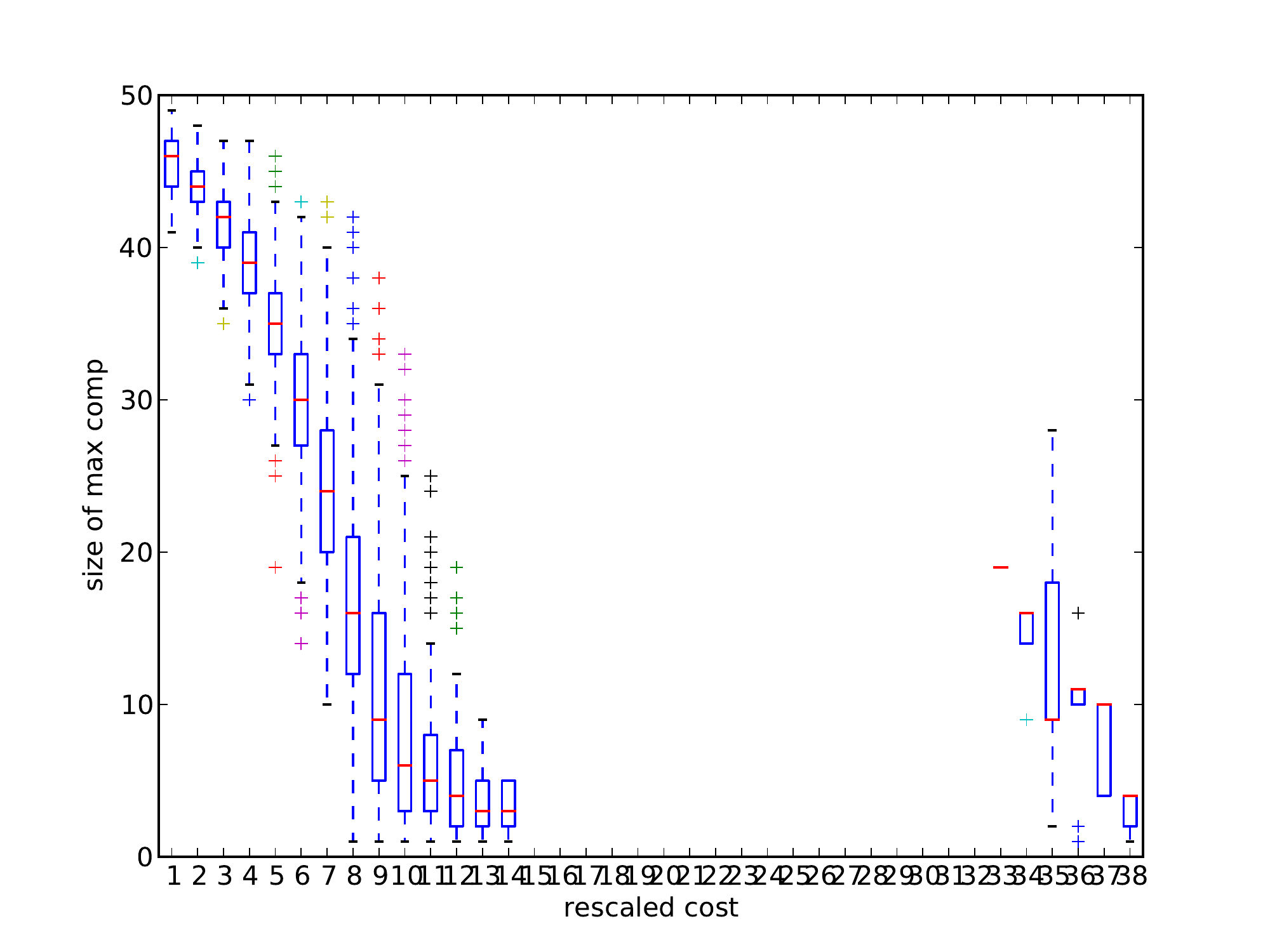} \\
(5) Full-Susceptibility & (6) Box-plot view of (5) \\
\includegraphics[width=0.49\textwidth,height=3.8cm]{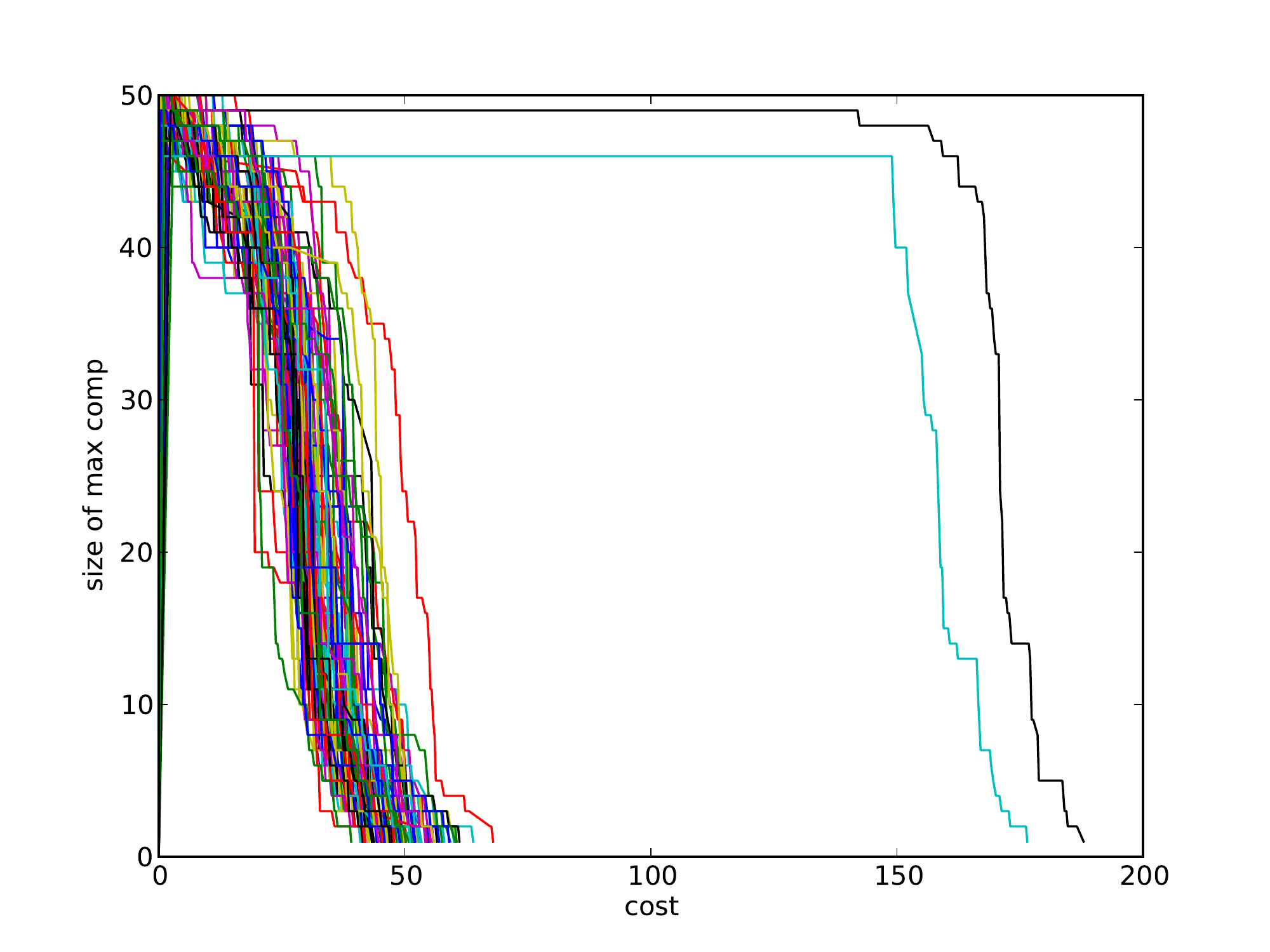} &
\includegraphics[width=0.49\textwidth,height=3.8cm]{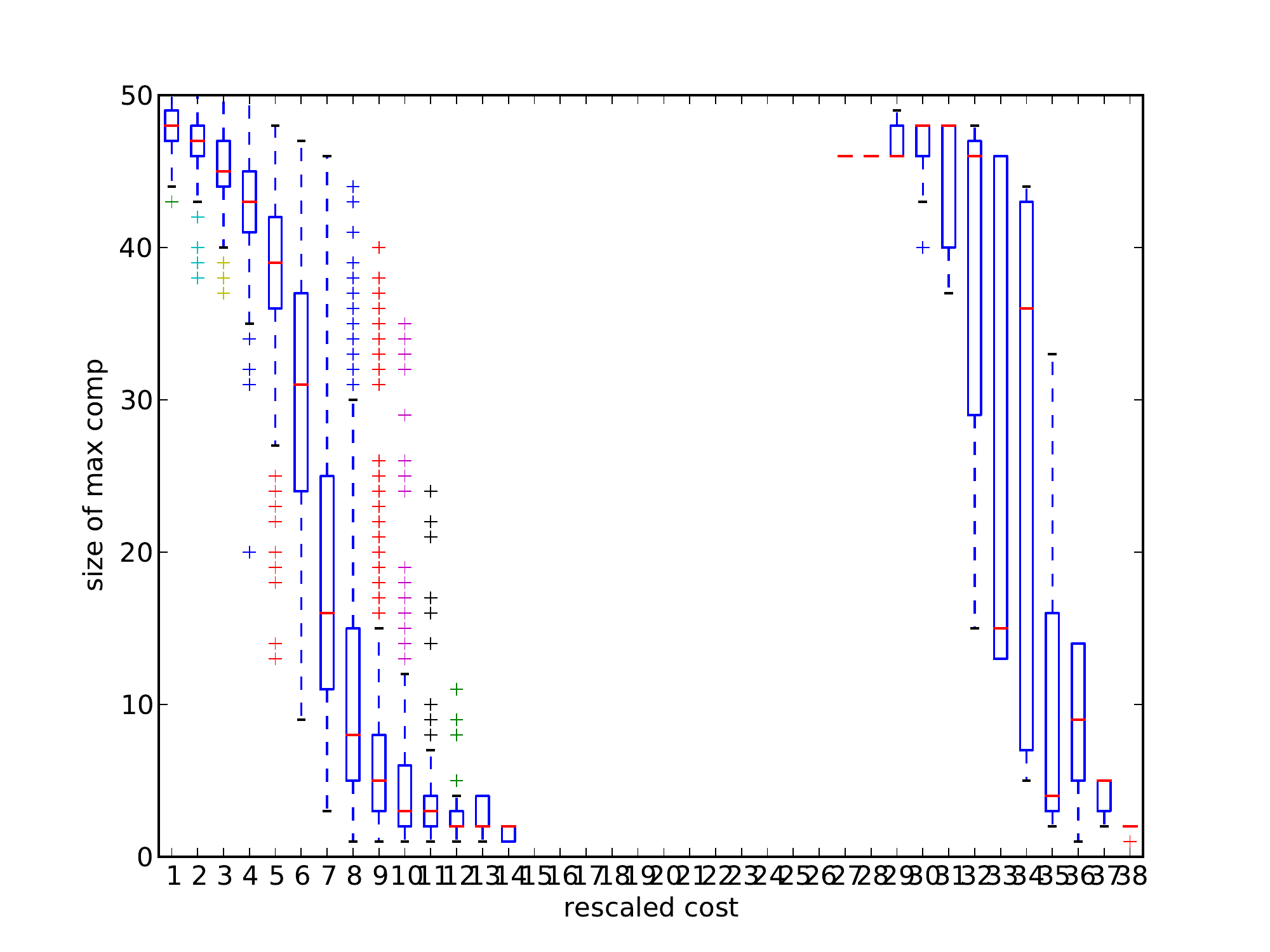} \\
(7) Full-Betweenness& (8) Box-plot view of (7)
\end{tabular}
\caption{Size of maximum component vs. edge removing cost for $G_{n,M}$ model with power-law edge weights.}
\label{fig:gnm-powerlaw}
\end{figure}

\begin{figure}%[htb!]
\centering
\begin{tabular}{cc}
\includegraphics[width=0.49\textwidth,height=3.8cm]{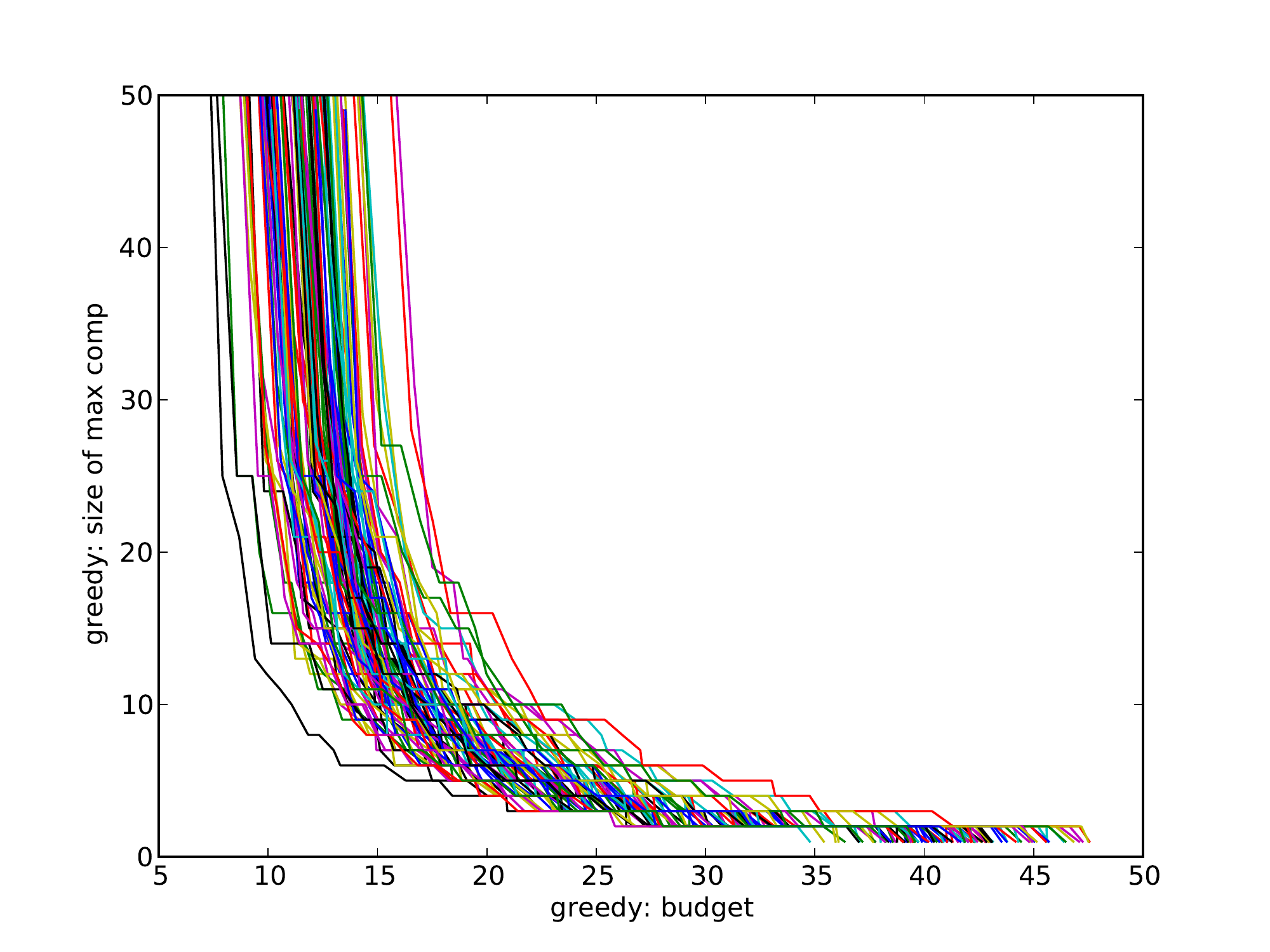} &
\includegraphics[width=0.49\textwidth,height=3.8cm]{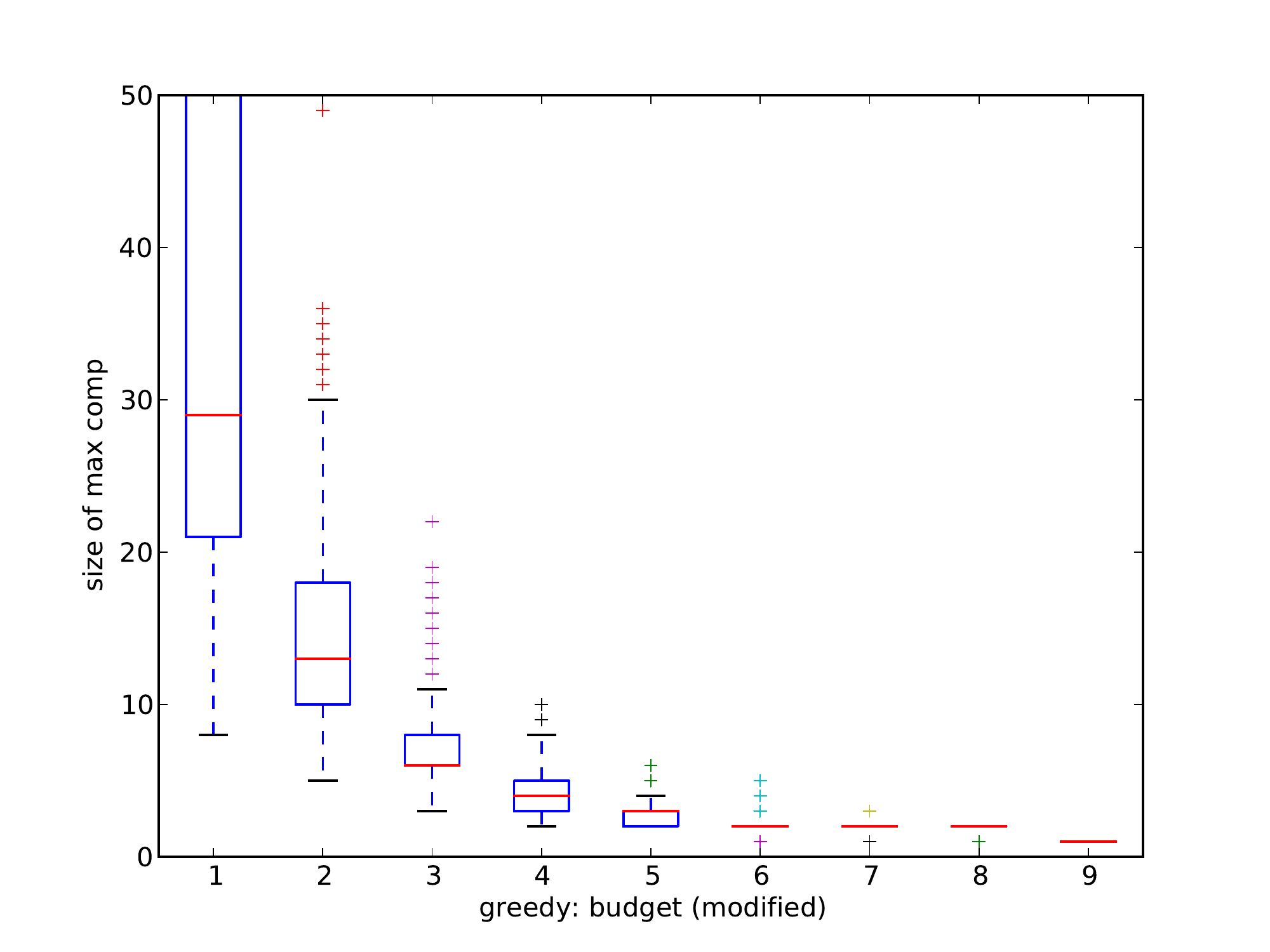} \\
(1) MaxSF-Susceptibility & (2) Box-plot view of (1) \\
\includegraphics[width=0.49\textwidth,height=3.8cm]{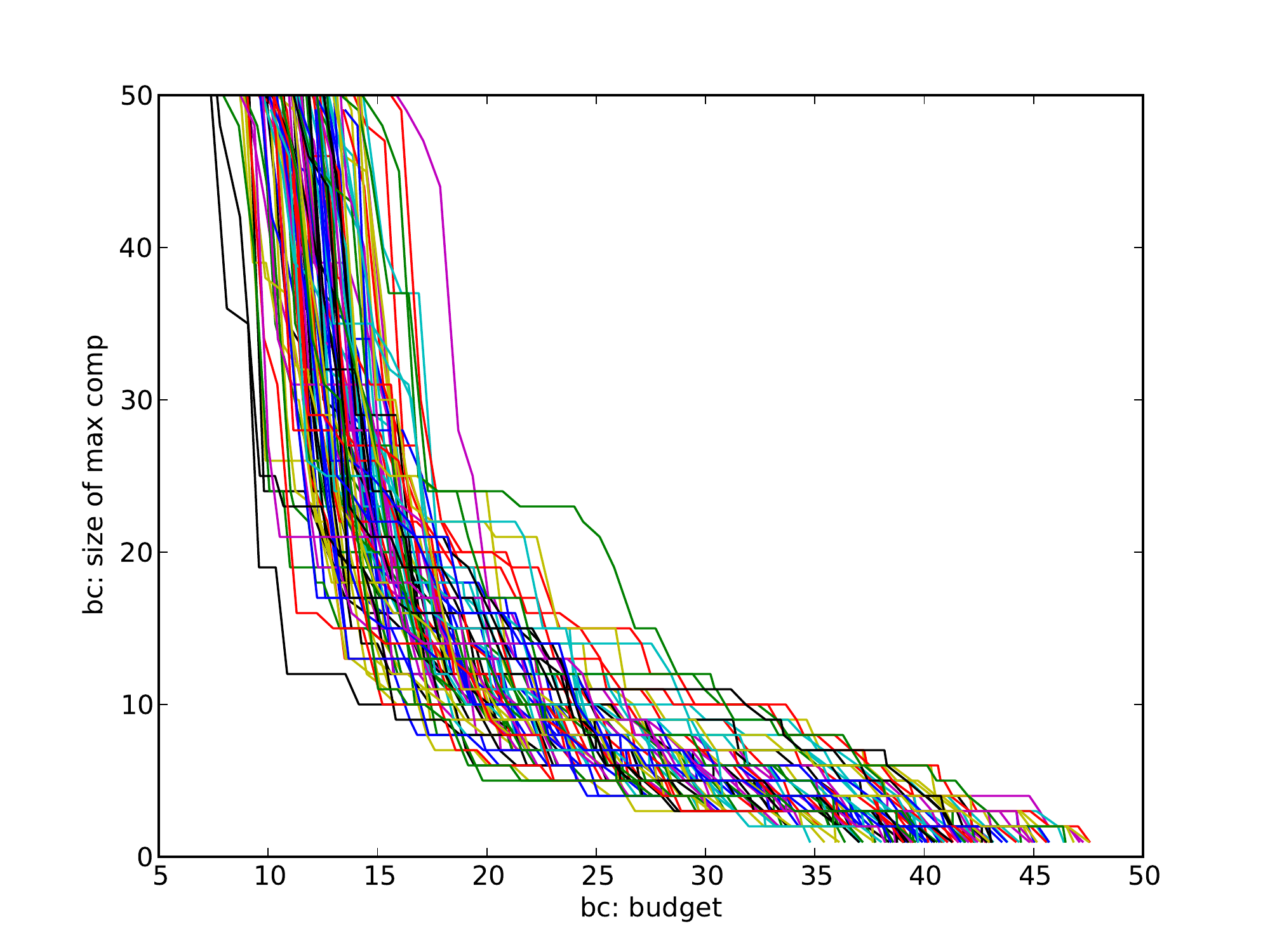} &
\includegraphics[width=0.49\textwidth,height=3.8cm]{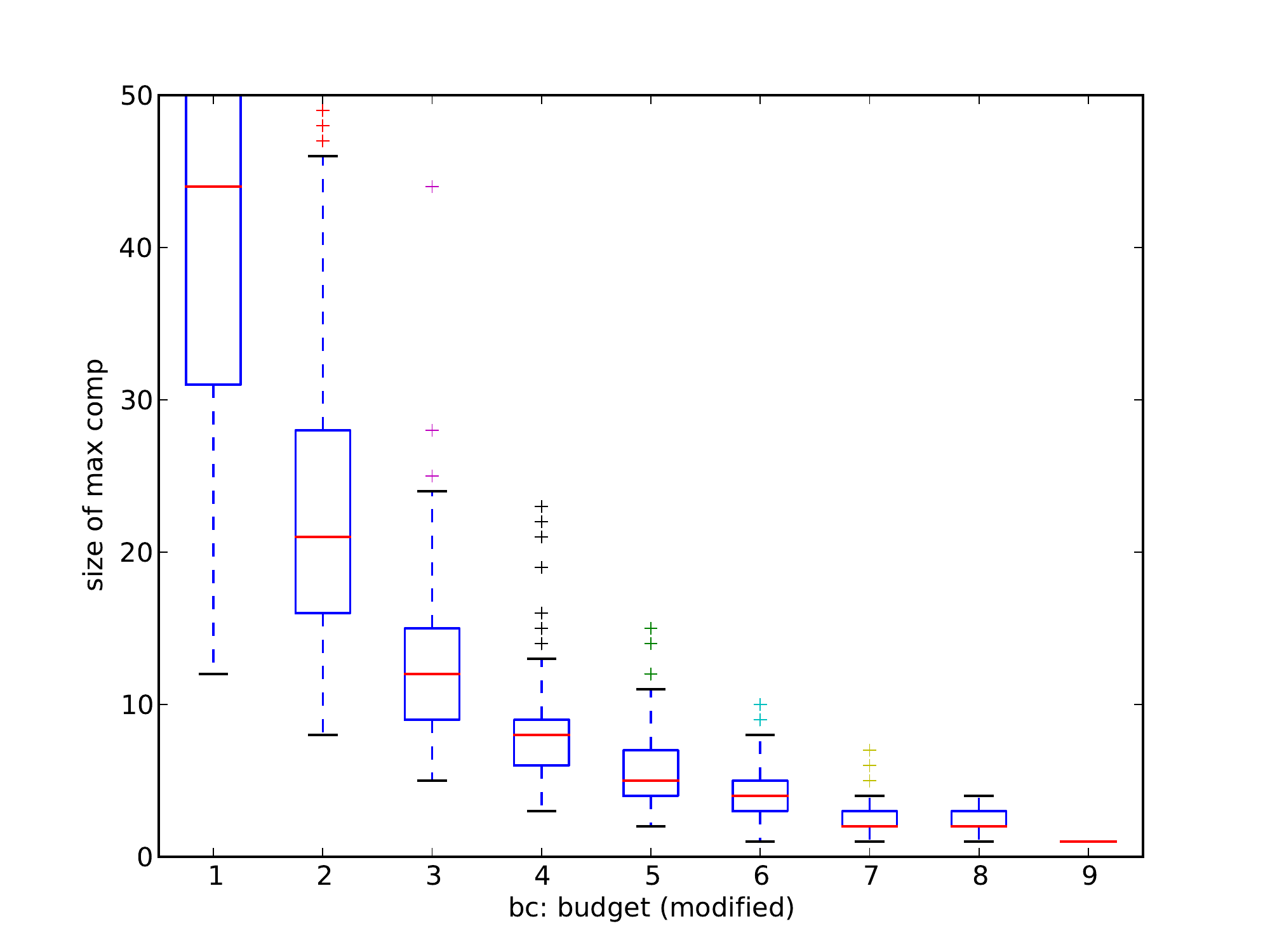} \\
(3) MaxSF-Betweenness & (4) Box-plot view of (3) \\
\includegraphics[width=0.49\textwidth,height=3.8cm]{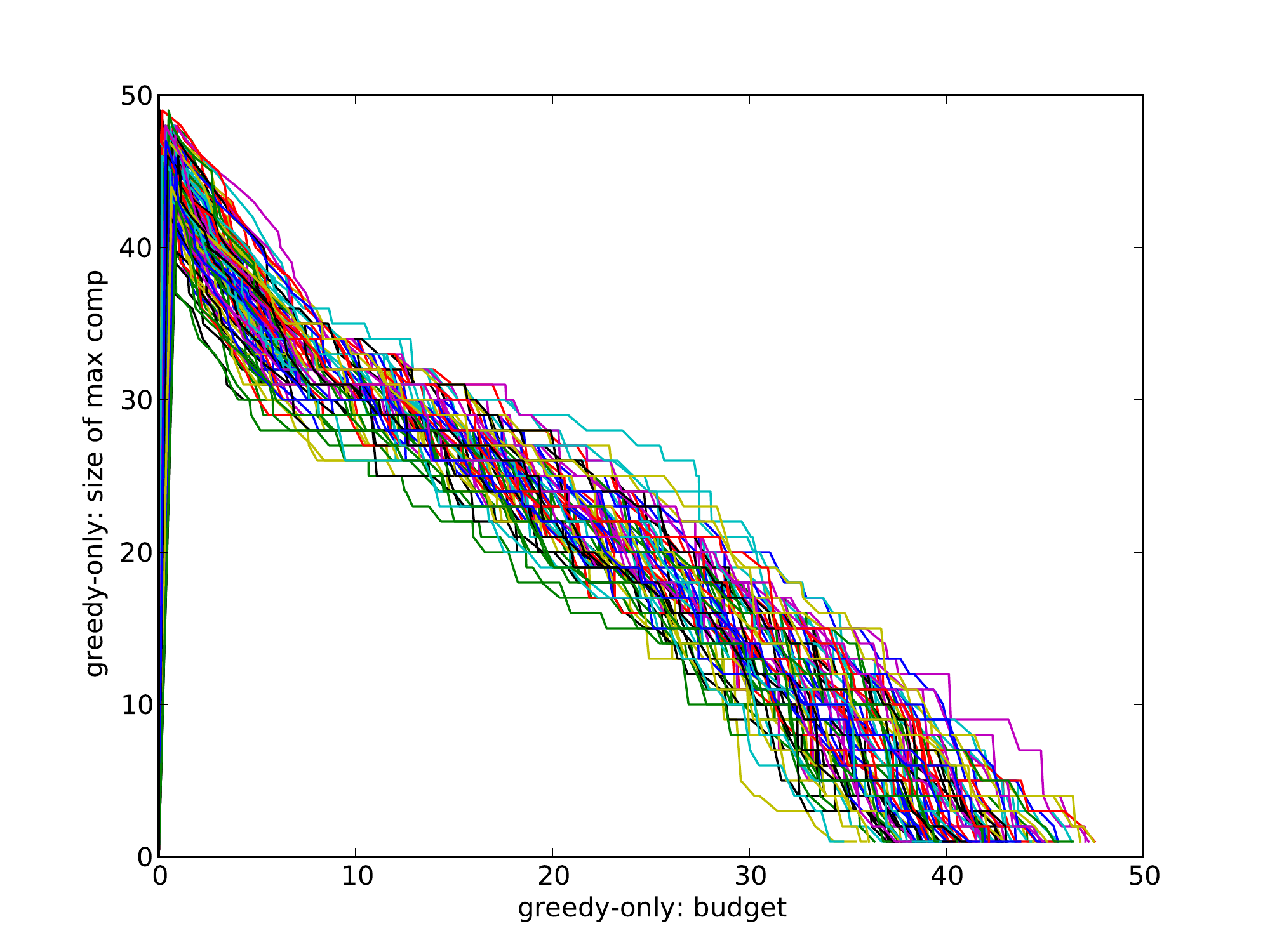} &
\includegraphics[width=0.49\textwidth,height=3.8cm]{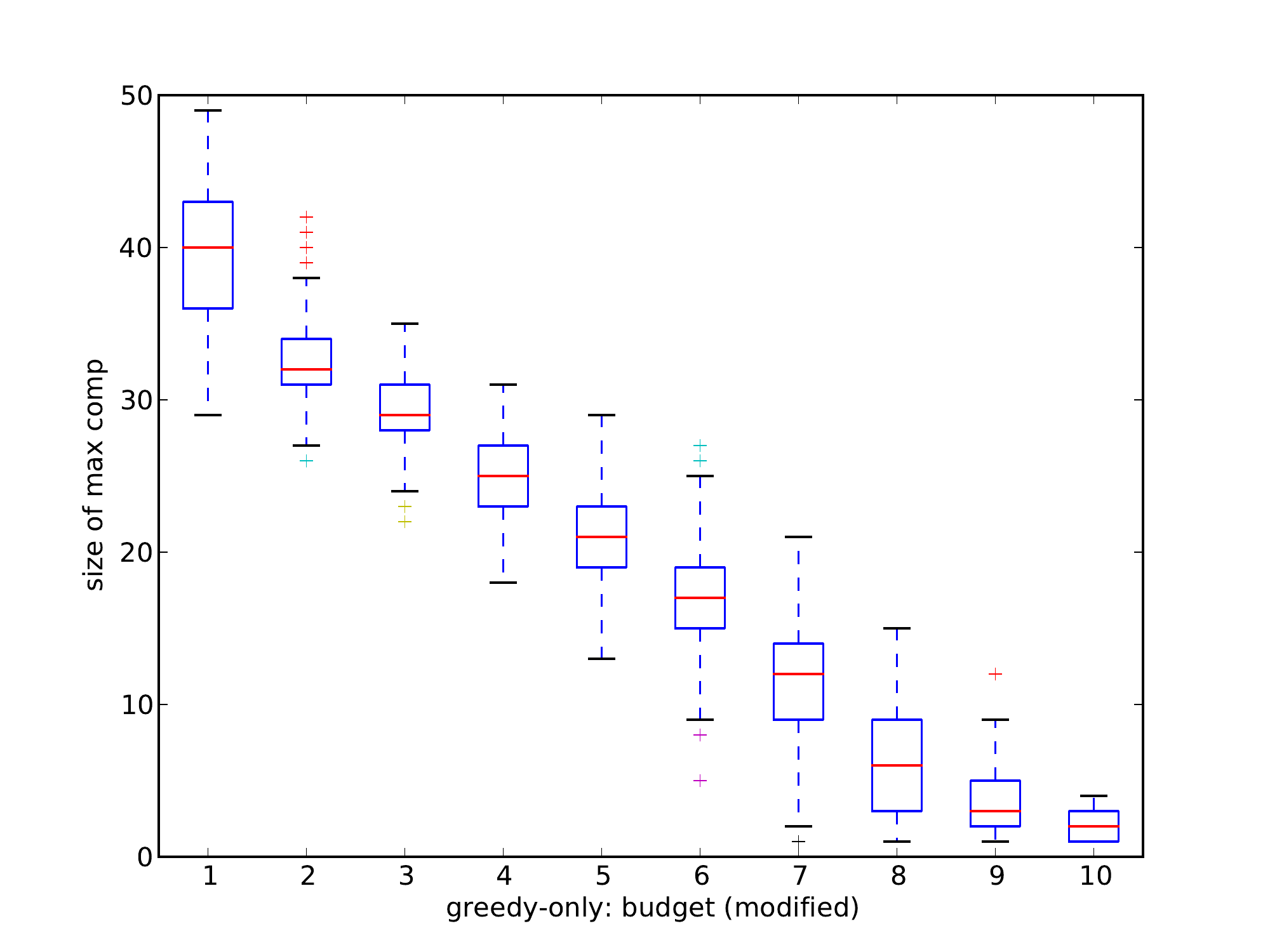} \\
(5) Full-Susceptibility & (6) Box-plot view of (5) \\
\includegraphics[width=0.49\textwidth,height=3.8cm]{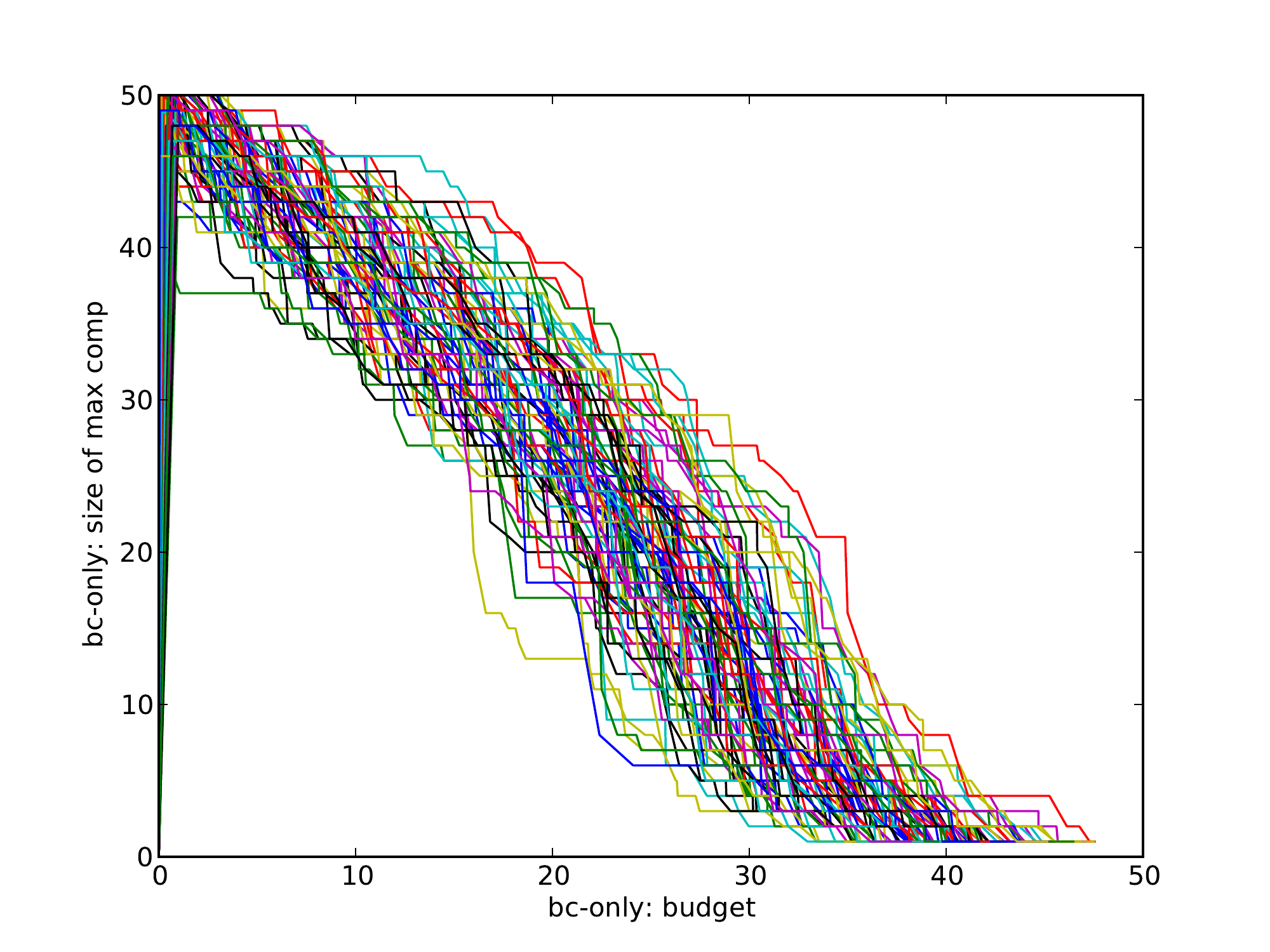} &
\includegraphics[width=0.49\textwidth,height=3.8cm]{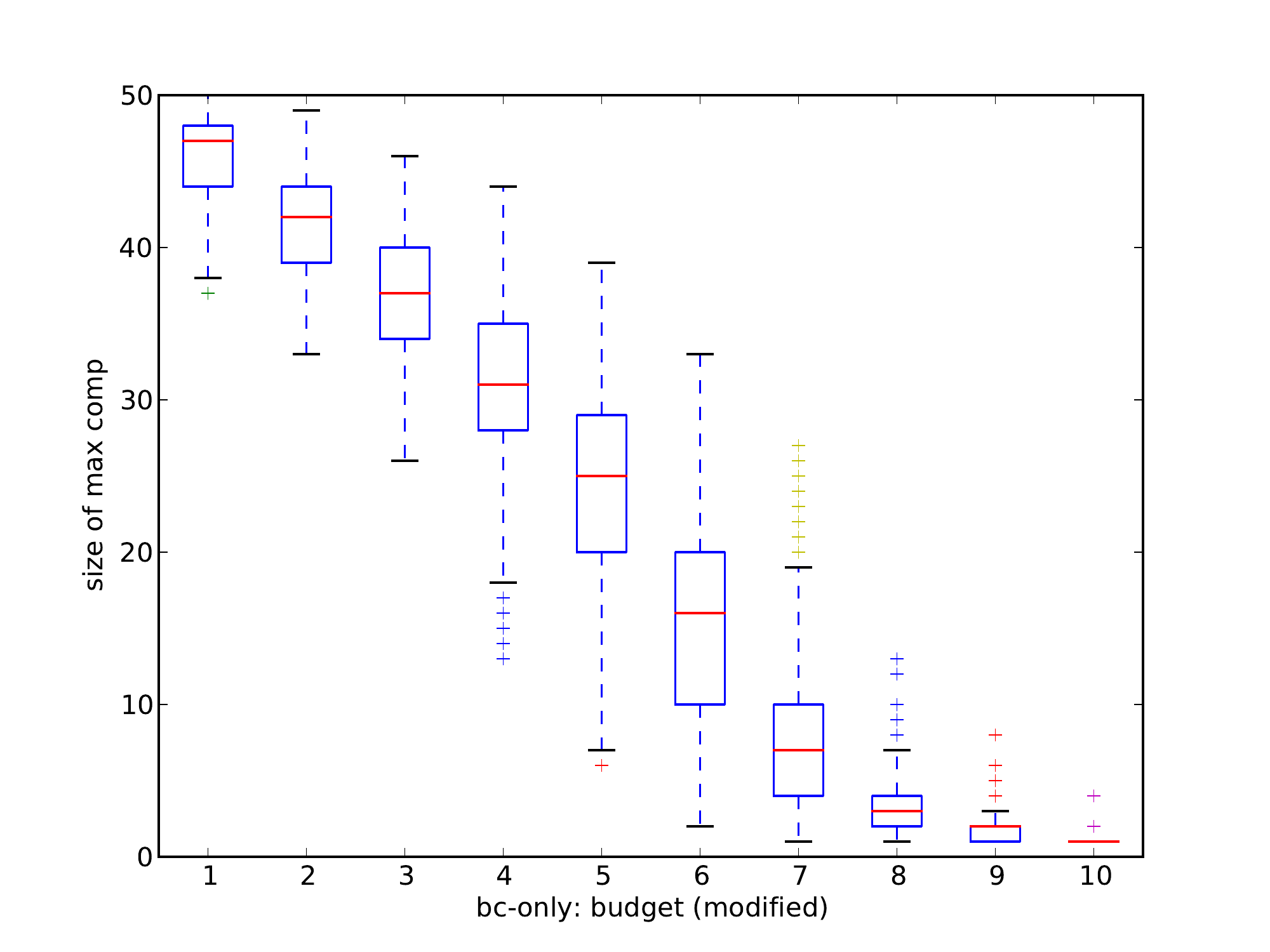} \\
(7) Full-Betweenness& (8) Box-plot view of (7)
\end{tabular}
\caption{Size of maximum component vs. edge removing cost for random graph model with power-law degree sequence, with uniform edge weights.}
\label{fig:pw-seq-unif}
\end{figure}

\begin{figure}%[htb!]
\centering
\begin{tabular}{cc}
\includegraphics[width=0.49\textwidth,height=3.8cm]{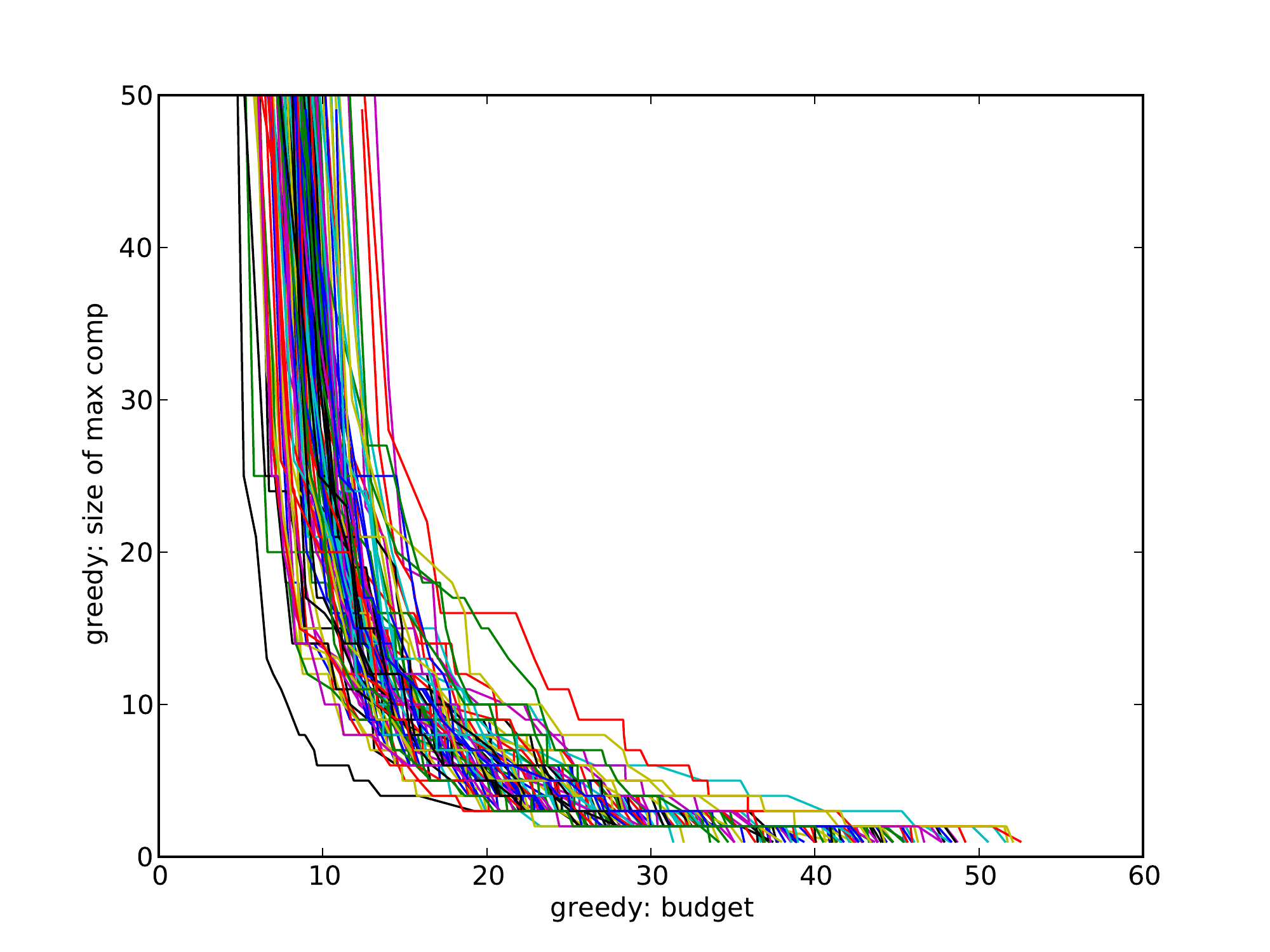} &
\includegraphics[width=0.49\textwidth,height=3.8cm]{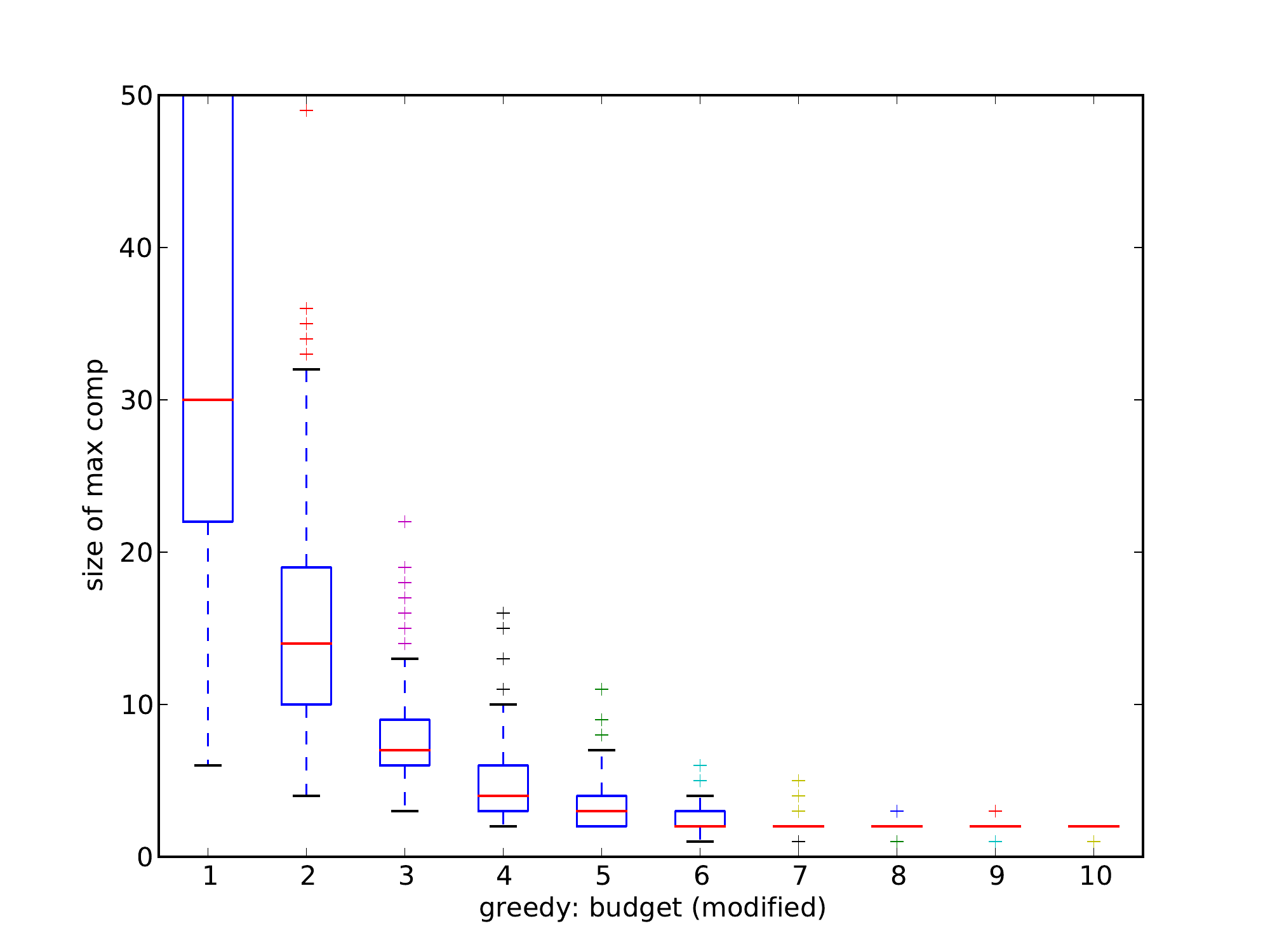} \\
(1) MaxSF-Susceptibility & (2) Box-plot view of (1) \\
\includegraphics[width=0.49\textwidth,height=3.8cm]{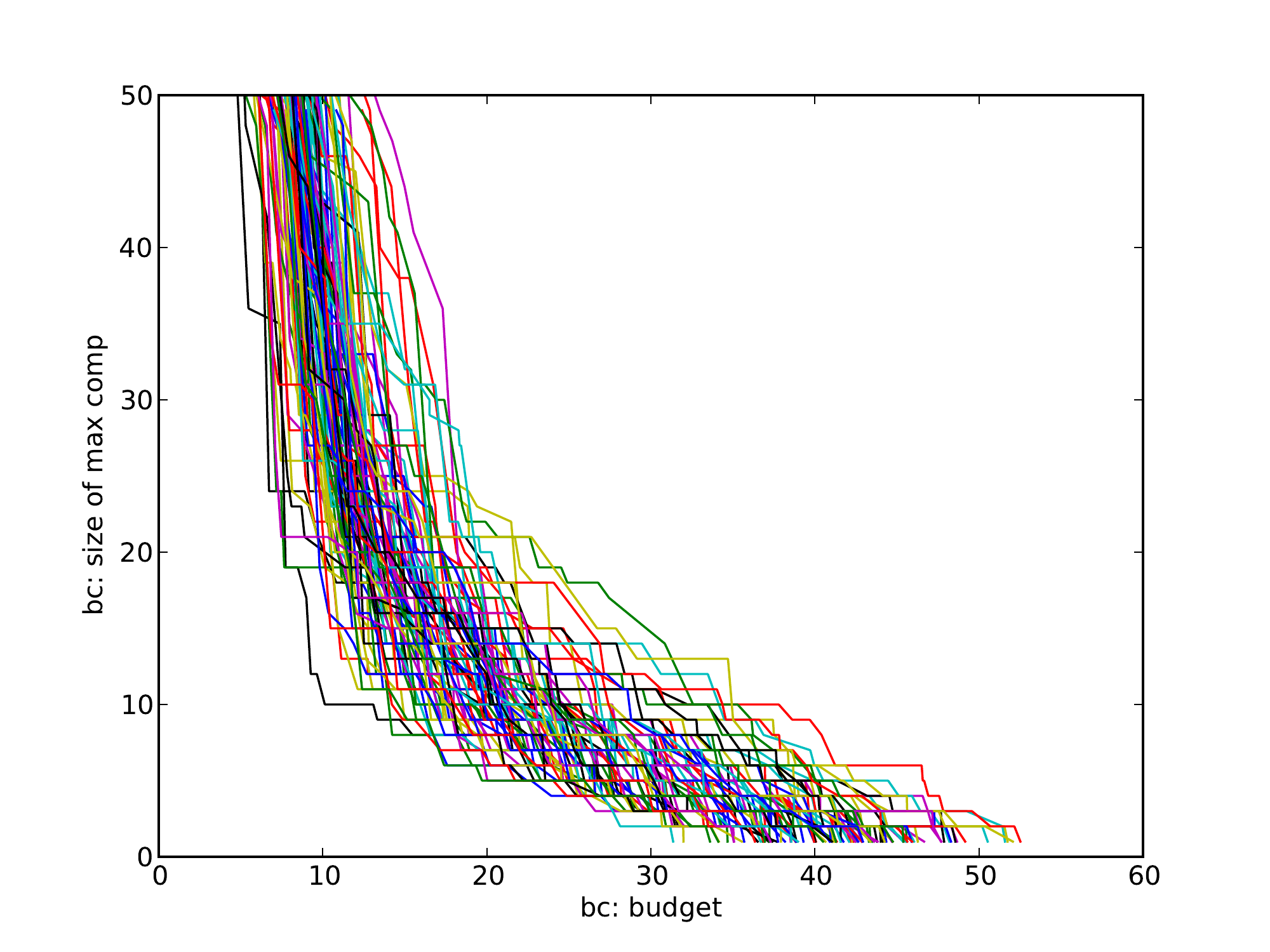} &
\includegraphics[width=0.49\textwidth,height=3.8cm]{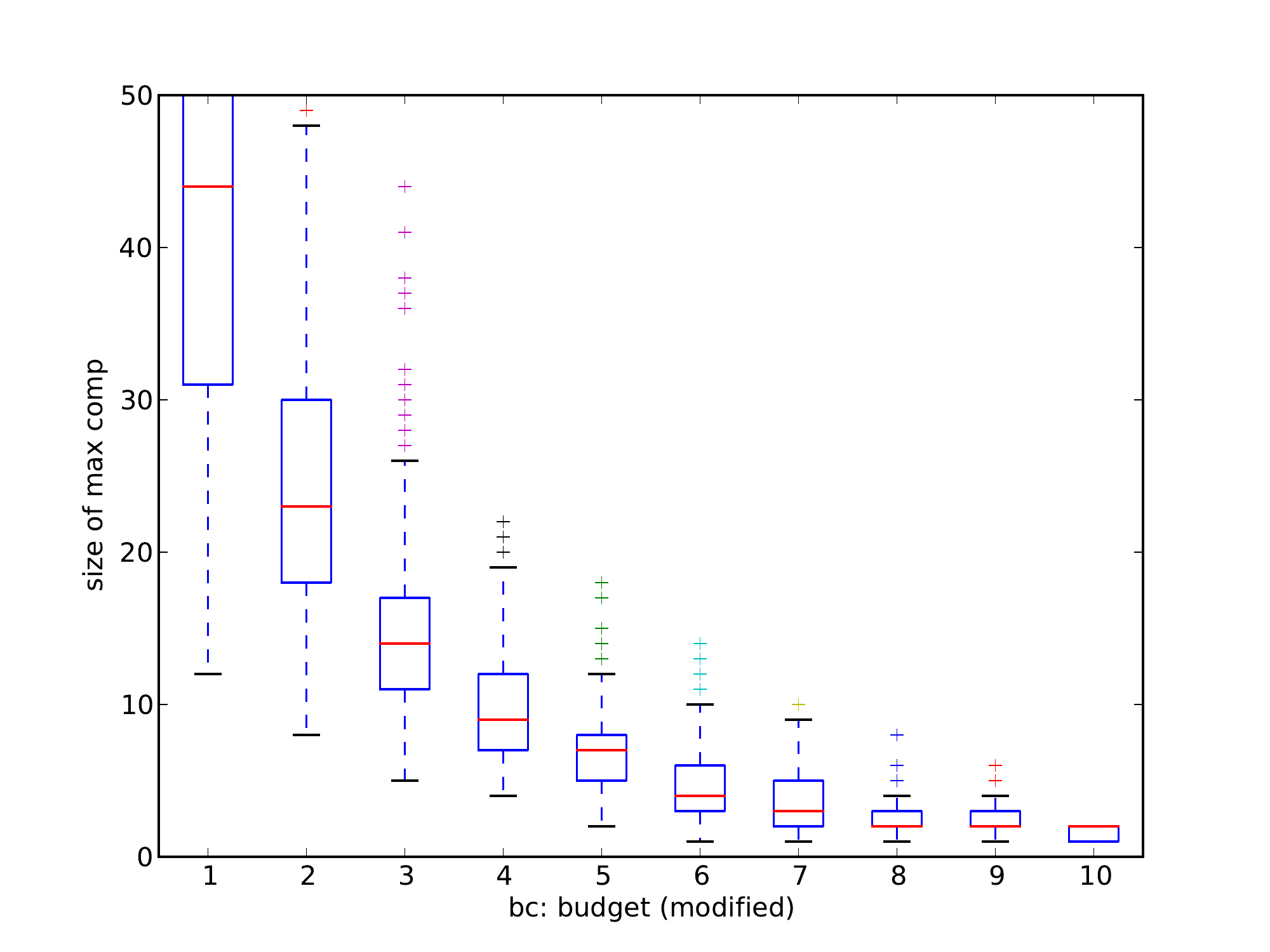} \\
(3) MaxSF-Betweenness & (4) Box-plot view of (3) \\
\includegraphics[width=0.49\textwidth,height=3.8cm]{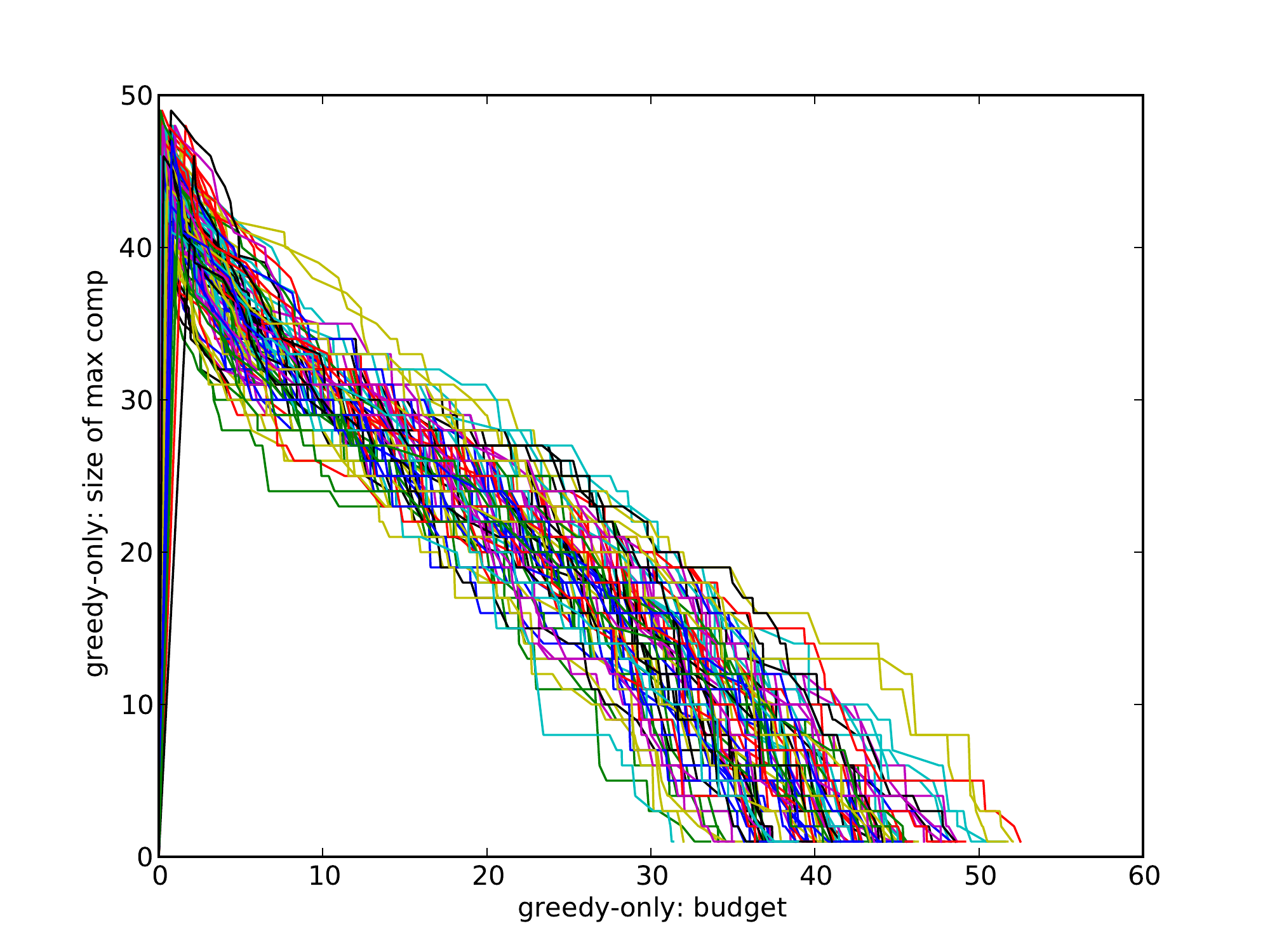} &
\includegraphics[width=0.49\textwidth,height=3.8cm]{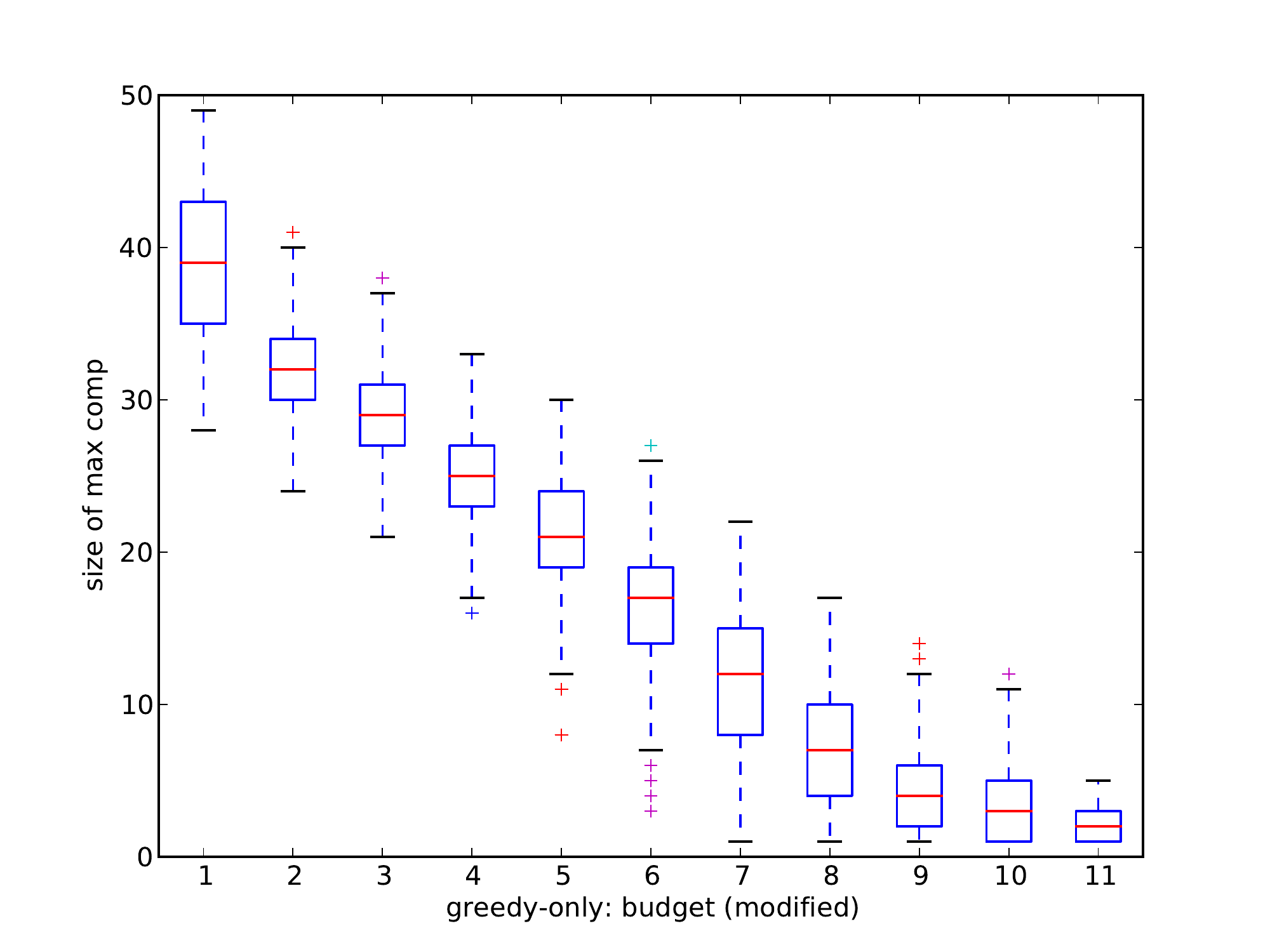} \\
(5) Full-Susceptibility & (6) Box-plot view of (5) \\
\includegraphics[width=0.49\textwidth,height=3.8cm]{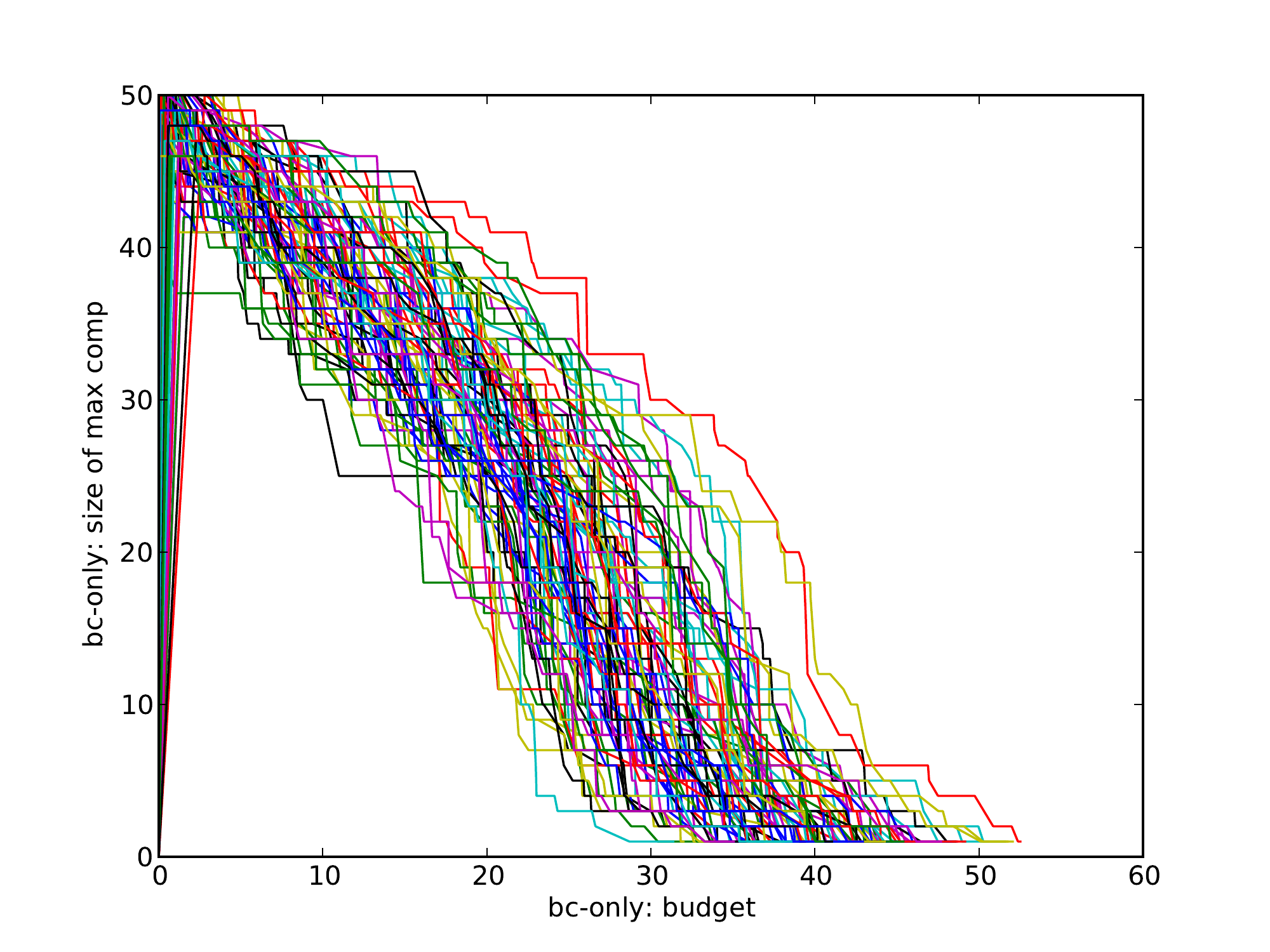} &
\includegraphics[width=0.49\textwidth,height=3.8cm]{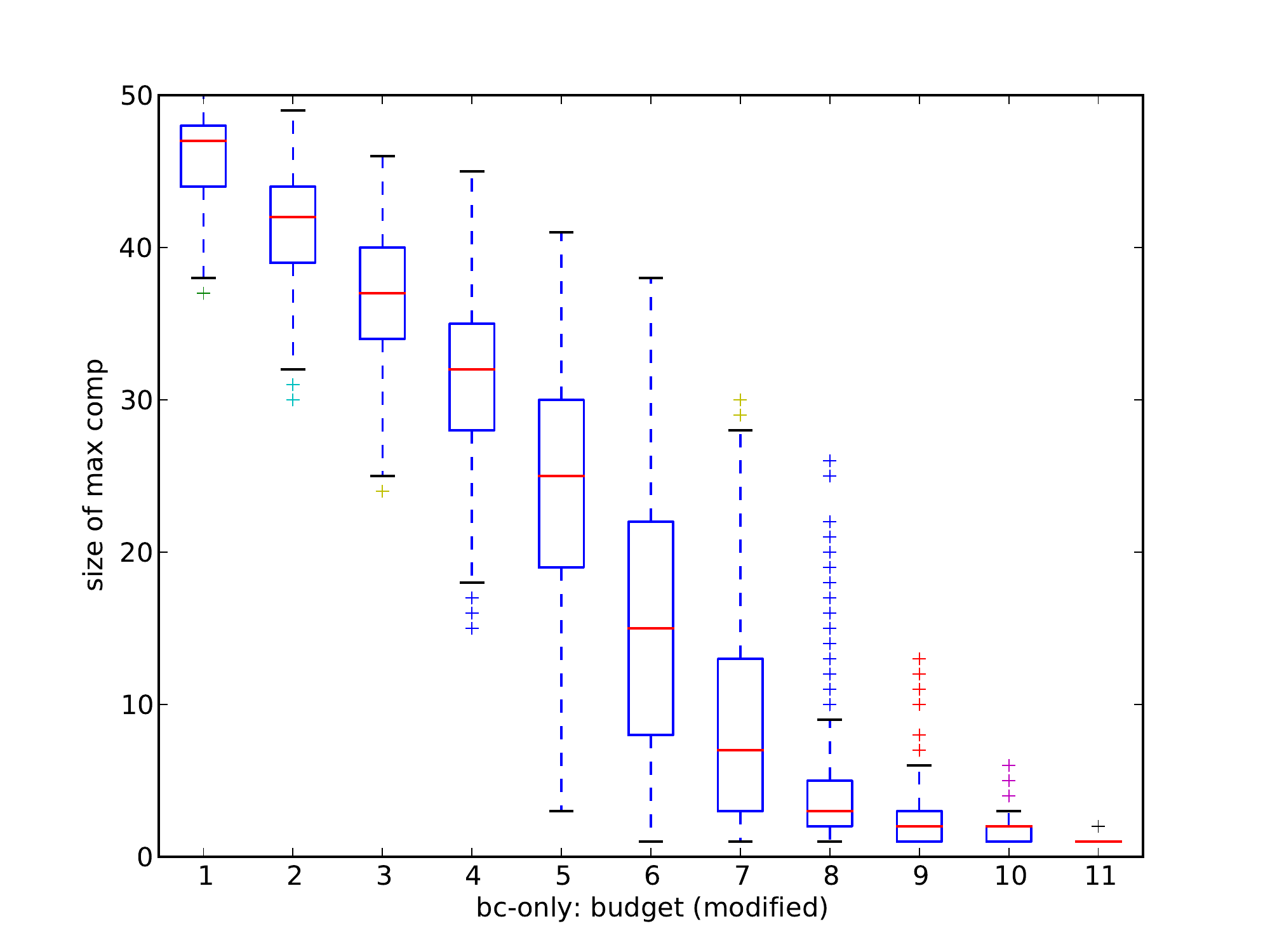} \\
(7) Full-Betweenness& (8) Box-plot view of (7)
\end{tabular}
\caption{Size of maximum component vs. edge removing cost for random graph model with power-law degree sequence, with exponential edge weights.}
\label{fig:pw-seq-exp}
\end{figure}

\begin{figure}%[htb!]
\centering
\begin{tabular}{cc}
\includegraphics[width=0.49\textwidth,height=3.8cm]{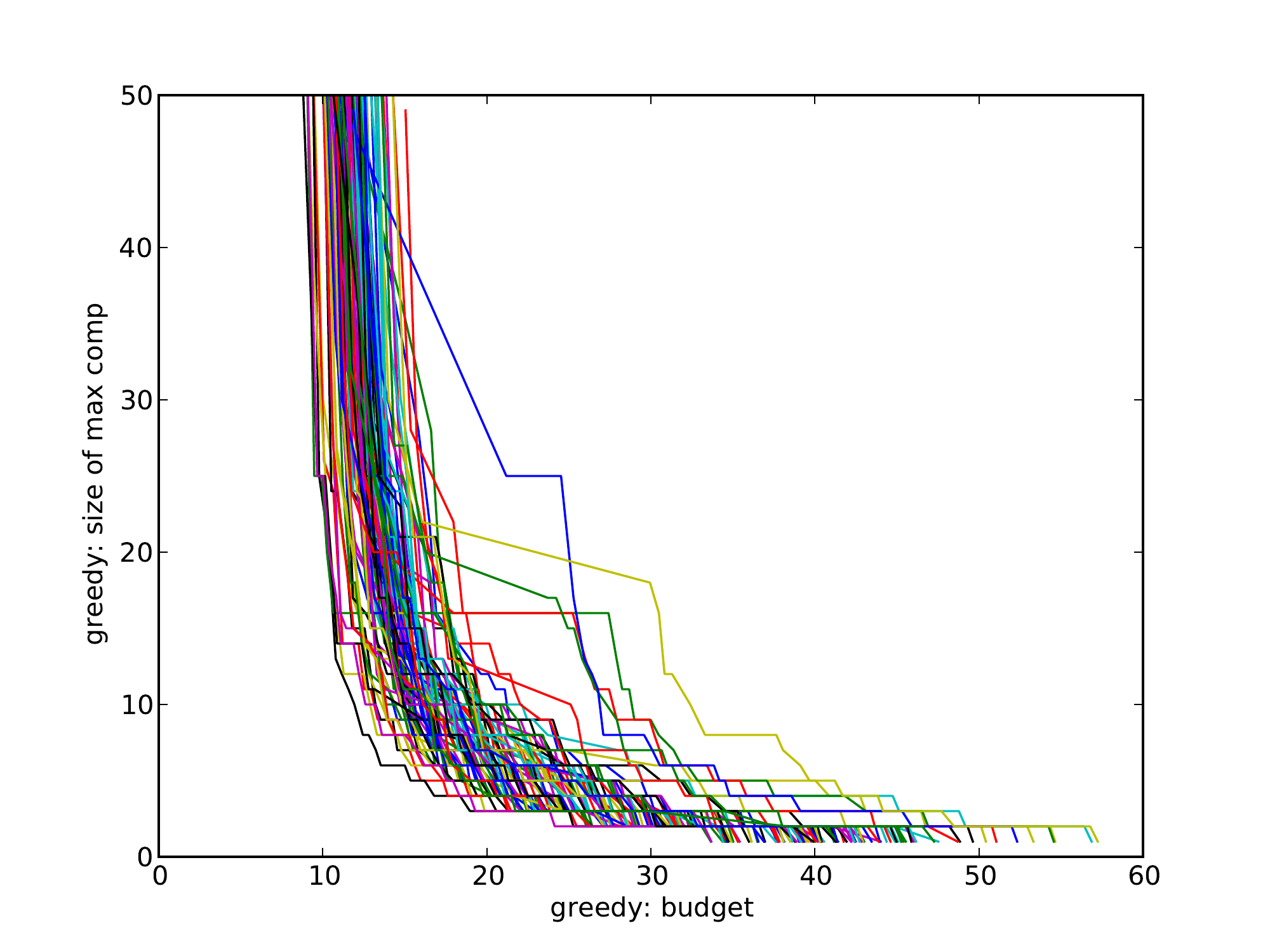} &
\includegraphics[width=0.49\textwidth,height=3.8cm]{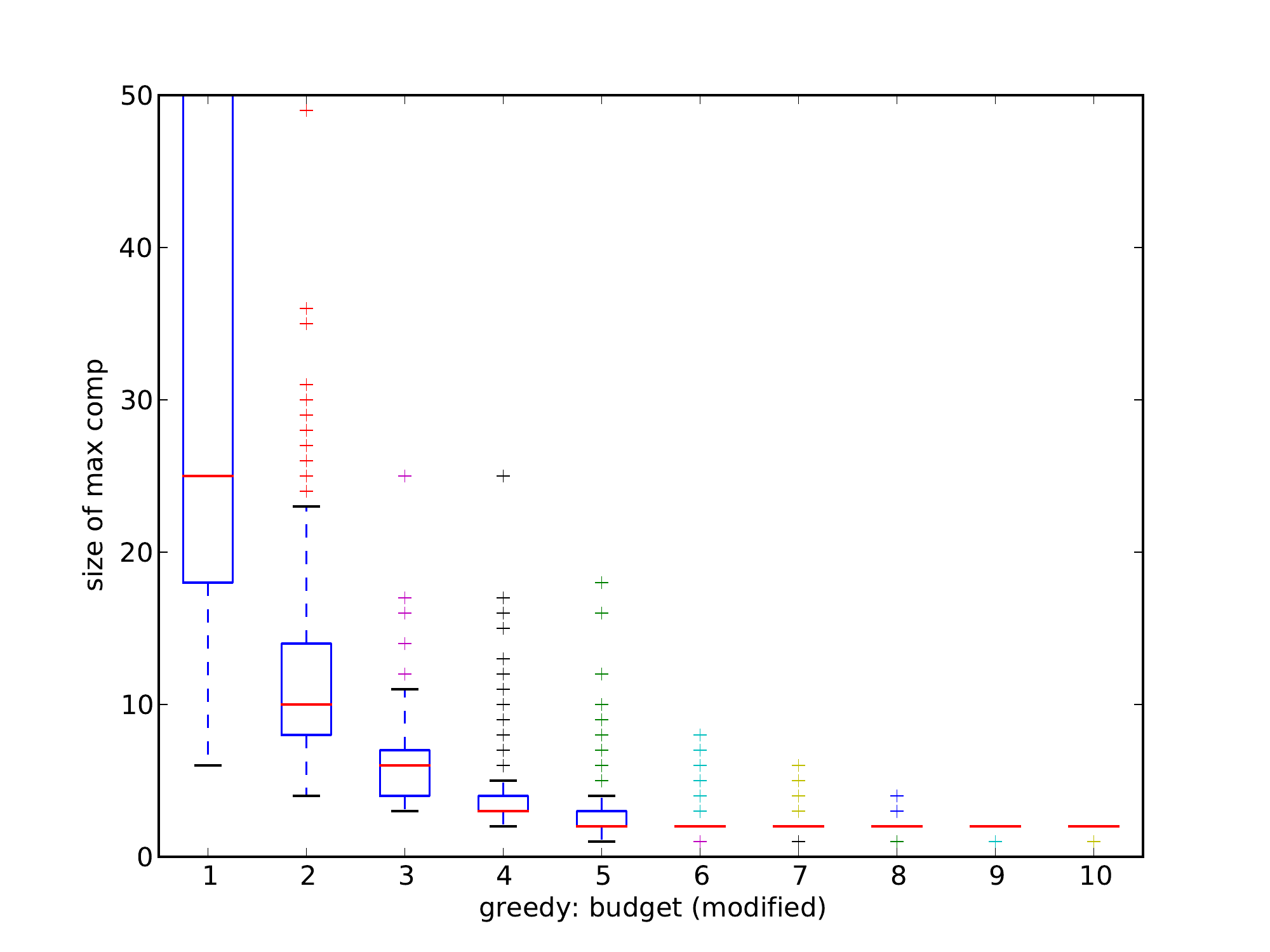} \\
(1) MaxSF-Susceptibility & (2) Box-plot view of (1) \\
\includegraphics[width=0.49\textwidth,height=3.8cm]{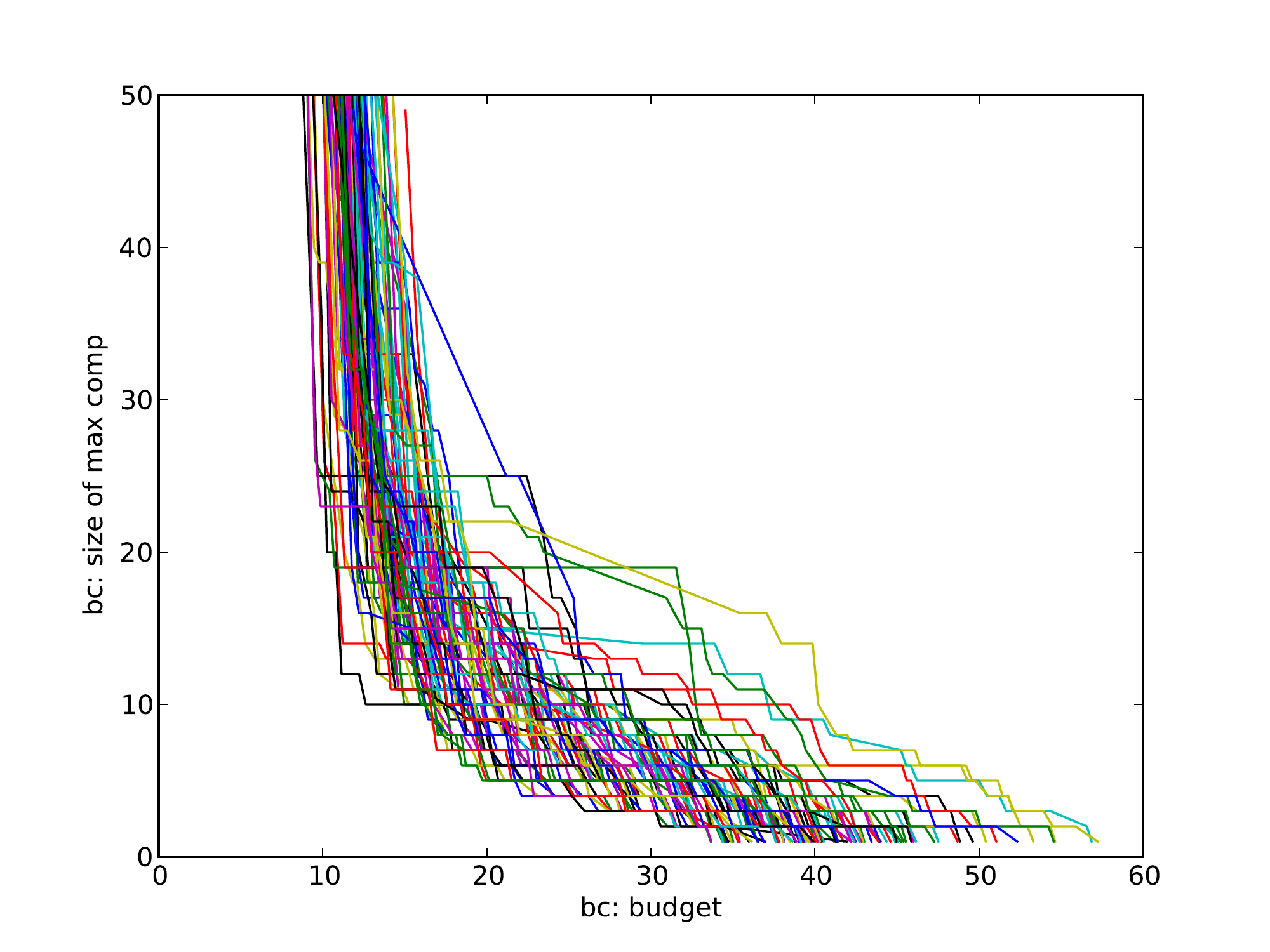} &
\includegraphics[width=0.49\textwidth,height=3.8cm]{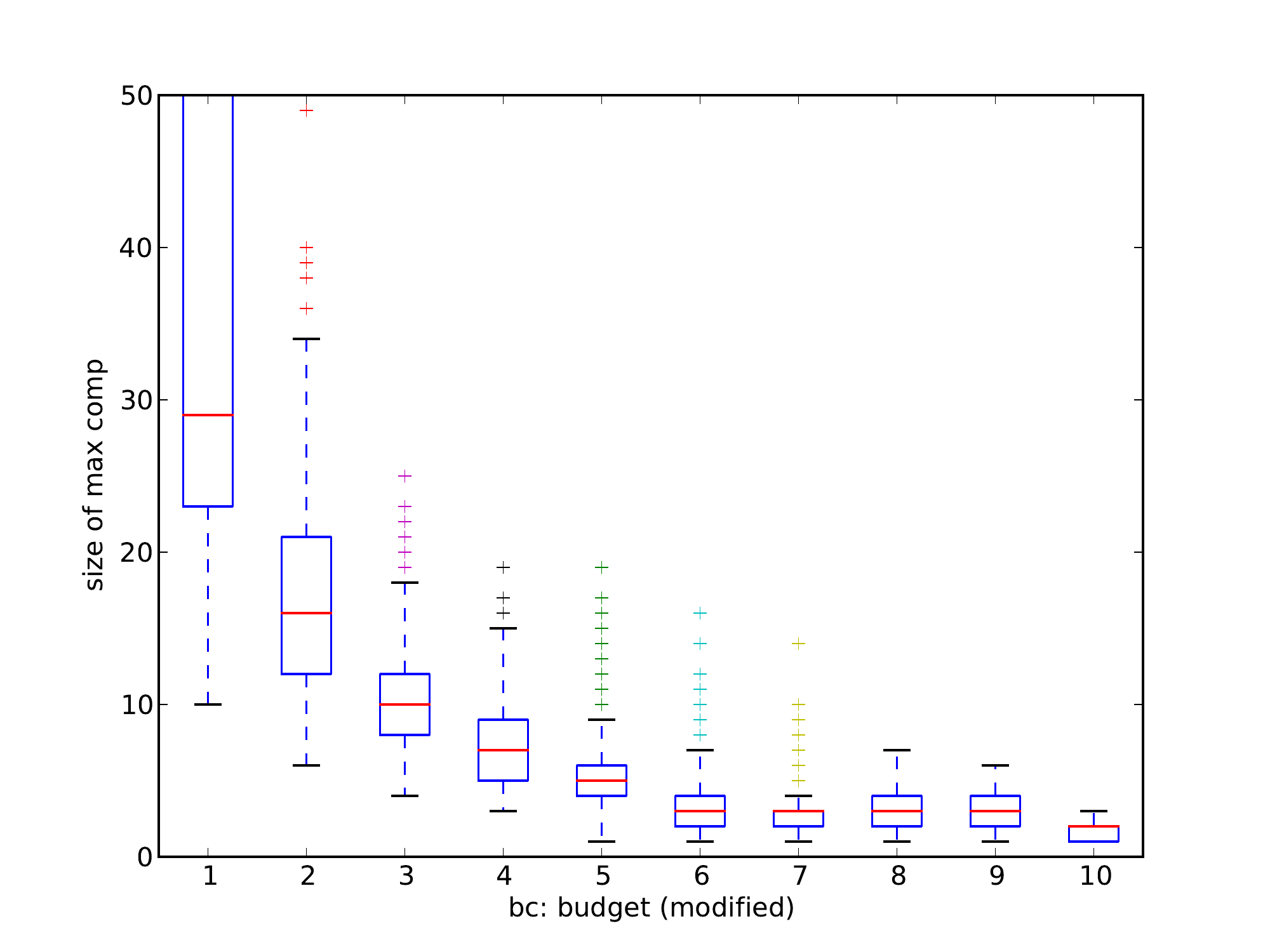} \\
(3) MaxSF-Betweenness & (4) Box-plot view of (3) \\
\includegraphics[width=0.49\textwidth,height=3.8cm]{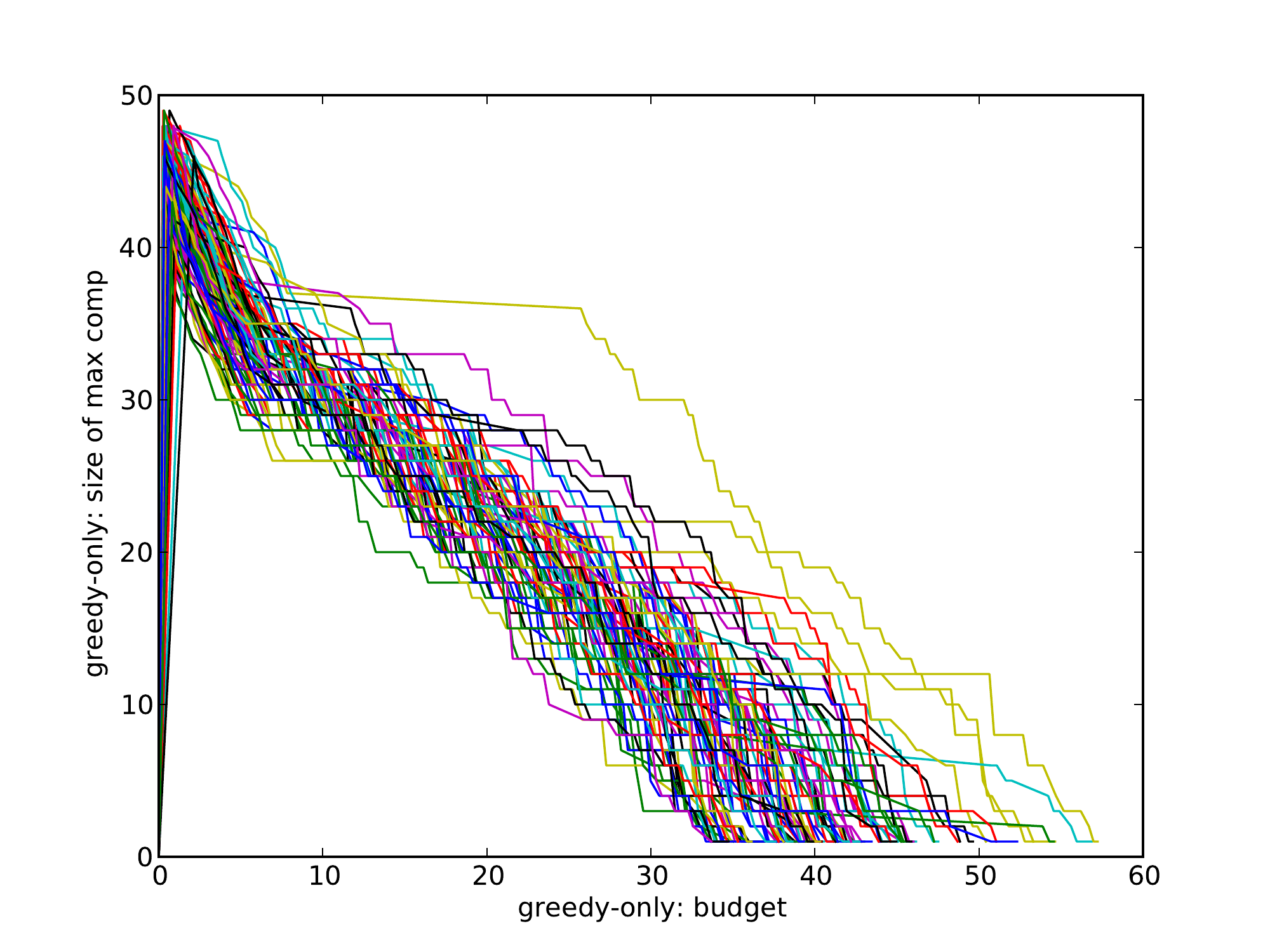} &
\includegraphics[width=0.49\textwidth,height=3.8cm]{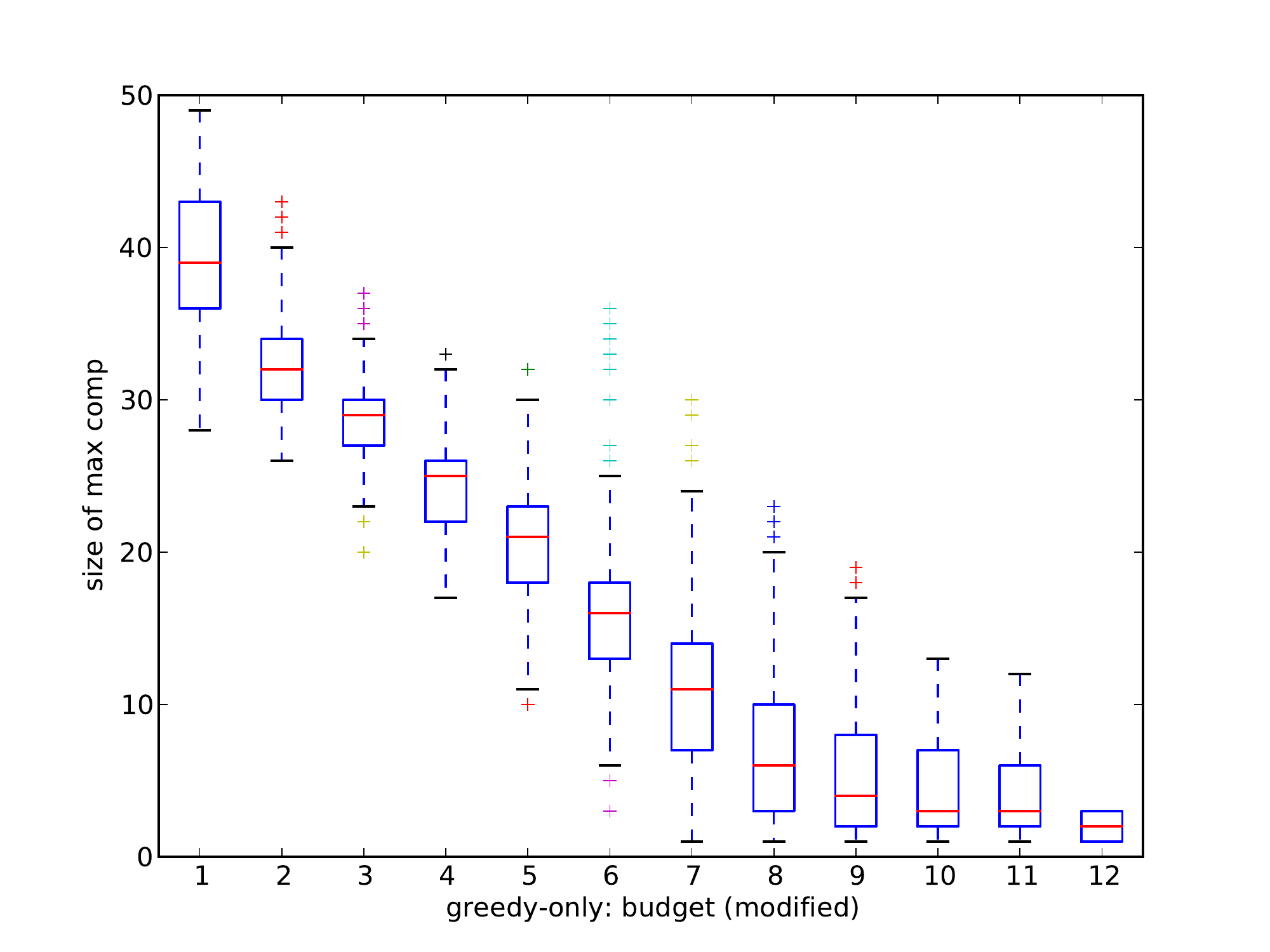} \\
(5) Full-Susceptibility & (6) Box-plot view of (5) \\
\includegraphics[width=0.49\textwidth,height=3.8cm]{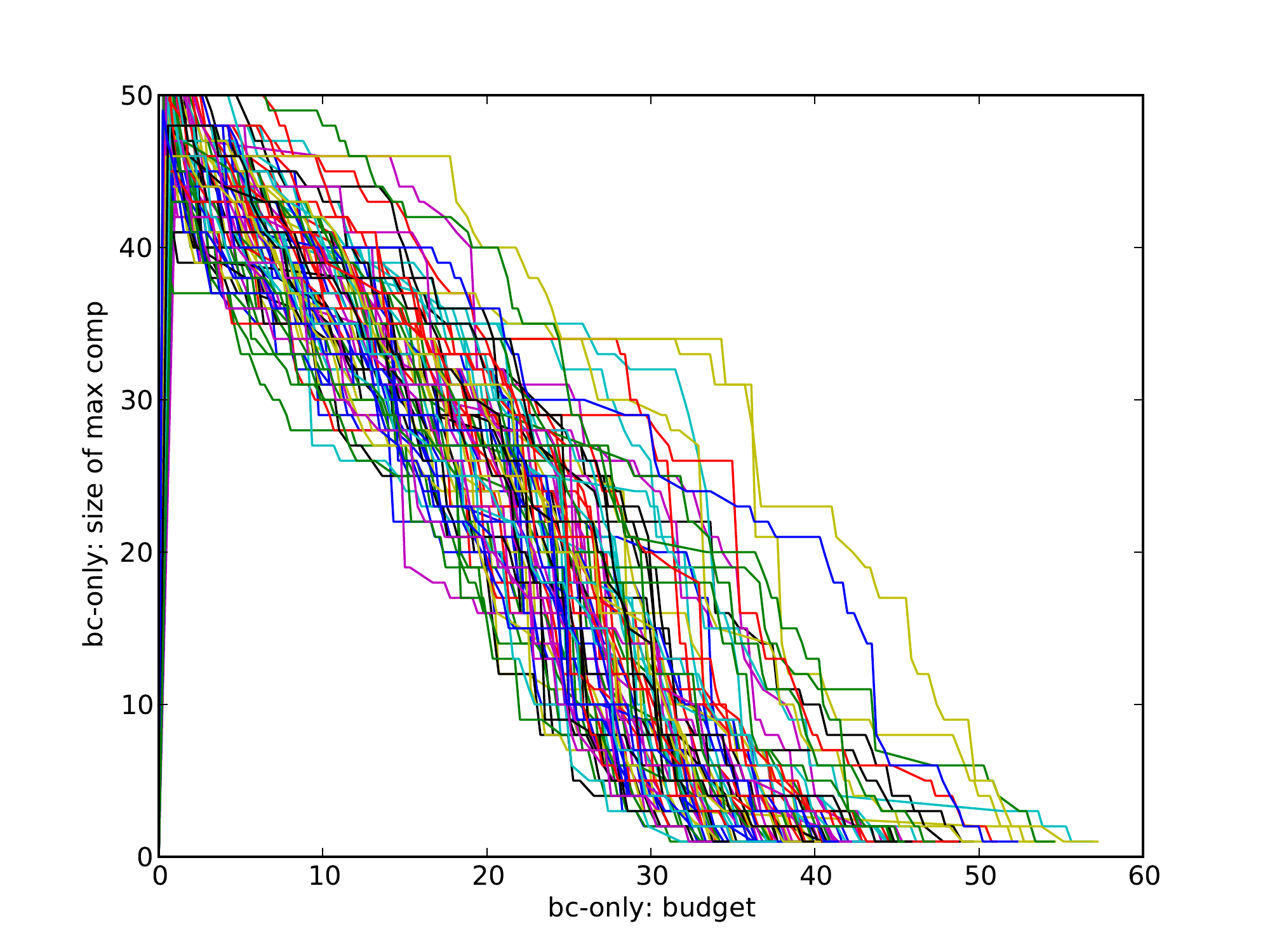} &
\includegraphics[width=0.49\textwidth,height=3.8cm]{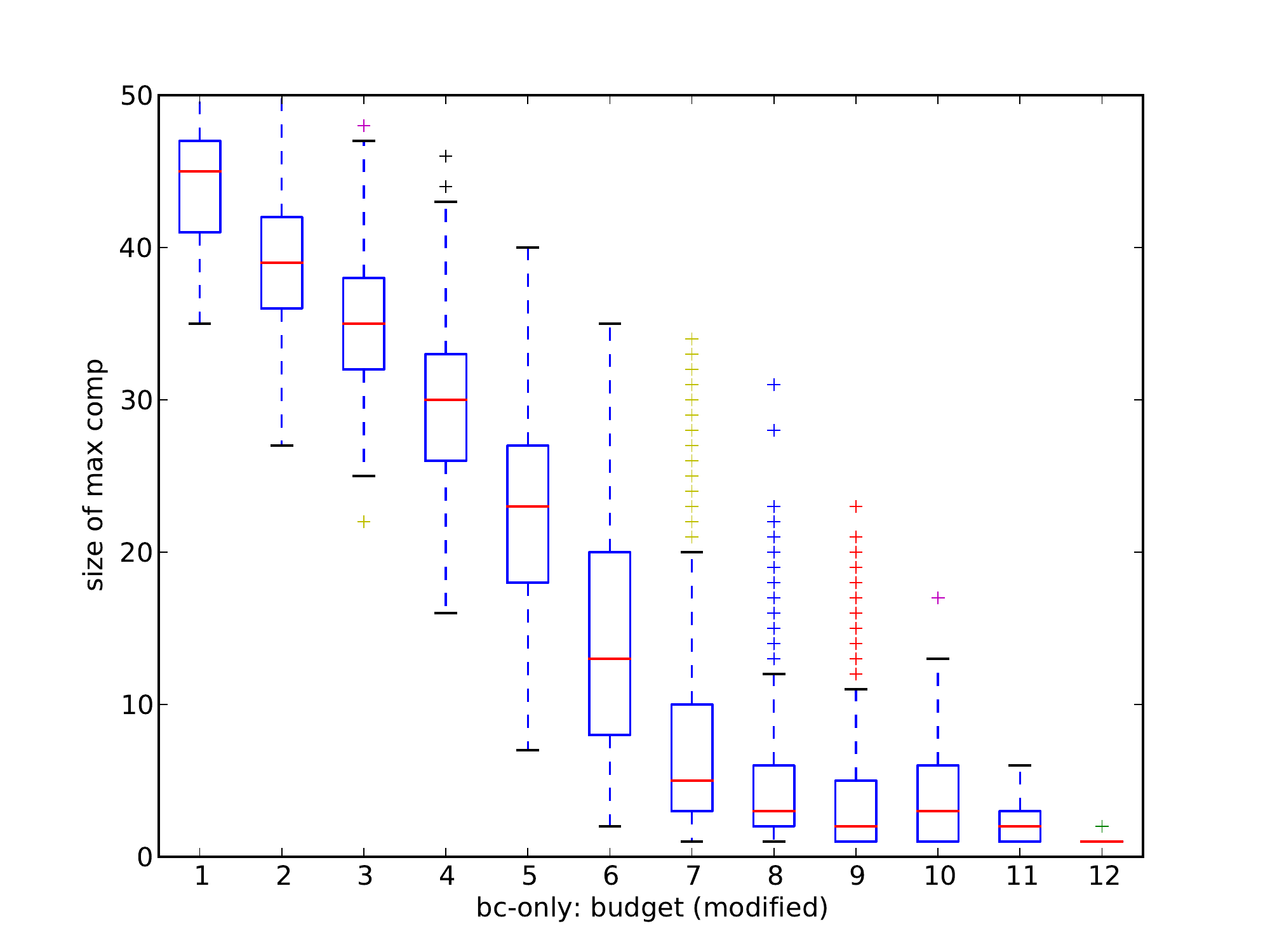} \\
(7) Full-Betweenness& (8) Box-plot view of (7)
\end{tabular}
\caption{Size of maximum component vs. edge removing cost for random graph model with power-law degree sequence, with power-law edge weights.}
\label{fig:pw-seq-powerlaw}
\end{figure}

\end{document}